\documentclass{article}
\usepackage[a4paper,top=3cm,bottom=3cm,left=2.1cm,right=2.1cm,marginparwidth=1.75cm]{geometry}

\usepackage{bm}
\usepackage{amsmath, amsfonts, amssymb, amsthm}
\usepackage{mathrsfs, dsfont}
\usepackage[dvipsnames]{xcolor}
\usepackage{graphicx}
\usepackage{verbatim}
\usepackage{float}
\usepackage{comment}
\usepackage{multicol}
\usepackage{ulem}
\usepackage{subfig}
\usepackage{bigints}
\def\twoplotswidth{0.48\linewidth}

\usepackage[colorlinks=true, allcolors=blue]{hyperref}
\usepackage{xparse} 
\usepackage{hyperref}
\usepackage{tasks}
\settasks{style=itemize}
\usepackage{enumitem}

\usepackage[numbers]{natbib}

\newtheorem{theorem}{Theorem}[section]
\newtheorem{lemma}[theorem]{Lemma}
\newtheorem{definition}[theorem]{Definition}

\newtheorem{proposition}[theorem]{Proposition}

\newtheorem{corollary}[theorem]{Corollary}

\theoremstyle{remark}
\newtheorem{remark}{Remark}[section]
\newtheorem{example}[remark]{Example}

\numberwithin{equation}{section}

\newenvironment{sqremark}{\begin{remark}}{\hfill \tiny $\blacksquare$ \end{remark}}
\newenvironment{sqexample}{\begin{example}}{\hfill \tiny $\blacksquare$ \end{example}}

\usepackage{pifont}

\newcommand{\R}{\mathbb{R}}
\newcommand{\N}{\mathbb{N}}

\newcommand{\E}{\mathbb{E}}

\def\d{\mathrm{d}}
\def\P{\mathbb{P}} 
\newcommand{\F}{\mathcal{F}}

\newcommand{\A}{\mathcal{A}}
\newcommand{\I}{\mathcal{I}}
\newcommand{\Ah}{{\A_h}}
\newcommand{\Aexp}{{\A_{\exp}}}

\newcommand{\set}[1]{\left\{#1 \right\}}

\newcommand{\bra}[1]{\left(#1 \right)}
\newcommand{\sqbra}[1]{\left[#1 \right]}
\newcommand{\inverse}[1]{\coninv{\bra{#1}}}
\newcommand{\res}[1]{\inverse{\emptyword - #1}}

\newcommand{\abs}[1]{\left|#1 \right|}
\newcommand{\norm}[1]{\abs{\abs{#1}}}
\newcommand{\normA}[2][t]{\norm{#2}_{#1}^\A}
\newcommand{\normh}[2][t]{\norm{#2}_{#1}^\Ah}

\newcommand{\indic}[1]{\mathds{1}_{\left\{ #1 \right\}}}

\newcommand{\alphabet}[1][d]{A_{#1}}
\newcommand{\TA}[1][d]{T(\R^{#1})}
\newcommand{\eTA}[1][d]{T((\R^{#1}))}
\newcommand{\tTA}[2][d]{T^{#2}(\R^{#1})}

\newcommand{\emptyword}{{\color{NavyBlue} \textup{\textbf{\o{}}}}}
\newcommand{\word}[1]{{\mathcolor{NavyBlue}{\mathbf{#1}}}}
\newcommand{\proj}[1]{|_{\word{#1}}}

\newcommand{\bp}{\bm{p}}
\newcommand{\bq}{\bm{q}}

\newcommand{\wv}{\word{v}}
\newcommand{\wu}{\word{u}}
\newcommand{\ww}{\word{w}}
\newcommand{\wi}{\word{i}}
\newcommand{\wj}{\word{j}}

\newcommand{\bgamma}{\bm{\gamma}}

\newcommand{\bell}{\bm{\ell}}
\newcommand{\bphi}{\bm{\varphi}}

\newcommand{\lgv}[1][t]{\bell_{#1}^\textnormal{GV}}

\newcommand{\lde}{\bell^\textnormal{DE}}
\newcommand{\pde}{\bp^\textnormal{DE}}
\newcommand{\qde}{\bq^\textnormal{DE}}

\newcommand{\lvol}{\bell^\textnormal{VOL}}
\newcommand{\pvol}{\bp^\textnormal{VOL}}
\newcommand{\qvol}{\bq^\textnormal{VOL}}

\newcommand{\shuprod}{\mathrel{\sqcup \mkern -3mu \sqcup}}
\newcommand{\shupow}[1]{^{\shuprod #1}}
\newcommand{\shuexp}[1]{e \shupow{#1}}
\newcommand{\conpow}[1]{^{\otimes #1}}
\newcommand{\conexp}[1]{e \conpow{#1}}
\newcommand{\coninv}[1]{{#1}^{-1}}
\newcommand{\dominated}{\preceq}
\newcommand{\fawcett}[1][t]{\bm{\mathcal{E}}_{#1}}

\NewDocumentCommand{\sigX}{O{t} O{X}}{\mathbb{#2}_{#1}}
\NewDocumentCommand{\sighat}{O{t} O{W}}{\widehat{\mathbb{#2}}_{#1}}
\NewDocumentCommand{\sigtilde}{O{t} O{W}}{\widetilde{\mathbb{#2}}_{#1}}
\NewDocumentCommand{\sig}{O{t} O{W}}{\sighat[#1][#2]}

\newcommand{\bracket}[2]{\left \langle #1, #2 \right \rangle}
\NewDocumentCommand{\bracketsigX}{O{t} O{X} m}{\bracket{#3}{\sigX[#1][#2]}}
\NewDocumentCommand{\bracketsig}{O{t} O{W} m}{\bracket{#3}{\sig[#1][#2]}}
\NewDocumentCommand{\bracketsigtrunc}{O{M} O{t} O{W} m}{\bracket{#4}{\sig[#2][#3]^{\leq #1}}}
\NewDocumentCommand{\bracketsigtilde}{O{t} O{W} m}{\bracket{#3}{\sigtilde[#1][#2]}}

\newcommand{\half}{\frac{1}{2}}
\newcommand{\thalf}{\tfrac{1}{2}}

\usepackage{mathtools}
\usepackage{shuffle}
\usepackage{bbm}
\usepackage{authblk}
\mathtoolsset{showonlyrefs}
\newcommand{\revone}{}
\newcommand{\revtwo}{}

\title{Path-dependent processes from signatures}

\author[1]{Eduardo Abi Jaber\thanks{eduardo.abi-jaber@polytechnique.edu. The first author is grateful for the financial support from the Chaires FiME-FDD, Financial Risks, Deep Finance \& Statistics and Machine Learning and systematic methods in finance at Ecole Polytechnique.}}
\author[2,3]{Louis-Amand Gérard\thanks{louis-amand.gerard@etu.univ-paris1.fr}}
\author[1]{Yuxing Huang\thanks{yuxing.math@gmail.com. \\ The authors would like to thank Xuyang Lin for fruitful discussions and the two anonymous referees for their insightful suggestions.}}
\affil[1]{Ecole Polytechnique, CMAP}
\affil[2]{Université Paris 1 Panthéon-Sorbonne, CES}
\affil[3]{Gefip}

\begin{document}

\maketitle

\begin{abstract}
    We provide explicit series expansions to certain stochastic path-dependent integral equations in terms of the path signature of the time augmented driving Brownian motion. Our framework encompasses a large class of stochastic linear Volterra and delay equations and in particular the fractional Brownian motion with a Hurst index $H \in (0, 1)$. Our expressions allow to disentangle an infinite dimensional Markovian structure and open the door to straightforward and simple approximation schemes, that we illustrate numerically. {\revone A key application is the derivation of explicit series representations for both conditional and unconditional moments.}
\end{abstract}

\begin{description}
\item[Mathematics Subject Classification (2010):] 60L10, 60L70, 60H20, 60G22 
\item[Keywords:] Path-signatures, Volterra processes, Stochastic Delay equations, fractional Brownian motion.
\end{description}

\section{Introduction}

    We provide explicit solutions to certain stochastic path-dependent integral equations in the form 
    $$ X_t = X_0 + \int_0^t b(t, s, (X_u)_{{\revtwo 0 \leq} u \leq s}) \d s + \int_0^t \sigma(t, s, (X_u)_{{\revtwo 0 \leq} u \leq s}) \d W_s, $$
    for some linear coefficients $b, \sigma$, in terms of the path signature of the time-augmented driving Brownian motion $(t, W_t)_{t \geq 0}$. \\
    
    We consider two non-Markovian specifications:
    \begin{itemize}
        \item \textbf{The Volterra case}, where the coefficients $b, \sigma$ are of the form 
        $$ g(t, s, (X_u)_{{\revtwo 0 \leq} u \leq s}) = K(t-s) (a + b X_s), $$
        for some locally integrable deterministic kernel $K$ and $a, b \in \R$. The Riemann-Liouville Fractional Brownian motion and Gaussian Volterra processes constitute a particular case. 
        
        \item \textbf{The Delay case}, where the coefficients $b, \sigma$ are of the form 
        $$ g(t, s, (X_u)_{{\revtwo 0 \leq} u \leq s}) = a + b X_s + \int_0^s K(s-u) X_u \d u.$$
    \end{itemize}
    
    In both cases, the unique strong solution $X$ is a non-anticipative measurable function of the whole path of the driving Brownian motion, that is 
    
    $$ X_t = f(t, (W_u)_{{\revtwo 0 \leq} u \leq t}), \quad t \geq 0, $$
    for some measurable functional $f$.  \\
    
    Our main contribution is to explicit this functional $f$. More precisely, we show that the solution $X$ can be written in terms of an infinite linear combination of the signature process $\sig$ of the time extended Brownian motion $\widehat{W}_t := (t, W_t)$ defined by the infinite sequence of iterated integrals in the Stratonovich sense:
    \begin{align*}
        \sig &
        = \left( 1,
        \begin{pmatrix}
            t \\ W_t
        \end{pmatrix},
        \begin{pmatrix}
            \frac{t^2}{2!} & \int_0^t s \d W_s \\
            \int_0^t W_s \d s & \frac{W_t^2}{2!}
        \end{pmatrix},
        \begin{pmatrix}
            \frac{t^3}{3!} & \int_0^t \int_0^s u \d  W_u \d s & \\
            \int_0^t \int_0^s W_u \d u \d s & \frac{1}{2!} \int_0^t W_s^2 \d s & \\
            & \frac{1}{2!} \int_0^t s^2 \d W_s & \int_0^t \int_0^s u \d W_u \circ \d W_s \\
            & \int_0^t \int_0^s W_u du \circ \d W_s & \frac{W_t^3}{3!}
        \end{pmatrix},
       \cdots \right),
    \end{align*}
    that is
    \begin{align}\label{eq:brackintro}
         X_t = \bracketsig{\bell_t}, \quad t \geq 0,
    \end{align}
    where the (possibly) time-dependent deterministic coefficients $\bell_t$, given by a sequence of {\revtwo numbers indexed by words}, are explicit. \\
    
    To highlight our strategy, consider a generalized geometric Brownian motion process $X$ solution to
    \begin{align} \label{eq:introOU}
        X_t = x + \int_0^t (a + bX_s) \d s + \int_0^t (\alpha + \beta X_s) \circ \d W_s, \quad \text{with } x, a, b, \alpha, \beta \in \R,
    \end{align}
    where $\circ$ denotes Stratonovich integration rule. \\

\textbf{Our recipe:}
    \begin{enumerate}
        \item Translate the probabilistic equation into an algebraic linear equation on \( \bell \):
        \begin{align} \label{eq:introlinear}
            \bell = x \emptyword + (a  +b \bell )\word{1} +  (\alpha + \beta \bell) \word{2},
        \end{align}
        where concatenation with the letters $\word{1}$ and $\word{2}$  play respectively the role of integration with respect to $\d t$ and $\d W_t$ (in the sense of Stratonovich).
        
        \item Solve the linear algebraic equation \eqref{eq:introlinear} by writing  
        \begin{align} \label{eq:introlinear2}
            \bell = \bp + \bell \bq,    
        \end{align}
        with 
        \begin{align} 
            \bp = (x \emptyword + a\word{1} + \alpha \word{2})\quad \text{and} \quad \bq = b \word{1} + \beta \word{2},
        \end{align}
        to get
        \begin{align} 
            \bell =  \bp \res{\bq} = \bp \sum_{n \geq 0} \bq^{\otimes n}.
        \end{align}
        
        \item Argue that the infinite series $\langle \bell, \sig \rangle$ is well-defined and solves the stochastic differential equation \eqref{eq:introOU}. 
    \end{enumerate}

    The recipe is simple yet powerful, as it remains robust across non-Markovian settings and for multi-dimensional Brownian motions, with the same algebraic equation \eqref{eq:introlinear2} but with different $\bp$ and $\bq$. To make the analysis rigorous, we proceed as follows. First, we study the existence of solutions of the algebraic equation \eqref{eq:introlinear} taking values in the extended tensor algebra together with some estimates for the solutions, see Proposition~\ref{prop:resolvent-solution} and Theorem~\ref{thm:ell_in_Sh_if_q_dominated}. Second, since $\bell$ can have infinitely many non-zero elements, a delicate analysis is needed to make sense of the series expansions in \eqref{eq:brackintro}, this is the object of Section~\ref{subsec:Sh}. In particular, we derive a tractable criterion for the global convergence in the supremum norm which is new compared to the related literature on infinite series of signature elements, see for instance \citet*{arous1989flots, cuchiero2023signature}. 
    The third step, a central ingredient for the derivation of our representation formulas in Section~\ref{sec:representations} is an Itô's formula for possibly infinite linear combinations of signature elements, see Theorem~\ref{thm:sig_ito_formula}. {\revtwo Finally, in Section~\ref{subsec:Aexp} we derive a tractable criterion based on the domination by a shuffle exponential of $\bp$ and $\bq$ that validates the aforementioned steps, i.e.~that the linear functional is well-defined, converges against the signature and is an Itô process.} \\       
    
    We provide two types of representations: (i) time-independent representations \eqref{eq:brackintro} for non-Markovian stochastic Volterra and delay equations with smooth enough kernels $K$ where the coefficients $\bell$ do not depend on time, see Theorems~\ref{thm:volterra} and \ref{thm:delayed}. Such $\bell$ would in principal correspond to the coefficients of a {\revtwo Stochastic}--Taylor expansion applied in a path-dependent setting, except that here we compute $\bell$ algebraically and establish the convergence of the Taylor series. For more general kernels, we provide approximation results in Corollary~\ref{corr:convergence-K^n}; but also (ii) time-dependent representations \eqref{eq:brackintro} with $\bell_t$ depending on time which allowed us to recover more general Gaussian processes in Theorem~\ref{thm:gaussian} with singular kernels. For instance, for an Ornstein-Uhlenbeck process, a time-dependent representation involves 
    $$ \bell_t^\textnormal{OU} = \shuexp{-\kappa (t \emptyword - \word{1})} \word{2}, $$
    which is the algebraic translation of the probabilistic expression 
    $$ X_t = \int_0^t e^{-\kappa (t-s)} \d W_s, $$
    with $\shuexp{}$ the shuffle exponential \eqref{nota:shuexp}. \\
    
    The representation \eqref{eq:brackintro} has at least two main features. First, it allows to disentangle an infinite dimensional Markovian structure in terms of the process $\sig$. Second it opens the door to straightforward and simple approximation schemes by truncating $\bell_t$, numerical illustrations are provided in Section~\ref{subsec:numerics}. {\revone A key application is the derivation of explicit series representations for both conditional and unconditional moments of Volterra and delay processes,   which can be efficiently computed numerically without relying on Monte Carlo simulations, as shown in Section~\ref{sec:conditional_moments}. In particular, the linear Volterra process is a polynomial Volterra process in the terminology of \citet{polynomial_volterra}, where unconditional moments were derived in terms of deterministic integral equations, without any treatment of conditional moments or numerical methods. Our results provide an alternative explicit representation for both conditional and unconditional moments and, most importantly, offer the first numerical implementation for such process.
 }
    \\

    \textbf{Motivation.} The elegant framework of path signatures, as introduced by \citet{chen1957integration} and extensively used in rough path theory \cite{friz_victoir_2010, lyons2014rough}, not only captures the intricate dynamics of path-dependent processes but also facilitates their integration into various fields of applications. {\revone A notable attribute driving the recent prominence of signatures is their inherent {linearization property}, or equivalently the universal approximation property of linear functionals of the signature, and the fact that signatures characterize paths, see \cite{chevyrev2022signaturemomentscharacterizelaws, giusti2022topologicalapproachmappingspace}.} These properties have found many practical applications in Machine Learning (\citet*{chevyrev2016primer, fermanian2021embedding, arribas-bipolardisorder}), Mathematical Finance (\citet*{sig_vol, bayer2023optimal, bayer2023primal, arribas-optiexec, cuchiero2023spvix, dupire2022functional, arribas-nonparam})... \\

    Our motivation for studying exact representation formulae in the form \eqref{eq:brackintro} for path-dependent processes arises from the importance of non-Markovian stochastic modeling in various fields. Specifically, stochastic Volterra and delay equations capture intricate temporal dependencies in time series data. For such non-Markovian processes, the representation formulae \eqref{eq:brackintro} not only reveal the underlying Markovian structure in the whole signature process $\sig$ but also enable us to re-frame them within a rather universal framework of linear functionals of the signature, for which one can develop generic methods. As illustration, recent works by \citet*{sig_vol, cuchiero2023signature} have demonstrated interest in infinite linear combinations of signature elements, unveiling a surprising infinite-dimensional affine structure underlying a broad spectrum of stochastic models. Certain Fourier-Laplace transforms of linear functionals of the signature can be computed in terms of non-standard tensor valued infinite-dimensional Riccati equations  and applied in practice for risk management purposes in mathematical finance for instance.
    \\

    \textbf{Related literature on series and functional expansions.} {\revtwo Stochastic Taylor expansions} of solutions to Markovian stochastic differential equations (SDEs) in terms of the signature of the time augmented driving Brownian motion $(t, W_t)_{t \geq 0}$ have appeared in \citet[Chapter 5]{kloeden1992stochastic} and have been used to develop a cubature formula on the Wiener space by \citet{lyons-smg}. \citet{arous1989flots} establishes convergence criterion for such stochastic Taylor series for Markovian SDEs with analytic coefficients; the so-called Chen–Fliess \cite{fliess1983concept,fliess1986vers} approximation that generalizes the classic stochastic Taylor expansion for path dependent functionals of Markovian SDEs is studied by \citet{litterer2014chen}. A functional Taylor expansion in terms of a finite series of higher order functional derivatives multiplied by iterated integrals is presented by \citet{dupire2022functional}, a nice overview on series and functional expansions is given there. Variation of constants formulae for linear forward and backward stochastic Volterra integral equations using certain sequence of iterated Volterra integrals, adapting Wiener–Itô chaos expansions, are derived by \citet{hamaguchi2023variation}. {\revone Related  finite order expansions also appeared in \citet*{bayer2017regularitystructureroughvolatility}, \citet{Hu01102013} and \citet*{neuenkirch2006treesasymptoticdevelopmentsfractional}.}\\

     Our expansions can be seen as an infinite Stochastic-Taylor series expansion of the solution of path-dependent stochastic equations whose coefficients depend linearly on the whole trajectory. In particular, we believe that the coefficients of our time-independent expansions, obtained in `one shot' by solving an algebraic equation, can be recovered from successive functional derivatives of the coefficients of the stochastic equation. None of the works cited above rely on such algebraic method. When it comes to studying the convergence of the stochastic Taylor series, our work is closest in spirit to \citet{arous1989flots} which operates in a Markovian setting. Just like Taylor expansions \cite{kloeden1992stochastic,dupire2022functional}, and unlike Volterra and Wiener expansions, the coefficients $\bell$ that depend on the model parameters and the signature $\sig$ of the driving noise are disentangled. Combined with the explicit knowledge of the coefficients, this allows for a straightforward and simple numerical implementation of non-Markovian processes. \\

{\revone Moreover, finite expansions of Volterra and delay processes have also appeared in the rough path literature, using tailor-made constructions of signature objects. 
We mention, for instance, {Volterra-type signatures}, where the iterated integrals involve weights coming from certain kernels, as developed in \citet{deya2008roughvolterraequations1, tindel2008roughvolterraequations2, harang2021volterraequationsdrivenrough, harang2022volterraequationsdrivenrough, 10.1214/22-EJP890}, as well as the {discrete delay signature} introduced in \citet*{neuenkirch2007delayequationsdrivenrough}.  Compared to these works, in the present paper we do not rely on such specialized signature constructions. 
Instead, we leverage the standard (time-augmented) signature, which offers several advantages: 
it provides a unified and versatile framework for expanding both Volterra and delay processes, 
it preserves the classical Chen identity and shuffle product, 
and it yields a Markovian lift of our inherently non-Markovian dynamics, making it particularly convenient for computations (e.g.~conditional moments in Section~\ref{sec:conditional_moments}).    Beyond these considerations, we stress that our contribution goes significantly further than previous works that focus primarily on formal expansions. While existing literature, including Volterra-based approaches, typically produces {finite} stochastic Taylor expansions, we {rigorously establish the convergence of the infinite signature expansion} for the class of processes we consider.
}\\

    \textbf{Outline.} Section~\ref{S:prel} deals with some preliminaries on path-signatures. Section~\ref{sec:inf_bracket_sig} is concerned with infinite linear combinations of signature elements and their properties. Section~\ref{sec:representations} collects our representations formulae together with a numerical illustration. Section~\ref{sec:conditional_moments} presents series expansions for the moments. The proofs are postponed to later sections and to the Appendices.

\section{Preliminaries}\label{S:prel}

    In this section, we setup the framework for dealing with signatures of semimartingales. We not only collect some (well-known) properties of finite linear combinations of signature elements, refer also to the first sections in \cite{bayer2023optimal,cuchiero2023signature,arribas-nonparam}, but we also introduce and study the resolvent object in Section~\ref{subsec:resolvent} which plays a crucial role in our main results.

\subsection{Tensor algebra}

    Let $d \in \N$ and denote by $\otimes$ the tensor product over $\R^d$, e.g. $(x \otimes y \otimes z)_{ijk} = x_i y_j z_k$, for $i,j,k = 1, \dots, d$, for $x,y,z \in \R^d$. For $n \geq 1$, we denote by $(\R^d) \conpow{n}$ the space of tensors of order $n$ and by $(\R^d) \conpow{0} = \R$. In the sequel, we will consider mathematical objects, path signatures, that live on the extended tensor algebra space $\eTA $ over $\R^d$, that is the space of (infinite) sequences of tensors defined by
    $$ \eTA := \left\{ \bell = (\bell^n)_{n=0}^\infty : \bell^n \in (\R^d) \conpow{n} \right\}. $$
    Similarly, for $M \geq 0$, we define the truncated tensor algebra $\tTA{M}$ as the space of sequences of tensors of order at most $M$ defined by
    $$ \tTA{M} := \left\{ \bell \in \eTA : \bell^n = 0, \quad \forall n > M \right\},$$
    and the tensor algebra $\TA$ as the space of all finite sequences of tensors defined by
    $$ \TA := \bigcup_{M \in \N } \tTA{M}. $$
    We clearly have $\TA \subset \eTA$. For $\bell = (\bell^n)_{n \in \N}, \bphi = (\bphi^n)_{n \in \N} \in \eTA$ and $\lambda \in \R$, we define the following operations:
    \begin{align*}
        \bell + \bphi :&
        = (\bell^n + \bphi^n)_{n\in \N}, \quad  \bell \otimes \bphi :
        = \left( \sum_{k=0}^n \bell^k \otimes \bphi^{n-k} \right)_{n \in \N}, \quad 
        \lambda \bell :
        = (\lambda \bell^n)_{n \in \N}.
    \end{align*}
    For the rest of the paper, we will use $\bell \bphi$ and $\bell \otimes \bphi$ interchangeably.
    These operations are closed on $\tTA{M}$ and $\TA$. \\
    
    \textbf{Important notations.} Let $\set{e_1, \dots, e_d} \subset \R^d$ be the canonical basis of $\R^d$ and $\alphabet = \set{\word{1}, \word{2}, \dots, \word{d}}$ be the corresponding alphabet. To ease reading, for $i \in \set{1, \dots, d}$, we write $e_{i}$ as the blue letter $\word{i}$ and for $n \geq 1,$ $i_1, \dots, i_n \in \set{1, \dots, d}$, we write $e_{i_1} \otimes \cdots \otimes e_{i_n} $ as the concatenation of letters $\word{i_1 \cdots i_n}$, that we call a word of length $n$. We note that the family $(e_{i_1} \otimes \cdots \otimes e_{i_n})_{(i_1, \dots, i_n) \in \set{1, \dots, d}^n}$ is a basis of $(\R^d) \conpow{n}$ that can be identified with the set of words of length $n$ defined by 
    \begin{equation} \label{def:sig_basis}
        V_n := \set{\word{i_1 \cdots i_n}: \word{i_k} \in \alphabet \text{ for } k = 1, 2, \dots, n}. 
    \end{equation} 
    
    Moreover, we denote by $\emptyword$ the empty word and we set $V_0 = \set{\emptyword}$ which serves as a basis for $(\R^d) \conpow{0} = \R$. {\revtwo Equipped with the product topology, $\eTA$ admits $V := \bigcup_{n\ge 0} V_n$ as a Schauder topological basis.} In particular, every $\bell \in \eTA$ can be decomposed as
    \begin{equation} \label{eq:sig_expansion}
        \bell = \sum_{n=0}^{\infty} \sum_{\word{v} \in V_n} \bell^{\word{v}} \word{v},
    \end{equation}
    where $\bell^{\word{v}} \in \R$ is the real coefficient of $\bell$ at coordinate $\word{v}$. {\revtwo Representation \eqref{eq:sig_expansion} induces a natural linear homeomorphism from $\eTA$ to $\R^\N$, showing that it is a locally convex space.} This representation will be frequently used in the paper. We stress again that in the sequel, every blue `word' $\word{v} \in V$ represents an element of the canonical basis of $\eTA$, i.e. there exists $n \geq 0$ such that $\word{v}$ is of the form $\word{v} = \word{i_1 \cdots i_n}$, which represents the element $e_{i_1} \otimes \cdots \otimes e_{i_n} $. The concatenation $\bell \word{v}$ of an element $\bell \in \eTA$ and the word $\word{v} = \word{i_1 \cdots i_n}$ means $\bell \otimes e_{i_1} \otimes \cdots \otimes e_{i_n}$.
    
    In addition to the decomposition \eqref{eq:sig_expansion} of elements $\bell \in \eTA$, we introduce the projection $\bell \proj{u} \in \eTA$ as
    \begin{equation} \label{eq:projection}
        \bell \proj{u} := \sum_{n=0}^{\infty} \sum_{\word{v} \in V_n} \bell^{\word{vu}} \word{v}
    \end{equation}
    for all $\word{u} \in V$. The projection plays an important role in the space of iterated integrals as it is closely linked to partial differentiation, in contrast with the concatenation that relates to integration. It will be used throughout the paper.
    
    \begin{sqremark} \label{rem:proj-decomposition}
        The projection allows us to decompose elements of the extended tensor algebra $\bell \in \eTA$ as 
        \begin{align}
            \bell = \bell^\emptyword \emptyword + \sum_{\word{i} \in \alphabet} \bell \proj{i} \word{i}.
        \end{align}
        For instance, for $\bell = 4 \cdot \emptyword + 3 \cdot \word{1} - 1 \cdot \word{12} + 2 \cdot \word{2212}$, we have that  $\bell^{\emptyword} = 4$, $\bell \proj{1} = 3 \cdot \emptyword$, $\bell \proj{2} = - 1 \cdot \word{1} + 2 \cdot \word{221}$
        and $\bell \proj{3} = 0$.
    \end{sqremark}
   
    We define the bracket between $\bell \in \TA$ and ${\revtwo \mathbb{X}} \in \eTA$ by
    \begin{align} \label{eq:dualdef}
        \langle \bell, {\revtwo \mathbb{X}} \rangle 
        = \sum_{n=0}^{\infty} \sum_{\word{v} \in V_n} \bell^\word{v} {\revtwo \mathbb{X}}^\word{v}. 
    \end{align}
    Notice that it is well-defined as $\bell$ has finitely many non-zero terms. For $\bell \in \eTA$, the series in \eqref{eq:dualdef} involves infinitely many terms and requires special care, this will be discussed in Section \ref{sec:inf_bracket_sig} below.
  
    We will also consider another operation on the space of words $V$, the shuffle product. The shuffle product plays a crucial role for an integration by parts formula on the space of iterated integrals, see Lemma~\ref{lem:shuffleproperty} below.
    
    \begin{definition}[Shuffle product] \label{def:shuffleprod}
        The shuffle product $\shuprod: V \times V \to \TA$ is defined inductively for all words $\word{v}$ and $\word{w}$ and all letters $\word{i}$ and $\word{j}$ in $\alphabet$ by
        \begin{align*}
            (\word{v} \word{i}) \shuprod (\word{w} \word{j}) &
            = (\word{v} \shuprod (\word{w} \word{j})) \word{i} + ((\word{v} \word{i}) \shuprod \word{w}) \word{j}
            \\ \word{w} \shuprod \emptyword &
            = \emptyword \shuprod \word{w} = \word{w}.
        \end{align*}
        
        With some abuse of notation, the shuffle product on $\eTA$ induced by the shuffle product on $V$ will also be denoted by $\shuprod$. The shuffle product is clearly commutative. See \cite{gainesshuffle,reeshuffles} for more information on the shuffle product.
    \end{definition}
    
    The shuffle product corresponds to the shuffling of two decks of cards together, while keeping the order of each single deck as illustrated on the following example: $\word{1} \shuprod \word{23} =  \word{123} + \word{213} + \word{231}$.

\subsection{Resolvent and linear equation}\label{subsec:resolvent}

    For $n \in \N$ and $\bell \in \eTA$, we define the concatenation power of $\bell$ by
    $$ \bell \conpow{n} := \overbrace{\bell \otimes \bell \otimes \cdots \otimes \bell}^{\text{$n$ times}}, $$
    with the convention $\bell \conpow{0} = \emptyword$.

    For every $\bell \in \eTA$ such that $\bell^\emptyword = 0$, we define its \textit{resolvent} by
    \begin{equation} \label{nota:resolvent}
        \inverse{\emptyword - \bell} = \sum_{n=0}^\infty \bell \conpow{n}.
    \end{equation}

    {\revtwo
    \begin{sqremark} \label{rmk:resolvent_scaled}
        The condition $\bell^\emptyword = 0$ can easily be loosened to $\bell^\emptyword \neq 1$ by scaling and shifting $\bell$, i.e.
        \begin{equation}
            \inverse{\emptyword - \bell} = \frac{1}{1 - \bell^\emptyword} \inverse{\emptyword - \frac{1}{1 - \bell^\emptyword} \bra{\bell - \bell^\emptyword \emptyword}} = \sum_{n=0}^\infty \frac{1}{\bra{1 - \bell^\emptyword}^{n+1}} \bra{\bell - \bell^\emptyword \emptyword} \conpow{n}.
        \end{equation}
        We thus write $\res{\bell}$ for all $\bell \in \eTA$ such that $\bell^\emptyword \neq 1$.
    \end{sqremark}
    }
    
    \begin{sqremark} \label{rmk:resolvent_is_inverse}
        It is clear that $\bra{\emptyword - \bell} \inverse{\emptyword - \bell} = \inverse{\emptyword - \bell} \bra{\emptyword - \bell} = \emptyword$.
    \end{sqremark}
    
    The purpose of the resolvent is to solve linear algebraic equations as shown in the next proposition.
    \begin{proposition} \label{prop:resolvent-solution}
        Let $\bp, \bq \in \eTA$ such that $\bq^\emptyword \neq 1$, then the unique solution $\bell \in \eTA$ to the linear algebraic equation 
        \begin{align} \label{eq:resolvent}
            \bell = \bp + \bell \bq
        \end{align}
        is given by 
        \begin{align}
            \bell = \bp \inverse{\emptyword - \bq}.
        \end{align}
    \end{proposition}
    
    \begin{proof}
        It is easy to verify that $\bell = \bp \inverse{\emptyword - \bq}$ is a solution of \eqref{eq:resolvent}. On the other hand, for every $\bphi \in \eTA$ s.t. $\bphi = \bp + \bphi \bq$, we directly derive $\bphi(\emptyword - \bq) = \bp = \bell(\emptyword - \bq)$ and then $(\bphi - \bell)(\emptyword - \bq) = 0$. Hence $\bphi = \bell$, verifying the uniqueness.
    \end{proof}
    
    Interestingly, whenever $\bell$ is a linear combination of single letters, the resolvent of $\bell$ is equal to the shuffle exponential $\shuexp{\bell}$ defined by
    \begin{equation} \label{nota:shuexp}
        \shuexp{\bell} := \sum_{n=0}^{\infty} \frac{\bell \shupow{n}}{n!}
    \end{equation}
    with
    $$ \bell \shupow{n} := \overbrace{\bell \shuprod \bell \shuprod \cdots \shuprod \bell}^{\text{$n$ times}}, \quad n \geq 1, \quad \bell \shupow{0} = \emptyword. $$
    
    \begin{proposition} \label{prop:resolvent-shuexp}
        Let $\bm{b} = \sum_{\word{i} \in \alphabet} \bm{b}^\word{i} \word{i}$, with $\bm{b}^\word{i} \in \R$, we have that
        
        \begin{equation} \label{eq:prop_2.3}
            \inverse{\emptyword - \bm{b}} = \shuexp{\bm{b}}.
        \end{equation}
        
        In particular, this implies     
        \begin{equation} \label{eq:prop_2.3_cor}            
            \shuexp{\bm{b}} = \emptyword + \shuexp{\bm{b}} \bm{b} = \emptyword + \bm{b} \shuexp{\bm{b}}.
        \end{equation} 
    \end{proposition}
    
    \begin{proof}
        We first observe that $\bm{b}^{\emptyword} = 0$, so $\res{\bm{b}}$ is well-defined. We then prove that $\frac{1}{k!} \bm{b} \shupow{k} = \bm{b} \conpow{k}$.
        By induction, it suffices to prove for every $k \in \N^*$,
        \begin{equation} \label{eq:lemma_for_prop2.3}
            \bm{b} \conpow{k-1} \shuprod \bm{b} = k \bm{b} \conpow{k}.
        \end{equation}
        It is easy to check that \eqref{eq:lemma_for_prop2.3} is satisfied for $k = 1$. Now, assuming it holds for some $k \geq 1$, we can deduce
        \begin{align}
            \bm{b} \conpow{k} \shuprod \bm{b}
            &= \sum_{\word{i} \in \alphabet} \bm{b} \conpow{k} \bm{b}^{\word{i}} \word{i} + \sum_{\word{i} \in \alphabet} \bra{\bm{b} \conpow{k} \proj{i} \shuprod \bm{b}} \word{i}
            \\
            &= \bm{b} \conpow{k+1} + \sum_{\word{i} \in \alphabet} \bra{\bm{b} \conpow{k-1} \shuprod \bm{b}} \bm{b}^{\word{i}} \word{i}
            \\
            &= \bm{b} \conpow{k+1}+ \bra{\bm{b} \conpow{k-1} \shuprod \bm{b}} \bm{b}
            \\
            &= \bm{b} \conpow{k+1} + (k \bm{b} \conpow{k}) \bm{b}
            \\
            & = (k+1) \bm{b} \conpow{k+1},
        \end{align}
        which proves 
        \eqref{eq:lemma_for_prop2.3} by induction, hence proving the equality \eqref{eq:prop_2.3}. Noticing that $\res{\bm{b}} = \emptyword + \bm{b} \res{\bm{b}} = \emptyword + \res{\bm{b}} \bm{b}$, we derive \eqref{eq:prop_2.3_cor}.
    \end{proof}

\subsection{Signature of a continuous semimartingale} \label{subsec:finitesignature}

    We define the (path) signature of a semimartingale as the sequence of iterated stochastic integrals in the sense of Stratonovich. Throughout the paper, the Itô integral is denoted by $\int_0^\cdot Y_t \d X_t$ and the Stratonovich integral by $\int_0^\cdot Y_t \circ \d X_t$. If both $X$ and $Y$ are semimartingales then, $\int_0^\cdot Y_t \circ \d X_t := \int_0^\cdot Y_t \d X_t + \half [X, Y]_\cdot$.
    
    \begin{definition}[Signature] \label{def:sig} Fix $T > 0$.
        Let $(X_t)_{t \geq 0}$ be a continuous semimartingale on $\R^d$ on a complete filtered probability space $(\Omega, \mathcal{F}, (\mathcal{F}_t)_{t \geq 0}, \P)$. The signature of $X$ is defined by
        \begin{align*}
            \mathbb{X}: \Omega \times [0, T] &
            \to \eTA
            \\ (\omega, t) &
            \mapsto \sigX (\omega) := (1, \sigX^1(\omega), \dots, \sigX^n(\omega), \dots),
        \end{align*}
        where
        $$ \sigX^n := \int_{0 < u_1 < \cdots < u_n < t} \circ \d X_{u_1} \otimes \cdots \otimes \circ \d X_{u_n} $$
        takes value in $(\R^d)^{\otimes n}$, $n \geq 0$. Similarly, the truncated signature of order $M \in \N$ is defined by
        \begin{align} \label{def:sig-trunc}
            \mathbb{X}^{\leq M}: [0, T] &
            \to \tTA{M}
            \\ (\omega, t) &
            \mapsto \sigX^{\leq M}(\omega) := (1, \sigX^1(\omega), \dots, \sigX^M(\omega), 0, \dots, 0, \dots).
        \end{align}
    \end{definition}
    
    The signature plays a similar role as polynomials on path-space. Indeed, in dimension $d=1$, the signature of $X$ is the sequence of monomials $\left( \frac{(X_t - X_0)^n}{n!} \right)_{n \in \N}$. In particular, for $\bell \in T^M(\R)$, $\bracketsigX{\bell}$ is a polynomial in $(X_t-X_0)$ of degree no greater than $M$.
    
    \begin{sqexample}
        For $X_t = (t, W_t^1, W_t^2, \dots, W_t^d)$, where $(W_t^1, W_t^2, \dots, W_t^d)$ is a $d$-dimensional Brownian motion, we can compute the first level of the  signature
        $$ \sigX^1 = \left( t, W_t^1, W_t^2, \cdots, W_t^d \right), $$
        and the  second level
        $$ \sigX^2 =
        \begin{pmatrix}
            \frac{t^2}{2} & \int_0^t s \d W^1_s & \cdots & \int_0^t s \d W_s^d \\
            \int_0^t W_s^1 \d s & \frac{\left( W_t^1 \right)^2}{2} & \cdots & \int_0^t W_s^1 \circ \d W_s^d  \\
            \vdots & \vdots & \ddots & \vdots \\
            \int_0^t W_s^d \d s & \int_0^t W_s^d \circ \d W_s^1 & \cdots & \frac{\left( W_t^d \right)^2}{2}
        \end{pmatrix}
        . $$
    \end{sqexample}
    
    \begin{sqremark} \label{rmk:sig_iteration_def}
        It follows from Definition~\ref{def:sig} that the level $n$ of the signature $\sigX^n := (\sigX^{\word{i_1 \cdots i_n}})_{(\word{i_1 \cdots i_n}) \in V_n} \in (\R^d) \conpow{n}$ can be written  in  the following iterated form
        \begin{align} \label{eq:signdef2}
            \sigX^{\word{i_1 \cdots i_n}} = \int_0^t \sigX[s]^{\word{i_1 \cdots i_{n-1}}} \circ \d X_s^{\word{i_n}}, \quad {\word{i_1 \cdots i_n}} \in V_n.
        \end{align} 
    \end{sqremark}

\subsection{Finite linear combinations of signature elements}

    In this section, we fix $X$ a continuous semimartingale on $\R^d$ and we let $\mathbb{X}$ be its signature. We recall some important properties of finite linear combinations of signature elements $\bracketsigX{\bell}$ for $\bell \in \TA$, recall \eqref{eq:dualdef}. \\
    
    The first one highlights a crucial linearization property of the signature, the product of two linear combinations of the signature is again a linear combination of the signature where the coefficients are given by the shuffle product, recall Definition~\ref{def:shuffleprod}. This property will be later extended to infinite linear combinations of signatures in Proposition~\ref{prop:shuffle_property} below.
    
    \begin{lemma}[Shuffle product property] \label{lem:shuffleproperty}
        For $\bell, \bphi \in \TA$,
        $$ \bracketsigX{\bell} \bracketsigX{\bphi} = \bracketsigX{\bell \shuprod \bphi}. $$
    \end{lemma}
    
    \begin{proof}
        This follows from the standard chain rule of Stratonovich integrals, see \citet[Proposition 2.2]{gainesshuffle}.
    \end{proof}
    
    The shuffle product can be seen as an extension of the Cauchy product to the space of signatures. For example, set $d=1$, $X_0=0$ and $\bell \in T^M(\R), \bphi \in T^N(\R)$ for some $N, M \in \N$, i.e. there exists $(a_m)_{m \leq M}, (b_n)_{n \leq N} \in \R$ such that $\bell = \sum_{m=0}^M a_m \word{1} \conpow{m}$, $\bphi = \sum_{n=0}^N b_n \word{1} \conpow{n}$. Then,
    \begin{align*}
        \bracketsigX{\bell} \bracketsigX{\bphi} &
        = \sum_{m=0}^M a_m \frac{(X_t)^m}{m!} \sum_{n=0}^N b_n \frac{(X_t)^n}{n!}
        = \sum_{k=0}^{M+N} c_k \frac{(X_t)^k}{k!}
    \end{align*}
    with $c_k = \sum_{i=0}^{k} \binom{k}{i} a_i b_{k-i}$. In this case $\bell \shuprod \bphi = \sum_{k=0}^{M+N} c_k \word{1} \conpow{k}$. \\
    
    The second property shows that any finite linear combination of signature elements is a semi-martingale with an explicit decomposition given in terms of the projection elements, recall \eqref{eq:projection}. This semi-martingality property is no longer always true when infinite linear combinations of signature elements are considered, see Theorem~\ref{thm:sig_ito_formula} below. 
    
    \begin{lemma} \label{lem:derivsig}
        Let $\bell \in \TA$, then $(\bracketsigX{\bell})_{t\geq 0}$ is a semimartingale with decomposition
        \begin{align} \label{eq:decompsitionsemi}
            \d \bracketsigX{\bell} = \sum_{\word{i} \in \alphabet} \bracketsigX{\bell \proj{i}} \circ \d X_t^{\word{i}}.
        \end{align}
    \end{lemma}
        
    \begin{proof}
        It follows from the equivalent definition of the signature \eqref{eq:signdef2} that
        $$ \bracketsigX{ \bell \proj{i} \word{i}} = \int_0^t \bracketsigX[s]{ \bell \proj{i}} \circ \d X_s^{\word{i}}, \quad \word{i} \in \alphabet. $$
        Then, using the projection relation \eqref{eq:projection} and the linearity of the bracket, we get that 
        \begin{align*}
            \bracketsigX{\bell} 
            = \bracketsigX{ \bell^\emptyword \emptyword + \sum_{\word{i} \in \alphabet} \bell \proj{i} \word{i}}
            = \bell^\emptyword + \sum_{\word{i} \in \alphabet} \int_0^t \bracketsigX[s]{\bell \proj{i}} \circ \d X_s^{\word{i}}.
        \end{align*}
    \end{proof}
    
    Finally, we show how to transform the Stratonovich integral into the Itô integral in the context of finite linear combination of the signature.

    \begin{lemma} \label{lem:stratoito}
        Let $\bell \in \TA$, then for every $\word{i} \in \alphabet$:
        $$ \bracketsigX{\bell} \circ \d X_t^{\word{i}} = \bracketsigX{\bell} \d X_t^{\word{i}} + \half\sum_{\wj\in\alphabet} \bracketsigX{\bell \proj{j}} \d [X^{\wj},X^{\wi}]_t. $$
    \end{lemma}
    
    \begin{proof}
        Denote by $Y_t = \bracketsigX{\bell}$ and fix $\word{i} \in \alphabet$. It follows from Lemma~\ref{lem:derivsig} that $Y$ is a semimartingale with dynamics
        \begin{align*}
             \d Y_t = \sum_{\word{j} \in \alphabet} \bracketsigX{\bell \proj{j}} \circ \d X_t^\word{j}.
        \end{align*}
        Therefore, by the definition of the Stratonovich integral:
        \begin{align*}
            Y_t \circ \d X_t^\word{i} = Y_t \d X_t^{\word{i}} + \half \d [Y, X^\word{i}]_t = Y_t \d X_t^\word{i} + \half \sum_{\wj\in\alphabet}\bracketsigX{\bell \proj{j}} \d [X^{\wj},X^{\wi}]_t. 
        \end{align*}
    \end{proof}

\section{Infinite linear combinations of signature elements} \label{sec:inf_bracket_sig}
    From now on, we fix a  complete filtered probability space ($\Omega, \F,\mathbb F:=(\F_t)_{t \in [0, T]}, \P$) satisfying the usual conditions, and $\bra{X_t}_{0 \leq t \leq T}$ a continuous $\mathbb F$-semimartingale on $\R^d$ and we let $\mathbb{X}$ be its signature. In Section~\ref{subsec:finitesignature}, we recalled well-known properties of finite linear combinations of signature elements $\bracketsigX{\bell}$ for $\bell \in \TA$. In this section, we study infinite linear combinations $\bracketsigX{\bell}$ for certain admissible $\bell \in \eTA$ for which the infinite series will make sense. Our central results in this section are an Itô formula for such elements in Theorem~\ref{thm:sig_ito_formula} below, {a control for their supremum norm in Propositions~\ref{prop:uniform_bound_L2} and \ref{prop:Ah},} and an estimate of solutions of algebraic linear equations taking values in the extended tensor algebra in Theorem~\ref{thm:ell_in_Sh_if_q_dominated}. \\

    For this, we consider the space $\A(\mathbb{X})$ of admissible elements $\bell$ using the associated stochastic semi-norm:
    
    $$ \norm{\bell}_t^{\A(\mathbb{X})} := \sum_{n=0}^\infty \left| \sum_{\word{v} \in V_n} \bell^\word{v} \sigX^\word{v} \right|, \quad t \geq 0, $$
    
    recall the definition of $V_n$ in \eqref{def:sig_basis} and the decomposition \eqref{eq:sig_expansion}. $\norm{\bell}_t^{\A(\mathbb{X})} $ is actually random variables. Whenever $\norm{\bell}_t^{\A(\mathbb{X})}< \infty$ a.s., the infinite linear combination 
    $$ \bracketsigX{\bell} := \sum_{n=0}^\infty \sum_{\word{v} \in V_n} \bell^\word{v} \sigX^\word{v} $$
    is well-defined. This leads to the following definition for the admissible set $\A(\mathbb{X}):$
    $$ \A(\mathbb{X}) := \set{ \bell \in \eTA: \norm{\bell}_t^{\A(\mathbb{X})} < \infty \text{ for all } t \in [0, T] \text{ a.s.} }. $$
    
    For simplicity, we write $\A$ (resp. $\normA{\cdot}$), in place of $\A(\mathbb{X})$ (resp. $\norm{\cdot}_t^{\mathcal{A}(\mathbb{X})}$), when it does not cause ambiguity. Note that $\TA \subset \A$ and that $\bracketsigX{\bell}$  extends \eqref{eq:dualdef}, as the two bracket operations $\langle \cdot, \cdot \rangle$ coincide whenever $\bell \in \TA$.

    \begin{sqremark}
        The space $\A$ plays a crucial role in this paper. Whenever we talk about $\bracketsigX{\bell}$ for $\bell \in \eTA$, we have to first verify that $\bell$ belongs to $\A$ to ensure that $\bracketsigX{\bell}$ is well-defined. The set $\A$ is a natural set to consider. It has also been studied in \citet*[Example 3.2 (iii)]{cuchiero2023signature}. In practice, however, we will use a stronger but more tractable criterion in Section~\ref{subsec:Sh} below to obtain that a given $\bell$ belongs to $\A$.
    \end{sqremark}

    \begin{sqremark}
        If $\bell \in \A$, then $\bracketsigX{\bell}$ is a progressively measurable process, as it is the a.s.-limit of a series of progressively measurable processes.
    \end{sqremark}    
    
    {\revtwo
    \begin{sqremark} \label{rmk:A_tilde}
       We could restrict to a slightly smaller space $\tilde{\A}$:
        $$ \tilde{\A}(\mathbb{X}) := \set{\bell \in \eTA: \sup_{t \in [0, T]} \norm{\bell}_t^{\A(\mathbb{X})} < \infty \text{ a.s.}}, $$
        and show that whenever $X$ and $Y$ follow the same law, $\tilde{\A}(\mathbb{X}) = \tilde{\A}(\mathbb{Y})$. This would however not be the case in general for $\A$ due to measurability issue.
    \end{sqremark}
    }

    It is straightforward to see that the sets $\A$ is closed under linear operations, i.e.~for $\alpha, \beta \in \R$ and $\bell, \bphi \in \A$, then $\alpha \bell + \beta \bphi \in \A$. Moreover, $\A$ is closed under the shuffle product.
    
    The next proposition extends the shuffle property, derived in Proposition~\ref{lem:shuffleproperty}, to the case of infinite linear combination of signatures. We note that it can be obtained from a particular instance of the more general result in \cite[Lemma 4.4]{cuchiero2023signature} dealing with so-called shuffle compatible partitions. For completeness we provide a proof for our specific setting in Appendix~\ref{A:proofshuffle}.
 
    \begin{proposition}[Extended Shuffle property] \label{prop:shuffle_property}
        If $\bell, \bphi \in \A$, then $\bell \shuprod \bphi \in \A$ and
        $$ \bracketsigX{\bell} \bracketsigX{\bphi} = \bracketsigX{\bell \shuprod \bphi}. $$
    \end{proposition}
    
    \begin{proof}
        The proof is provided in Appendix~\ref{A:proofshuffle}.
    \end{proof}


\subsection{The space \texorpdfstring{$\I$}{I} and Itô's formula} 

   The aim of this subsection is the derivation of an Itô's formula in Theorem~\ref{thm:sig_ito_formula} below for infinite linear combinations of signature elements. In preparation, we define $\mathcal{Q}_\word{i}(X)$, the set of coordinates of $X$ that have a non-zero quadratic covariation with $X^\word{i}$:
   $$ \mathcal{Q}_\word{i}(X) := \set{\word{j} \in \alphabet: [X^\word{j}, X^\word{i}]_t \neq 0, \, \text{on a set of non-zero $\d t \otimes \d \P$ measure}}, \quad \word{i} \in \alphabet. $$
    
    \begin{sqremark}
        $\mathcal{Q}_\word{i}(X)$ is introduced in order to get rid of unnecessary assumptions in Lemma~\ref{thm:sig_semimartingale}, since we only need $\bell \proj{j} \in \A$ for those $\word{j} \in \mathcal{Q}_\word{i}(X)$. For example, in the particular case of $X_t = (t, W_t)$, with $W$ a one-dimensional Brownian motion, assumptions are only needed for the terms $\bell, \bell \proj{1}, \bell \proj{2}, \bell \proj{22} \in \A$, but not for $\bell \proj{21}$ and $\bell \proj{12}$.
    \end{sqremark}
    
    \begin{lemma} \label{thm:sig_semimartingale}
        Fix $\word{i} \in \alphabet$ and let $\bell \in \A$ be such that  
        
        \begin{equation} \label{eq:assumption_3.1}
            \bell \word{i} \in \A, \quad \int_0^T \left( \normA{\bell} \right)^2 \d [X^\word{i}, X^\word{i}]_t < \infty \text{ and } \int_0^T \normA{\bell} \d \abs{F_t^\word{i}} < \infty,
        \end{equation}
        where $F^{\word{i}}$ is the finite variation part of $X^{\word{i}}$. We also assume that for every $\word{j} \in \mathcal{Q}_\word{i}(X)$,
        
        \begin{equation} \label{eq:assumption_ito_decop}
            \bell \proj{j} \in \A \quad \text{and} \quad \int^T_0 \normA{\bell \proj{j}} \d [X^\word{j}, X^\word{i}]_t < \infty. 
        \end{equation}
        Then,
        \begin{align} \label{eq:lemmaito}
            \bracketsigX{\bell \word{i}} = \int_0^t \bracketsigX[s]{\bell} \d X_s^\word{i} + \half \sum_{\word{j} \in \alphabet} \int^t_0 \bracketsigX[s]{\bell \proj{j}} \d [X^\word{j}, X^\word{i}]_s,  \quad \text{a.s.}
        \end{align}
        for every $t \in [0, T]$. 
    \end{lemma}

    \begin{proof}
        Fix $t \in [0, T]$. Clearly, under the assumptions \eqref{eq:assumption_3.1}-\eqref{eq:assumption_ito_decop} on $\bell$, the right side of equation \eqref{eq:lemmaito} is well-defined.  For every word $\word{v} \in V$, we can write it as $\word{v} = \sum_{\word{j} \in \alphabet} \word{v} \proj{j} \word{j}$. By Lemma~\ref{lem:stratoito}, we obtain:  
        \begin{align}
            \bracketsigX{\word{vi}} = \int_0^t\bracketsigX{\wv}\circ \d X^{\wi}_t
         = \half \sum_{\word{j} \in \alphabet} \int_0^t \bracketsigX[s]{\word{v} \proj{j}} \d [X^\word{j}, X^\word{i}]_s + \int_0^t \bracketsigX[s]{\word{v}} \d X_s^\word{i}.
        \end{align}
        
        Combined with the decomposition \eqref{eq:sig_expansion}, it yields that 
        \begin{align}
            \bracketsigX{\bell \word{i}} &
            = \sum_{n=0}^\infty \sum_{\word{v} \in V_n} \bell^\word{v} \bracketsigX{\word{vi}}
            = \sum_{n=0}^\infty \sum_{\word{v} \in V_n} \bell^\word{v} \int_0^t \left( \half \sum_{\word{j} \in \alphabet} \bracketsigX[s]{\word{v} \proj{j}} \d [X^\word{j}, X^\word{i}]_s + \bracketsigX[s]{\word{v}} \d X_s^\word{i} \right). \label{eq:proof_sig_ito_1}
        \end{align}
        
        It remains to argue, using dominated convergence, that both terms appearing on the right-hand side of \eqref{eq:proof_sig_ito_1} correspond to the terms in \eqref{eq:lemmaito}. \\
        \begin{enumerate}[label=\normalfont(\textbf{\roman*})]
            \item
            For the first term, if $\word{j} \notin \mathcal{Q}_\word{i}(X)$, $[X^\word{j}, X^\word{i}]_t = 0$ for all $t$, so we only care about the letters $\word{j} \in \mathcal{Q}_\word{i}(X)$. Since $\bell \proj{j} = \sum_{n=0}^\infty \sum_{\word{v} \in V_n} \bell^\word{v} \word{v} \proj{j} \in \A$ by assumption \eqref{eq:assumption_ito_decop}, we can apply the dominated convergence theorem, with $\normA{\bell \proj{j}}$ as the dominated function, and get
            \begin{align} 
                \sum_{n=0}^\infty \sum_{\word{v} \in V_n} \bell^\word{v} \int_0^t \bracketsigX[s]{\word{v} \proj{j}} \d [X^\word{j}, X^\word{i}]_s &
                = \int_0^t \left( \sum_{n=0}^\infty \sum_{\word{v} \in V_n} \bell^\word{v} \bracketsigX[s]{\word{v} \proj{j}} \right) \d [X^\word{j}, X^\word{i}]_s
                \\
                &= \int_0^{t} \bracketsigX[s]{\bell \proj{j}} \d [X^\word{j}, X^\word{i}]_s\label{proof_sig_ito_3}
            \end{align}

            \item
            For the second term, let us define
            
            $$ Y_t^N := \sum_{n=0}^N \sum_{\word{v} \in V_n} \bell^\word{v} \int_0^t \bracketsigX[s]{\word{v}} \d X_s^\word{i}. $$
            
            On the one hand, as $\int_0^T \bra{\normA{\bell}}^2 \d [X^\word{i}, X^\word{i}]_t < \infty$ and $\sum_{n=0}^N \sum_{\word{v} \in V_n} \bell^\word{v} \bracketsigX{\word{v}} \leq \normA{\bell}$ for every $N$ and assumption \eqref{eq:assumption_3.1}, by the dominated convergence theorem for stochastic integrals, the random variable $Y^N$ converges in probability to 
            
            $$ Y_t^\infty := \int_0^t \sum_{n=0}^\infty \sum_{\word{v} \in V_n} \bell^\word{v} \bracketsigX[s]{\word{v}} \d X_s^\word{i} $$
            as $N \to \infty$.

            On the other hand, since $\bell \word{i} \in \A$, the left-hand side of \eqref{eq:proof_sig_ito_1} is a.s.~finite, and since second term of its right-hand side has been proved to be finite, it follows that $\sum_{n=0}^\infty \sum_{\word{v} \in V_n} \bell^\word{v} \int_0^t \bracketsigX[s]{\word{v}} \d X_s^\word{i}$ exists a.s.~and corresponds to an almost sure limit of $Y^N$ as $N \to \infty$. Which yields
            
            \begin{equation} \label{proof_sig_ito_2}
                \sum_{n=0}^\infty \sum_{\word{v} \in V_n} \bell^\word{v} \int_0^t \bracketsigX[s]{\word{v}} \d X_s^\word{i} 
                = Y^\infty_t
                = \int_0^t \bracketsigX[s]{\bell} \d X_s^\word{i} \quad {a.s.}
            \end{equation}
        \end{enumerate}
        
        Combining \eqref{eq:proof_sig_ito_1}, \eqref{proof_sig_ito_3} and \eqref{proof_sig_ito_2} yields \eqref{eq:lemmaito} and completes the proof. 
    \end{proof}
   
    \begin{sqremark}\label{rmk:3.5}
        To obtain that equality \eqref{eq:lemmaito} holds for every $t$ almost surely, one could argue that $\bracketsigX{\bell \word{i}}$ is a continuous process for instance. This requires additional conditions on $\bell$, see for instance Proposition~\ref{prop:Ah}-\ref{P:continuity} below.
    \end{sqremark}

    We are now in place to state the main result of this section which is an Itô's formula for infinite linear combinations of signature elements. For this, we need to define the following set
    \begin{equation} \label{def:I}
        \I(X) := \set{
        \bell \in \A:
        \begin{matrix}
            \text{for every } \word{i} \in \alphabet \text{ and } \word{j} \in \mathcal{Q}_\word{i}(X), \; \bell \proj{i}, \bell \proj{ji} \in \A \text{ and } \hfill
            \\
            \int_0^T \left( \normA{\bell \proj{i}} \d \abs{F_t^\word{i}} + \normA{\bell \proj{ji}} \d [X^\word{j}, X^\word{i}]_t + \bra{\normA{\bell \proj{i}}}^2 \d [X^\word{i}, X^\word{i}]_t \right) < \infty, \; \text{a.s.} \hfill
        \end{matrix}
        },
    \end{equation}
    which can be seen as the analog of the space  $C^2(\R)$ when one applies Itô's formula on $f(W)$, see Example~\ref{ex:ito} below.

    \begin{theorem}[Itô's formula] \label{thm:sig_ito_formula}
         Let $\bell \in \I(X)$, then $(\bracketsigX{\bell})_{t \leq T}$ is an Itô process such that
         \begin{equation}\label{eq:Ito_formula}
             \bracketsigX{\bell}
             = \bell^\emptyword 
             + \sum_{\word{i} \in \alphabet} \int_0^t \bracketsigX[s]{\bell \proj{i}} \d X_s^\word{i} 
             + \half \sum_{\word{i} \in \alphabet} \sum_{\word{j} \in \alphabet} \int_0^t \bracketsigX[s]{\bell \proj{ji}} \d [X^\word{j}, X^\word{i}]_s,
         \end{equation}
        for all $t \leq T$. In particular, when we take $X_t$ to be $\widehat{W}_t = (t, W^2_t, \ldots, W^d_t)$, with $W$ a $(d-1)$-Brownian motion then, the set $\mathcal I$ reads 
        \begin{equation} \label{eq:I-hatW}
            \I(\widehat{W}) = \set{
            \bell \in \A:
            \begin{matrix}
                \bell \proj{1}, \bell \proj{j}, \bell \proj{jj} \in \A, \;  \word{j} \in \set{\word{2}, \dots, \word{d}}, \text{ and } \hfill
                \\
                \int_0^T \normA{\bell \proj{1}} \d t + \sum_{\word{j} \in \set{\word{2}, \dots, \word{d}}} \int_0^T \bra{\normA{\bell \proj{jj}} + \bra{\normA{\bell \proj{j}}}^2} \d t < \infty, \; \text{a.s.} \hfill
            \end{matrix}
            },
        \end{equation}
        and for every $\bell \in \I(\widehat{W})$:
        \begin{align} \label{eq:ItohatW}
            \bracketsig{\bell}
            = \bell^\emptyword
            + \int_0^t \bracketsig[s]{\bell\proj{1} + \half \sum_{\word{j} \in \set{\word{2}, \dots, \word{d}}} \bell \proj{jj}} \d s
            + \sum_{\word{j} \in \set{\word{2}, \dots, \word{d}}} \int_0^t \bracketsig[s]{\bell \proj{j}} \d W_s^\word{j}.    
        \end{align}
    \end{theorem} 
    
    \begin{proof}
        The proof follows from an application of Lemma~\ref{thm:sig_semimartingale} combined with the decomposition~\eqref{eq:projection}.
    \end{proof}

    The following example highlights that Theorem \ref{thm:sig_ito_formula} is indeed an extension of the usual Itô's formula on $f(W)$.
    \begin{sqexample} \label{ex:ito}
        Let $d=2$ and $\widehat{W}_t = (t, W_t)$, and fix an analytic function $f(y) := \sum_{n \geq 0} a_n y^n$ with infinite radius of convergence which we apply to $W$. It is clear that 
        $$ f(W_t) = \sum_{n \geq 0} a_n W_t^n = \bracketsig{\bell}, $$
        where
        $$ \bell = \sum_{n \geq 0} a_n \word{2} \shupow{n} = \sum_{n \geq 0} a_n  n! \word{2} \conpow{n}. $$
        
        In particular, the projections read:
        \begin{align*}
            \bell \proj{1} &
            = 0,
            \\ \bell \proj{2} &
            = \sum_{n \geq 0} a_{n+1} (n+1)! \word{2} \conpow{n}
            = \sum_{n \geq 0} a_{n+1} (n+1) \word{2} \shupow{n},
            \\ \bell \proj{22} &
            = \sum_{n \geq 0} a_{n+2} (n+2)! \word{2} \conpow{n}
            = \sum_{n \geq 0} a_{n+2} (n+1) (n+2) \word{2} \shupow{n}.
        \end{align*}
        
        It is easy to verify that $\bell, \bell \proj{2}, \bell \proj{22} \in \A$ since $f$ has infinite radius of convergence. We can thus further derive that 
        \begin{align*}
            f'(W_t) = \bracketsig{\bell \proj{2}},
            \quad 
            f''(W_t) = \bracketsig{\bell \proj{22}}.
        \end{align*}
        
        On the other hand we can see that $\normA{\bell \proj{2}} = \sum_{n=0}^\infty (n+1) |a_{n+1}| \cdot |W_t|^n$, and since $g(x) := \sum_{n \geq 0} (n+1) |a_{n+1}| x^n$ is also analytic and $\normA{\bell \proj{2}}$ has continuous sample path almost surely, then
        $$ \sup_{t \in [0, T]} \normA{\bell \proj{2}} < \infty. $$
        
        With similar arguments we can show that $\sup_{t \in [0, T]} \normA{\bell \proj{22}} < \infty.$ This allows us to verify that $\bell \in \I(\widehat W)$. An application of our  Itô's formula \eqref{eq:ItohatW} yields 
        $$ \bracketsig{\bell} = \bell^\emptyword + \int^t_0 \bracketsig[s]{\bell \proj{2}} \d W_s + \half \int^t_0 \bracketsig[s]{\bell \proj{22}} \d s, $$
        which translates into
        $$ f(W_t) = f(0) + \int_0^t f'(W_s) \d W_s + \half \int_0^t f''(W_s) \d s, $$
        the standard Itô's formula.
    \end{sqexample}

    \begin{sqremark}\label{rmk:ito_strong}
        It should be noted that the identity \eqref{eq:Ito_formula} only holds almost surely for every fixed $t$. However, as remark~\ref{rmk:3.5}, if we additionally assume $\bracketsigX{\bell}$ is continuous, then \eqref{eq:Ito_formula} holds for every $t \in [0, T]$ almost surely. In particular, when $X_t = \widehat{W}_t = (t, W_t^2, \dots, W_t^d)$, and $\bell\in \Ah$ (where $\Ah$ is introduced in Section \ref{subsec:Sh} below), the continuity assumption is satisfied by Proposition \eqref{prop:Ah}-\ref{P:continuity}. Our applications of Theorem \ref{thm:sig_ito_formula} in later section are all in this strong sense. 
    \end{sqremark}
    
    We can easily extend Theorem~\ref{thm:sig_ito_formula} to time-dependent linear functionals, i.e.~$\bell: [0, T] \to \eTA$. For this, let us define
    \begin{equation} \label{def:Iprime}
        \I'(X) := \set{
        \bell: [0, T] \to \A :
        \begin{matrix}
            \text{for every } \word{v} \in V, \: \bell^\word{v} \in C^1([0, T],\R) \text{ and } \hfill
            \\
            \text{for every } \word{i} \in \alphabet, \, \word{j} \in \mathcal Q_\word{i}(X) \text{ and for all } t \in [0, T], \; \bell_t \proj{i}, \bell_t \proj{ji}, \dot{\bell_t} \in \A \text{ and } \hfill
            \\
            \int_0^T \left( \normA{\bell_t \proj{i}} \d \abs{F_t^\word{i}} + \normA{\bell_t \proj{ji}} \d [X^\word{j}, X^\word{i}]_t  + \bra{\normA{\bell_t \proj{i}}}^2 \d [X^\word{i}, X^\word{i}]_t \right) < \infty \; \text{ a.s.} 
        \end{matrix}
        },
    \end{equation}
    where $\dot{\bell_t} = \sum_{n=0}^\infty \sum_{\word{v} \in V_n} \frac{\d}{\d t} \bell_t^\word{v} \word{v}$.

    \begin{corollary} \label{coro:ito}
        Let $\bell \in \I'(X)$, then $(\bracketsigX{\bell_t})_{t \leq T}$ is an Itô process such that 
        \begin{equation}
            \bracketsigX{\bell_t}
            = \bell_t^{\emptyword}
            + \sum_{\word{i} \in \alphabet} \int_0^t \bracketsigX[s]{\bell_s \proj{i}} \d X_s^\word{i}
            + \half \sum_{\word{j} \in \alphabet} \sum_{\word{i} \in \alphabet} \int_0^t \bracketsigX[s]{\bell_s \proj{ji}} \d [X^\word{j}, X^\word{i}]_s
            + \int_0^t\bracketsigX[s]{\dot{\bell}_s} \d s,
        \end{equation}
        for all $t \leq T.$    
    \end{corollary}

    \begin{proof}
        The proof is similar to the proof of Theorem \ref{thm:sig_ito_formula}.
    \end{proof}
    
    \begin{sqremark}\label{rmk:itofrac}
        The example of the fractional Riemann-Liouville Brownian motion, treated in Section~\ref{subsec:gaussianvolterra}, shows that, contrary to the finite case, not all infinite linear combinations of signature elements of $\widehat{W}$ are semimartingales. In particular, Theorem~\ref{thm:sig_ito_formula} and Corollary~\ref{coro:ito} cannot be applied in such cases. 
        
    \end{sqremark}

\subsection{The space \texorpdfstring{$\Ah$}{A\_h} and moment estimates} \label{subsec:Sh}
    
    In practice, it is not an easy task to verify directly whether a given $\bell$ belongs to $\A$, even for the case of the time extended $(d-1)$-dimensional Brownian motion $X_t = \widehat{W}_t = (t, W_t^2, \dots, W_t^d)$. The aim of this subsection is to introduce a subset $\Ah$ of $\A$, specific to such $\widehat{W}$ with $d\geq 2$, by essentially controlling the expectation of its signature $\widehat{\mathbb{W}}$. We will give a tractable criterion to prove that a given $\bell$ belongs to $\Ah$ and thus to $\A$.

    In order to motivate the set $\Ah$, we will derive an upper bound for the moments of the signature elements of $\widehat{\mathbb{W}}$ with the help of Fawcett's formula \cite{fawcett} using a deterministic function $h$. For that and to ease the notations, for every $\wv \in V$ we let $n(\wv)$ denote length of $\wv$, i.e.~$\wv \in V_{n(\wv)}$, and $x(\wv)$ the number of occurrences of the letter $\word{1}$ in $\wv$. In the following, if it does not cause ambiguity, we will write $n$ and $x$ in place of $n(\wv)$ and $x(\wv)$ respectively for brevity. \\
    
    \begin{sqexample}
        Let $\wv = \word{111321}$, then $n = 6$ and $x = 4$.
    \end{sqexample}

    We first provide a tight bound for the second moment:

    \begin{proposition} \label{prop:uniform_bound_L2}
        Fix $t \in [0, T]$ and $\wv \in V$. We have
        \begin{enumerate}[label=\normalfont(\textbf{\roman*})]
            \item \label{prop:uniform_bound_L2:item1}
            \begin{equation} \label{eq:uniform_bound_L2}
                \E \abs{\bracketsig{\wv}}^2
                \leq \binom{2n}{n} \frac{t^{n+x}}{(n+x)!} 2^{x-n},
            \end{equation}
            
            \item \label{prop:uniform_bound_L2:item2}
            and there exists a non-negative constant $C \geq 2$ independent of $\wv$ and $t$, such that
            \begin{equation} \label{eq:uniform_bound_L2_sup}
                \E \sqbra{\sup_{s \in [0, t]} \abs{\bracketsig[s]{\wv}}^2} \leq h_t(\wv)^2,
            \end{equation}
            where
            $$ h_t(\wv) := \frac{C}{({n+1})^{\frac{1}{4}}}{\frac{(2t)^{\frac{1}{2}(n+x)}}{\sqrt{(n+x-1)!}}}, $$
            with the convention that $n! = 1$ for $n \leq 0$. In particular, $h$ satisfies the following sub-multiplicative property
            \begin{align}
                h_t(\wu \wv) \leq h_t(\wu) h_t(\wv), \quad  \wu, \wv \in V.\label{eq:prop:h-uv}
            \end{align}
        \end{enumerate}
    \end{proposition}        
            
    \begin{proof}
        The proof is given in Appendix \ref{app:gen_results}.
    \end{proof}
   
    Based on Proposition~\ref{prop:uniform_bound_L2}, an application of Jensen's inequality yields
    \begin{equation} \label{eq:L1_uniform_bound_of_sig}
        \E \sqbra{\sup_{s \in [0, t]} \abs{\bracketsig[s]{\wv}}} \leq h_t(\wv).
    \end{equation}
    
    We are now ready to introduce a norm on $\eTA$ using the function $h$: 
    \begin{align} \label{eq:normSh}
       \normh{\bell} := \sum_{n=0}^\infty \sum_{\wv \in V_n} \abs{\bell^\wv} h_t(\wv), \quad \bell \in \eTA.  
    \end{align}

    Unlike $\norm{\cdot}^\A$, $\norm{\cdot}^\Ah$ is deterministic, thus simplifying the computation. Moreover, notice that $\normh{\cdot}$ is an increasing function of $t$, which naturally leads to the definition of the normed vector space:
    $$ \Ah := \set{ \bell \in \eTA: \normh[T]{\bell} < \infty }. $$
    {\revtwo
    \begin{sqremark}
        We can see that $\Ah$ is a Banach space by simply constructing an isometry from $\Ah$ to \allowbreak $\set{(a_i)_{i=0}^\infty: \sum_{i=0}^\infty \abs{a_i} < \infty}$, and in particular it is a Köthe sequence space, see \cite{kothe_space, Ramos-Fernandez2017}.
    \end{sqremark}
    }
    \begin{proposition} \label{prop:Ah}
        Fix $\bell \in \Ah$. The following properties hold:
        \begin{enumerate}[label=\normalfont(\textbf{\roman*})]
            \item
            $\Ah \subset \A$, and
            \begin{align}
                \E \sqbra{\sup_{s \in [0, t]} \normA[s]{\bell}} &
                \leq \normh{\bell}, \quad t \leq T.
            \end{align}
            
            \item \label{lem:to_verify_stochastic_fubini}
            $\Ah$ induces a natural upper bound
            $$ \E \sqbra{\sup_{s \in[0, t]} \bracketsig[s]{\bell}^2} \leq \bra{\normh{\bell}}^2, \quad t \leq T. $$
            
            \item \label{P:continuity}
            $t \mapsto \bracketsig{\bell}$ is continuous and the strong error of convergence of the truncated linear combination of the signature elements satisfies
            \begin{align} \label{eq:boundconvergenceh} 
                \E \sqbra{\sup_{t \in [0, T]} \abs{\bracketsigtrunc[M]{\bell} - \bracketsig{\bell}}} 
                \leq \sum_{n=M}^\infty \sum_{\wv \in V_n} \abs{\bell^\wv} h_T(\wv), \quad M \in \N.
            \end{align}
        \end{enumerate}
    \end{proposition}

    \begin{proof}
        \begin{enumerate}[label=\normalfont(\textbf{\roman*})]
            \item
            For all $\bell \in \Ah$, the inequality  \eqref{eq:L1_uniform_bound_of_sig} yields 
            \begin{align}
                \E \sqbra{\sup_{s \in [0, t]} \normA[s]{\bell}} 
                \leq \sum_{n=0}^\infty \sum_{\wv \in V_n} \abs{\bell^\wv} \E \sqbra{\sup_{s \in [0, t]} \abs{\bracketsig[s]{\wv}}}
                \leq \sum_{n=0}^\infty \sum_{\wv \in V_n} \abs{\bell^\wv}h_t(\wv)
                \leq \normh{\bell} < \infty.
            \end{align}
            Consequently, $\sup_{t \in [0, T]} \normA{\bell}< \infty \text{ a.s.}$ and thus $\bell \in \A(\widehat{\mathbb{W}})$.
            
            \item
            Starting with
            \begin{align*}
               \bracketsig[s]{\bell}^2 &
                = \bra{\sum_{\wv \in V} \bell^\wv \bracketsig[s]{\wv}}^2
                \leq \sum_{\wu \in V} \sum_{\ww \in V} \abs{\bell^\wu \bell^\ww} \cdot \abs{\bracketsig[s]{\wu} \bracketsig[s]{\ww}},
            \end{align*}
            we get that 
            $$ \sup_{s \in [0, t]} \bracketsig[s]{\bell}^2 \leq  \sum_{\wu \in V} \sum_{\ww \in V} \abs{\bell^\wu \bell^\ww} \cdot \sup_{s \in [0, t]} \abs{\bracketsig[s]{\wu}} \cdot \sup_{s \in [0, t]} \abs{\bracketsig[s]{\ww}}.$$
            Taking the expectation and applying Cauchy-Schwarz inequality yields that  
            \begin{align*}
                \E \left[ \sup_{s \in [0, t]} \bracketsig[s]{\bell}^2 \right] &
                \leq \sum_{\wu \in V} \sum_{\ww \in V} \abs{\bell^\wu \bell^\ww} \sqrt{\E \sup_{s \in [0, t]} \bracketsig[s]{\wu}^2} \sqrt{\E\sup_{s \in [0, t]} \bracketsig[s]{\ww}^2}
                \\ &
                = \bra{\sum_{\wv \in V} \abs{\bell^\wv} \sqrt{\E \sup_{s\in[0,t]}\bracketsig[s]{\wv}^2}}^2
                \\ &
                \leq  \bra{\sum_{\wv \in V} \abs{\bell^\wv} h_t(\wv)}^2
                \\ &
                = \bra{\normh{\bell}}^2.
            \end{align*}
            \item
            Define
            \begin{align}
                X_t :
                = \bracketsig{\bell} = \sum_{n=0}^\infty \sum_{\wv \in V_n} \bell^\wv \sig^\wv
                \quad \text{and} \quad
                X_t^M :
                = \bracketsigtrunc[M]{\bell} = \sum_{n=0}^M \sum_{\wv \in V_n} \bell^\wv \sig^\wv.
            \end{align}
            It follows from inequality \eqref{eq:L1_uniform_bound_of_sig} that 
            \begin{align}
                \E \sqbra{\sup_{t \in [0, T]} \abs{X_t - X_t^M}} &
                \leq \sum_{n=M+1}^\infty \sum_{\wv \in V_n} \abs{\bell^\wv} \E \sqbra{\sup_{t \in [0, T]} \bracketsig{\wv}}
                \leq \sum_{n=M+1}^\infty \sum_{\wv \in V_n} \abs{\bell^\wv} h_T(\wv).
            \end{align}
            
            Moreover, $\bell \in \Ah$ implies that
            $$ \lim_{M \to \infty} \sum_{n=M+1}^\infty \sum_{\wv \in V_n} \abs{\bell^\wv} h_T(\wv) = 0. $$
            
            Thus
            \begin{equation}
                \E \sqbra{\sup_{t \in [0, T]} \abs{X_t - X_t^M}} \to 0,
            \end{equation}
            implying that
            $$ \sup_{t \in [0, T]} \abs{X_t - X_t^M} \to 0 \quad \text{in probability as $M \to \infty$.} $$
            
            So there exists a sub-sequence $(M_k)_{k \in \N}$ such that $\sup_{t \in [0, T]} \abs{X_t - X_t^{M_k}} \to 0$ a.s. Moreover, as $X^{M_k}$ is continuous, $X$ is also continuous.
        \end{enumerate}
    \end{proof}

    Derived from \eqref{eq:prop:h-uv}, the next proposition shows that $\norm{\cdot}^\Ah$ is sub-multiplicative for the tensor product.
    
    \begin{proposition} \label{prop:submultiplicativity_Sh}
        For every $\bell, \bphi \in \eTA$,
        $$ \normh{\bell \bphi} \leq \normh{\bell} \normh{\bphi}. $$
        In particular, if $\bell, \bphi \in \Ah$, then $\bell \bphi \in \Ah$.
    \end{proposition}
    
    \begin{proof}
        First notice that for all $k \in \N$ and $\ww \in V_k$, 
        $$ \bracket{\bell \bphi}{\ww} = \sum_{n=0}^k \sum_{\wu \in V_n} \sum_{\wv \in V_{k-n}} \bell^\wu \bphi^\wv \bracket{\wu \wv}{\ww}, $$
        so that
        \begin{align}
            \normh{\bell \bphi} &
            = \sum_{k=0}^\infty \sum_{\ww \in V_k} \abs{\bracket{\bell \bphi}{\ww}} h_t(\ww)
            \\ &
            = \sum_{k=0}^\infty \sum_{\ww \in V_k} \abs{\sum_{n=0}^k \sum_{\wu \in V_n} \sum_{\wv \in V_{k-n}} \bell^\wu \bphi^\wv \bracket{\wu \wv}{\ww}} h_t(\ww)
            \\ &
            \leq \sum_{k=0}^\infty \sum_{n=0}^k \sum_{\wu \in V_n} \sum_{\wv \in V_{k-n}} \sum_{\ww \in V_k} \abs{\bell^\wu \bphi^\wv} \bracket{\wu \wv}{\ww} h_t(\ww)
        \end{align}
        
        Then, using \eqref{eq:prop:h-uv}
        \begin{align}
            \normh{\bell \bphi} &
            \leq \sum_{k=0}^\infty \sum_{n=0}^k \sum_{\wu \in V_n} \sum_{\wv \in V_{k-n}} \abs{\bell^\wu \bphi^\wv} h_t(\wu) h_t(\wv) \sum_{\ww \in V_k} \bracket{\wu \wv}{\ww}
            \\ &
            = \sum_{m=0}^\infty \sum_{n=0}^\infty \sum_{\wu \in V_n} \sum_{\wv \in V_m} \abs{\bell^\wu \bphi^\wv} h_t(\wu) h_t(\wv)
            \\ &
            = \normh{\bell} \normh{\bphi}.
        \end{align}
    \end{proof}

        \begin{sqremark} \label{rmk:benarous}
        Compared to the literature, our estimate for the supremum norm \eqref{eq:uniform_bound_L2_sup} appears to be used for the first time and plays a crucial role in our convergence analysis. The closest related results appeared in \citet{arous1989flots} where moment estimates of $\E | \langle \wv, \sig \rangle |^2$ and certain convergence criterion for infinite series of signatures of $\widehat{W}$, i.e.~iterated sequence of Stratonovich integrals, are established. The estimate in \cite[Lemma 2]{arous1989flots} has the same order as \eqref{eq:uniform_bound_L2} given here, although the constant there is not as sharp as ours (since we make use of Fawcett's formula) but has the same order of magnitude (see the remark following \cite[Lemma 3]{arous1989flots}). We also provide estimates for $\E | \langle \wv, \sig \rangle |^p$ for every $p \in \N_+$, see proposition \ref{prop:uniform_boundary}. But most importantly, we provide an estimate for $\E [\sup_{s \in [0, t]} | \langle \wv, \sig[s] \rangle |^2]$ in \eqref{eq:uniform_bound_L2_sup}, which is necessary to prove convergence for all $t$ almost surely instead of simply for a fixed $t$, when compared to the convergence criterion in \cite[Theorem 2 (1)]{arous1989flots}.
        Additionally, the second criteria \cite[Theorem 2 (2)]{arous1989flots} requires the second order norm to be bounded, we only require the first order norm $\normh{\cdot}$ to be bounded, which is hence milder.
    \end{sqremark}

\subsection{Back to the linear algebraic equation}

    We now prove that the solution $\bell$ to the linear algebraic equation \eqref{eq:resolvent} belongs to $\Ah$ under suitable conditions on $\bp$ and $\bq$, see Theorem~\ref{thm:ell_in_Sh_if_q_dominated}. This is a crucial ingredient in obtaining convergence of our representations formula $\bracketsig{\bell}$ in Section~\ref{sec:representations}. For this, we start by introducing the following partial order:
   
    \begin{definition} \label{def:partial_order}
        Let $\bell, \bphi \in \eTA$. We say that $\bell$ is dominated by $\bphi$ if for all $\wv \in V$, $\abs{\bell^\wv} \leq \bphi^\wv$. We then write
        $$ \bell \dominated \bphi. $$
    \end{definition}

    \begin{theorem} \label{thm:ell_in_Sh_if_q_dominated}
        Let $\bp \in \Ah$, let $\bq \in \eTA$ be such that $\abs{\bq^\emptyword} < 1$ and assume there exists $\bphi \in \eTA$ such that
        \begin{align}
            \bq \dominated \bphi
            \quad \text{and} \quad
            \res{\bphi} \in \Ah.
        \end{align}
        Then, the resolvent $\res{\bq}$ is well-defined and belongs to $\Ah$. Furthermore, the unique solution $\bell \in \eTA$ to the linear equation
        $$ \bell = \bp + \bell \bq $$
        belongs to $\Ah$.
    \end{theorem}
    
    In order to prove Theorem~\ref{thm:ell_in_Sh_if_q_dominated}, we will first show that the resolvent is order-preserving. That is:
    \begin{lemma} \label{lem:resolvent_order_preserving}
        Let $\bq, \bphi \in \eTA$ such that $\abs{\bq^\emptyword}, \bphi^\emptyword < 1$ and $\bq \dominated \bphi$, then
        $$ \res{\bq} \dominated \res{\bphi}. $$
    \end{lemma}

    \begin{proof}
        We will first rewrite the equation of the resolvent \eqref{nota:resolvent} as
        \begin{align}
            \res{\bq} &
            = \sum_{k=0}^\infty \bq \conpow{k}
            = \sum_{k=0}^\infty \bra{\sum_{m=0}^\infty \sum_{\wv \in V_m} \bq^\wv \wv} \conpow{k}
            = \sum_{m=0}^\infty \sum_{\wv \in V_m} f_\wv (\bq^\wu, n(\wu) \leq m) \wv \label{eq:res_poly}
        \end{align}
        
        where $f_\wv$ is a polynomial function of degree $n$ for all $\wv \in V_n$, which is independent of $\bq$ and has $\frac{d^{n+1}-1}{d-1}$ variables $\set{\bq^\wu: \abs{\wu} \leq \abs{\wv}}$. We also notice that the coefficients of $f_{\wv}$ are all non-negative integers. (For example, $f_\word{11}(\bq^\wu, |\wu| \leq 2) = 2 \bq^\emptyword \bq^\word{11} + (\bq^\word{1})^2$ and $f_\word{12}(\bq^\wu, |\wu| \leq 2) = 2 \bq^\emptyword \bq^\word{12} + 2 \bq^\word{1} \bq^\word{2}$.) 
        Now recall the Definition~\ref{def:partial_order} where $\bq \dominated \bphi$ implies $\abs{\bq^\wv} \leq \bphi^\wv$ for all $\wv \in V$. Combined this with the fact that $f$ has non-negative coefficients, it is easy to see that, for all $\wv \in V_m$ we have
        $$ \abs{f_\wv (\bq^\wu, n(\wu) \leq m)} \leq f_\wv (\bphi^\wu, n(\wu) \leq m), $$
        which ends the proof.
    \end{proof}
    
    We can now prove Theorem~\ref{thm:ell_in_Sh_if_q_dominated}.
    \begin{proof}[Proof of Theorem~\ref{thm:ell_in_Sh_if_q_dominated}]
        By assumption on $\bq$ and by Lemma~\ref{lem:resolvent_order_preserving}, $\res{\bq}$ is well-defined and belongs to $\Ah$. Moreover, by Proposition~\ref{prop:resolvent-solution}
        \begin{align}
            \bell = \bp \inverse{\emptyword - \bq}.
        \end{align}
        Finally, by assumption on $\bp$ and Proposition~\ref{prop:submultiplicativity_Sh}, it is clear that $\bell \in \Ah$.
    \end{proof}

    {\revtwo
\subsection{The space \texorpdfstring{$\Aexp$}{A\_exp} and exponentially dominated functionals} \label{subsec:Aexp}
    
    This final subsection is dedicated to the use of shuffle exponentials to dominate linear functionals. Shuffle exponentials allow for straightforward computations and it is easy to show their inclusion in $\Ah$, see Lemma~\ref{lem:shuexp_in_Sh}. Moreover, they allow for the definition of the space $\Aexp \subset \Ah$ of linear functionals dominated by a shuffle exponential. This space holds interesting properties, among them, $\Aexp$ is closed under the projection operation, allowing for the inclusion of $\Aexp$ in $\I(\widehat{W})$, see Corollary~\ref{coro:pq_Aexp_implies_ell_Aexp}, making it much easier to prove that a linear functional is an Itô process. \\
    }
    
    First, we prove that $\shuexp{\bm{b}} \in \Ah$:
    \begin{lemma} \label{lem:shuexp_in_Sh}
        Let $\bm{b} = \sum_{\word{i} \in \alphabet} \bm{b}^\word{i} \word{i}$ with $\bm{b}^\word{i} \in \R$, then
        $$ \shuexp{\bm{b}} = \res{\bm{b}} \in \Ah. $$
    \end{lemma}    

    \begin{proof}
      First, the equality is ensured by Proposition~\ref{prop:resolvent-shuexp}, so it suffices to prove $\res{\bm{b}}\in\Ah$. For simplicity, we provide a proof for the case where $d = 2$, the extension to the general case follows along the same lines. For a word $\wv \in V$, we recall the notations $n(\wv)$ and $x(\wv)$ that denote the length of $\wv$ and the number of $\word{1}$ in $\wv$ respectively. We also recall that $h_T(\wv)$ depends on $\wv$ only through $n(\wv)$ and $x(\wv)$, meaning that for any $\wu \in V$, $n(\wv) = n(\word{u})$ and $x(\wv) = x(\wu)$ imply $h_T(\wv) = h_T(\wu)$. Therefore, assuming $\bm{b} = a \word{1} + b \word{2}$ and $\bell = \res{(a \word{1} + b \word{2})}$, we can decompose $\normh[T]{\bell}$ as:
        \begin{align}
            \normh[T]{\bell} &
            = \sum_{n=0}^\infty \sum_{\wv \in V_n} \abs{\bell^\wv} h_T(\wv)
            \\ &
            = \sum_{n=0}^\infty \sum_{k=0}^n \binom{n}{k} \abs{a}^k \abs{b}^{n-k} h_T(\word{1} \conpow{k} \word{2} \conpow{n-k})
            \\ &
            = \sum_{m=0}^\infty \sum_{k=0}^\infty \binom{k+m}{k} \abs{a}^k \abs{b}^m h_T(\word{1} \conpow{k} \word{2} \conpow{m})
            \\ &
            = \sum_{m=0}^\infty \sum_{k=0}^\infty \frac{(k+m)!}{k!~m!} \abs{a}^k \abs{b}^m \frac{C}{({k+m+1})^{\frac{1}{4}}}{\frac{(2T)^{k + \frac{m}{2}}}{\sqrt{(2k+m-1)!}}}
            \\ &
            = C \sum_{m=0}^\infty \sum_{k=0}^\infty \frac{\abs{2T a}^k}{k!} \frac{\abs{\sqrt{2T} b}^m}{\sqrt{m!}} ({k+m+1})^{-\frac{1}{4}} \sqrt{\frac{(k+m)!^2}{m!(2k+m-1)!}}.
        \end{align}
        After remarking that $({k+m+1})^{-\frac{1}{4}} \leq 1$ for all $k, m \in \N$, we can shift our attention to the square root on the right-hand side:
        \begin{align}
            \frac{(k+m)!^2}{m! (2k+m-1)!} &
            = \frac{\prod_{j=1}^k (m+j)}{\prod_{j=1}^{k-1} (k+m+j)}
            \leq k + m +1
            \leq (k+1) (m+1).
        \end{align}
        This gives
        \begin{align}
            \normh[T]{\bell} &
            \leq C \sum_{m=0}^\infty \bra{\sqrt{\frac{m+1}{m!}} \abs{\sqrt{2T} b}^m} \sum_{k=0}^\infty \bra{\frac{\sqrt{k+1}}{k!} \abs{2T a}^k}.
        \end{align}
        Now, thanks to Stirling's formula,
        $$ \frac{\sqrt{k+1}}{k!} \sim \bra{\frac{e}{k}}^k \quad \text{and} \quad \sqrt{\frac{m+1}{m!}} \sim m^{\frac{1}{4}} \bra{\frac{e}{m}}^{\frac{m}{2}} $$
        so it is clear that both $\sum_{k \geq 0} \frac{\sqrt{k+1}}{k!} x^k$ and $\sum_{m \geq 0} \sqrt{\frac{m+1}{m!}} y^m$ have an infinite radius of convergence, hence proving
        $$ \normh[T]{\bell} < \infty. $$
    \end{proof}

    {\revtwo
    Naturally, we will define the set $\Aexp$ of all linear functionals that are \textit{exponentially dominated}:
    \begin{align}
        \Aexp := \set{\bell: \text{there exists $C \in \R_+$ such that } \bell \dominated C \shuexp{C \sum_{\wi \in \alphabet} \wi}}.
    \end{align}

    In dimension 1, $\Aexp$ represents the power series that have a factorial decay. Obviously $\Aexp \subset \Ah$ thanks to Lemma~\ref{lem:shuexp_in_Sh}. 

    \begin{proposition} \label{prop:Aexp_closed}
        $\Aexp$ is closed under linear operations, shuffle product, concatenation, resolvent and projection. That is, let $\bell, \bphi \in \Aexp$, then
        \begin{multicols}{2}
            \begin{enumerate}[label=\normalfont(\textbf{\roman*})]
                \item \label{prop:Aexp_closed_linear} $\alpha \bell + \beta \bphi \in \Aexp$ for $\alpha, \beta \in \R$,
                \item \label{prop:Aexp_closed_proj} $\bell \proj{u} \in \Aexp$ for all $\wu \in V$,
            \end{enumerate}
        \end{multicols}
        \begin{multicols}{2}
            \begin{enumerate}[label=\normalfont(\textbf{\roman*})]
                \setcounter{enumi}{2}
                \item \label{prop:Aexp_closed_shuffle} $\bell \shuprod \bphi \in \Aexp$,
                \item \label{prop:Aexp_closed_concat} $\bell \bphi \in \Aexp$,
            \end{enumerate}
        \end{multicols}
        \begin{enumerate}[label=\normalfont(\textbf{\roman*})]
            \setcounter{enumi}{4}
            \item \label{prop:Aexp_closed_resolvent} If in addition $\bell^\emptyword \neq 1$, $\res{\bell} \in \Aexp$.
        \end{enumerate}
    \end{proposition}

    \begin{proof}
        To simplify notations, let us define $\bm{b} := \sum_{\wi \in \alphabet} \wi$. Moreover, let $\alpha, \beta \in \R$ and $C_{\bell}, C_{\bphi} > 1$, and let $\bell, \bphi \in \Aexp$ be such that $\bell \dominated C_{\bell} \shuexp{C_{\bell} \bm{b}}$ and $\bphi \dominated C_{\bphi} \shuexp{C_{\bphi} \bm{b}}$.
        \begin{itemize}
            \item[\ref{prop:Aexp_closed_linear}] Clearly, one has $\alpha \bell + \beta \bphi \dominated (|\alpha| C_{\bell} + |\beta| C_{\bphi}) \shuexp{(C_{\bell} \vee C_{\bphi}) \bm{b}}$,
            
            \item[\ref{prop:Aexp_closed_proj}] Using Proposition~\ref{prop:resolvent-shuexp}, one also has $\bell \proj{u} \dominated (C_{\bell})^{n+1} \shuexp{C_{\bell} \bm{b}}$ for $\wu \in V_n$ and $n \in \N$,
            
            \item[\ref{prop:Aexp_closed_shuffle}] Using \cite[Lemma 2.10]{cuchiero2023signature} on the shuffle product of shuffle exponentials, one can deduce
            $$ \bell \shuprod \bphi \dominated C_{\bell} \shuexp{C_{\bell} \bm{b}} \shuprod C_{\bphi} \shuexp{C_{\bphi} \bm{b}} \dominated (C_{\bell} C_{\bphi}) \shuexp{(C_{\bell} + C_{\bphi}) \bm{b}}, $$
            
            \item[\ref{prop:Aexp_closed_concat}] Seeing that the first (and last) terms in the shuffle product are simple concatenations of the two functionals, one can conveniently dominate a concatenation by a shuffle product, i.e.
            $$ \bell \bphi \dominated \abs{\bell} \abs{\bphi} \dominated \abs{\bell} \shuprod \abs{\bphi} \dominated (C_{\bell} C_{\bphi}) \shuexp{(C_{\bell} + C_{\bphi}) \bm{b}} ,$$
            where $\abs{\bell} := \sum_{\word{v} \in V} \abs{\bell^\word{v}}\wv$,
            \item[\ref{prop:Aexp_closed_resolvent}] We can first assume $\bell^\emptyword = 0$, as $\res{\bell} = \frac{1}{1 - \bell^{\emptyword}} \res{\frac{1}{1 - \bell^{\emptyword}} \bra{\bell - \bell^{\emptyword} \emptyword}}$, and $C_{\bell}, C_{\bphi} > 1$ without loss of generality. It implies $\bell \dominated C_{\bell} (\shuexp{C_{\bell} \bm{b}} - \emptyword) = (C_{\bell})^2 \bm{b} \shuexp{C_{\bell} \bm{b}}$. Moreover, by Lemma~\ref{lem:resolvent_order_preserving}, 
            $$ \res{\bell} \dominated \res{(C_{\bell})^2 \bm{b} \shuexp{C_{\bell} \bm{b}}}. $$
            There thus only remains to make sure that, for any $\alpha, \beta \in \R$, $\res{\alpha \bm{b} \shuexp{\beta \bm{b}}} \in \Aexp$. For this purpose, recalling that $\res{\beta \bm{b}} (\emptyword - \beta \bm{b}) = \emptyword$, we have that
            \begin{align}
                \emptyword - (\alpha + \beta) \bm{b} &
                = \emptyword - \alpha \bm{b} \res{\beta \bm{b}} \bra{\emptyword - \beta \bm{b}} - \beta \bm{b}
                \\ &
                = \bra{\emptyword - \alpha \bm{b} \res{\beta \bm{b}}} \bra{\emptyword - \beta \bm{b}},
            \end{align}
            which directly implies from Proposition~\ref{prop:resolvent-shuexp} that
            \begin{align}
                \inverse{\emptyword - \alpha \bm{b} \shuexp{\beta \bm{b}}} &
                = \bra{\emptyword - \beta \bm{b}} \inverse{\emptyword - (\alpha + \beta) \bm{b}} = \bra{\emptyword - \beta \bm{b}} \shuexp{(\alpha + \beta) \bm{b}}.
            \end{align}
            Finally applying \ref{prop:Aexp_closed_concat}, we clearly see that the \textit{rhs} is dominated, ending the proof.
        \end{itemize}
    \end{proof}

    \begin{corollary} \label{coro:pq_Aexp_implies_ell_Aexp}
        Let $\bp, \bq \in \Aexp$, $\bq^{\emptyword}\neq 1$, then the unique solution $\bell \in \eTA$ to the linear equation
        $$ \bell = \bp + \bell \bq $$
        belongs to $\Aexp$. Moreover, it also belongs to $\I(\widehat{W})$.
    \end{corollary}

    \begin{proof}
        First, applying Proposition~\ref{prop:Aexp_closed}\,\ref{prop:Aexp_closed_concat}-\ref{prop:Aexp_closed_resolvent}, we obtain $\bell = \bp \res{\bq} \in \Aexp$. Then, using Proposition~\ref{prop:Ah}\,\ref{lem:to_verify_stochastic_fubini}, together with Proposition~\ref{prop:Aexp_closed}\,\ref{prop:Aexp_closed_proj}, it follows that $\Aexp \subset \I(\widehat{W})$.
    \end{proof}
    
    Corollary~\ref{coro:pq_Aexp_implies_ell_Aexp} essentially implies that, provided $\bp$ and $\bq$ are exponentially dominated, an easily verifiable condition in practice, 
    $$ X_t := \bracketsig{\bp \res{\bq}} $$
    is well-defined and is an Itô process as in~\eqref{eq:ItohatW}.  
    }

\section{Signature representations of path-dependent processes} \label{sec:representations}

    In this section, we provide our main signature representation formulae for linear path-dependent equations. We consider three categories:
    \begin{itemize}
        \item Linear Volterra equations: exact, time-independent representations for smooth enough kernels, and an approximation result for more general kernels, in Section~\ref{sec:volterra},
        
        \item Certain linear delay equations: exact, time-independent representations for a weighted sum of exponential kernels, in Section~\ref{sec:delayed},
        
        \item Gaussian Volterra processes: exact, time-dependent representations for possibly singular kernels, including the Riemann-Liouville fractional Brownian motion for a Hurst index $H \in (0, 1)$, in Section~\ref{subsec:gaussianvolterra}
    \end{itemize}
    
    In the sequel, we will consider the $2$-dimensional semimartingale $\widehat{W}_t = (t, W_t)$, where $W$ is a one-dimensional Brownian motion, we will discuss extensions for a $d$-dimensional Brownian motion in Remarks~\ref{rmk:generalize_to_multi_d} and  \ref{rmk:generalize_to_multi_d2}.

\subsection{Linear Volterra equations} \label{sec:volterra}
    
    The first class we consider is linear Volterra equations of the form
    \begin{equation} \label{eq:volterra}
        Y_t = y + \int_0^t K_1(t-s) \left( a_1 + b_1 Y_s \right) \d s + \int_0^t K_2(t-s) \left( a_2 + b_2 Y_s \right) \d W_s,
    \end{equation}
    for some real coefficients $y, a_1, b_1, a_2, b_2$ and locally square-integrable kernels $K_1, K_2$. Since the drift and volatility coefficients are linear in $Y$ and hence Lipschitz continuous, the stochastic integral equation \eqref{eq:volterra} admits a unique strong solution $Y$, see for instance \cite[Theorem 3.3]{jaber2019affine}. \\
    
    In our next theorem, we provide an explicit infinite linear representation for the solution $Y_t = \bracketsig{\bell}$ in terms of the signature elements of $\widehat{W}$ and time-independent coefficients $\bell$. For this, we will need the following structure on the kernels $K_1$ and $K_2$:
    \begin{align} \label{eq:assKlinear}
        K_1(u) = \int_{[0, \infty)} e^{-xu} \mu_1(\d x)
        \quad \text{ and } \quad
        K_2(u) = \int_{[0, \infty)} e^{-xu} \mu_2(\d x),
    \end{align}
    for finite measures $\mu_1$ and $\mu_2$ and such that 
    \begin{equation} \label{eq:measure_condition}
        \int_{[0, \infty)} x^n \mu_1(\d x) + \int_{[0, \infty)} x^n \mu_2(\d x) < M^n, \quad n \in \N,
    \end{equation}
    for some constant $M > 0$. \\
    
    In particular, this implies that $K_1, K_2$ are infinitely continuously differentiable on $[0, T]$. For such smooth coefficients $K_1, K_2$ the solution $Y$ is even a semimartingale. \\
    
    \begin{sqexample}\label{ex:sumofexp}
        Let us consider $\mu(\d x) = f(x) \d x$ for a positive bounded measurable function $f$ with compact support. It is easy to see that for such $\mu$, the assumptions \eqref{eq:assKlinear}-\eqref{eq:measure_condition}  are satisfied. Another example is the case $\mu(\d x) = \sum_{k=1}^n w_k^n \delta_{x_k^n} (\d x)$ and $K(u) = \sum_{k=1}^n w_k^n e^{-x_k^n u}$ for $w_k^n, x_k^n \in \R$.
    \end{sqexample}

    Before we state the theorem, we need to clarify how one can make sense of $\int_{[0, \infty)} \bell(x) \mu(\d x)$ when $\bell: x \mapsto \bell(x)$ is a function taking values in $\eTA$.

    \begin{definition}[Integral of parameterized linear functionals] \label{def:integral_sig}

        Fix a Borel measure $\mu$ on $[0, \infty)$. We define $L^1(\eTA[d], \mu)$ to be the set of all the weakly integrable maps from $\bra{[0, \infty),\mu}$ to the locally convex topological vector space $\eTA$, i.e.
        $$ L^1(\eTA[d], \mu) := \set{\bell: [0, \infty)\longrightarrow \eTA \text{ s.t. } \int_{[0, \infty)} \abs{\bell^\wv(x)} \mu(\d x) < \infty, \text{ for all } \wv \in V }. $$
        
        Furthermore, we also define 
        $$ L^1(\A, \mu) := \set{\bell \in L^1(\eTA[d], \mu):\int_{[0, \infty)} \bell(x) \mu(\d x) \in \A}. $$
        
    \end{definition}
    
    \begin{sqexample}
        Take $\bell(x) = \shuexp{-x \word{1}}$ and $\mu(\d x) = e^{-\frac{x^2}{2}} \d x$, then $\bell \in L^1(\eTA, \mu)$ and thus $\int_{[0, \infty)} \bell(x) \mu(\d x) $ is a well-defined element of $\eTA$. Furthermore we can check that $\int_{[0, \infty)} \bell(x) \mu(\d x) \in \A$, showing that $\bell \in L^1(\A, \mu)$. Now fixing $\bell(x) = \shuexp{-x \word{1}}$, for a general measure $\mu$, we can easily see that $\bell\in L^1(\A, \mu)$ if and only if $\sum_{n=0}^{\infty} \int_{[0,\infty)} \frac{x^n}{n!} \mu(\d x) < \infty$. This partially explains why the assumption \eqref{eq:measure_condition} arises.
    \end{sqexample}

    \begin{theorem} \label{thm:volterra}
        Fix $K_1, K_2$ satisfying \eqref{eq:assKlinear}-\eqref{eq:measure_condition}. The solution $Y$ to the linear Volterra equation \eqref{eq:volterra} admits the time-independent signature representation
        \begin{align} \label{eq:lvol}
            Y_t = \bracketsig{\lvol} \text{ for every $t \leq T$, a.s., with } \lvol = \pvol \inverse{\emptyword - \qvol}
        \end{align}
        and
        \begin{align}
            \pvol :&
            = a_1 \word{1} \int_{[0, \infty)} \shuexp{-x \word{1}} \mu_1(\d x) + a_2 \bra{\word{2} - \thalf b_2 K_2(0) \word{1}} \int_{[0, \infty)} \shuexp{-x \word{1}} \mu_2(\d x) + y \emptyword, \label{eq:pvol}
            \\ \qvol :&
            = b_1 \word{1} \int_{[0, \infty)} \shuexp{-x \word{1}} \mu_1({\d x}) + b_2 \bra{\word{2} - \thalf b_2 K_2(0) \word{1}} \int_{[0, \infty)} \shuexp{-x \word{1}} \mu_2(\d x), \label{eq:qvol}
        \end{align}
        and are such that $\pvol, \qvol, \lvol \in \Aexp$. In particular, they all belong to $\A$ and $\I(\widehat W)$.
    \end{theorem}

    \begin{proof}
        The proof is given in Section \ref{app:volterra}.
    \end{proof}
    
    For $b_2 = 0$ and $K_1 = K_2 \equiv 1$, one recovers from the representation in Theorem~\ref{thm:volterra} a time-independent representation for the Ornstein-Uhlenbeck process as shown in the following example. For a time-dependent representation of the same process, we refer to Example~\ref{ex:OU} below. 

    \begin{sqexample}[Ornstein-Uhlenbeck]\label{ex:OUtimeindep}
        Set $b_2 = 0$, and $\mu_1(\d x) = \mu_2(\d x) = \delta_0(\d x)$, where $\delta_0$ is the Dirac mass at $0$. In this case, $K_1(u) = K_2(u) = 1$ and equation \eqref{eq:volterra} reads
        $$ \d Y_t = (a_1 + b_1 Y_t) \d t + a_2 \d W_t, \quad Y_0 = y \in \R. $$
        meaning that $Y$ is an  Ornstein-Uhlenbeck process. Moreover, in this case, $\int_{[0, \infty)} \shuexp{-x \word{1}} \mu_1(\d x) = \\ \int_{[0, \infty)} \shuexp{-x \word{1}} \mu_2(\d x) = \emptyword$. So that 
        $$ \qvol = b_1 \word{1}, \quad \pvol = y \emptyword + a_1 \word{1} + a_2 \word{2}. $$
        An application of Theorem~\ref{thm:volterra} yields that the Ornstein-Uhlenbeck process $Y$ can be written as 
        $$ Y_t = \bracketsig{\bell^{\textnormal{OU}}} $$
        with    
        \begin{align} \label{eq:repOU1}
           \bell^{\textnormal{OU}} &
            = \bra{y \emptyword + a_1 \word{1} + a_2 \word{2}} \inverse{\emptyword - b_1 \word{1}}
            = \bra{y \emptyword + a_1 \word{1} + a_2 \word{2}} \shuexp{b_1 \word{1}},
        \end{align}
        where the second equality follows from  Proposition~\ref{prop:resolvent-shuexp}.
      \end{sqexample}

    Furthermore, for $a_1 = a_2 = 0$ and $K_1 = K_2 \equiv 1$, one recovers from the representation in Theorem~\ref{thm:volterra} the explicit solution of the geometric Brownian motion:
    
    \begin{sqexample}[Geometric Brownian motion]
        Take $a_1 = a_2 = 0$, and $\mu_1(\d x) = \mu_2(\d x) = \delta_0(\d x)$, where $\delta_0$ is the Dirac mass at $0$. In this case, $K_1(u) = K_2(u) = 1$ and $\int_{[0, \infty)} \shuexp{-x \word{1}} \mu_1(\d x) = \int_{[0, \infty)} \shuexp{-x \word{1}} \mu_2(\d x) = \emptyword$. \\
        
        Therefore $\qvol = \bra{b_1 - \thalf b_2^2} \word{1} + b_2 \word{2}$, $\pvol = y \emptyword$ and thus, using Proposition~\ref{prop:resolvent-shuexp},

        \begin{align}
            \lvol &
            = y \inverse{\emptyword - \bra{b_1 - \thalf b_2^2} \word{1} - b_2 \word{2}}
            = y \shuexp{\bra{b_1 - \half b_2^2} \word{1} + b_2 \word{2}}.
        \end{align}
        
        It follows from the shuffle property \ref{prop:shuffle_property} and the definition of the shuffle exponential in \eqref{nota:shuexp}, that
        \begin{align}
            Y_t &
            = y \bracketsig{\shuexp{\bra{b_1 - \half b_2^2} \word{1} + b_2 \word{2}}}
            = y e^{\bracketsig{\bra{b_1 - \thalf b_2^2} \word{1} + b_2 \word{2}}}
            = y e^{\bra{b_1 - \thalf b_2^2} t + b_2 W_t},
        \end{align}
        
        which is the geometric Brownian motion, solution to 
        $$ \frac{\d Y_t}{Y_t} = b_1 \d t + b_2 \d W_t, \quad Y_0 = y \in \R. $$
    \end{sqexample}

    Upon initial inspection, our assumptions regarding $K_1$ and $K_2$ may appear limiting. However, the stability results of stochastic Volterra equations allow us to approximate solutions for broader classes of kernels. For example, we can approximate solutions for the singular fractional kernel $K(t) = t^{H-1/2}$ with $H \in (0, 1/2)$ using linear combinations of the signature elements, as demonstrated in the following corollary.

    \begin{corollary} \label{corr:convergence-K^n}
        Let $K_1, K_2$ be locally square-integrable kernels and denote by $Y$ the solution of \eqref{eq:volterra}. For $n \in \N$ and $i=1,2$, let $\mu_i^n (\d x) = \sum_{k=1}^n w_k^{n,i} \delta_{x_k^{n,i}}(\d x)$ for some $w_k^{n,i}, x_k^{n,i} \in \R$, and set $K_1^n$ and $K_2^n$ as in \eqref{eq:assKlinear} and $Y_t^n = \bracketsig{\lvol_n}$ as in \eqref{eq:lvol} both with $\mu_1^n$ and $\mu_2^n$ instead of $\mu_1$ and $\mu_2$. Assume that 
        \begin{align} \label{eq:convekrnel}
            \int_0^T \abs{K_1^n(s) - K_1(s)}^2 \d s + \int_0^T \abs{K_2^n(s) - K_2(s)}^2 \d s \to 0, \quad \text{ as } n \to \infty.
        \end{align}
        
        Then, 
        \begin{align}
           \sup_{t \in [0, T]} \E \sqbra{\abs{Y_t - \bracketsig{\lvol_n}}^p} \to 0, \quad \text{ as } n \to \infty, \quad p \in \N.
        \end{align}
    \end{corollary}
    
    \begin{proof}
        For $n \in \N$, it follows from Example~\ref{ex:sumofexp} that $\mu_1^n$ and $\mu_2^n$ satisfy the assumptions \eqref{eq:measure_condition}. Hence, an application of Theorem~\ref{thm:volterra} yields that $Y^n_t = \bracketsig{\lvol_n}$ solves the stochastic Volterra equation 
        \begin{equation} \label{eq:volterra_n}
            Y^n_t = y + \int_0^t K^n_1(t-s) \bra{a_1 + b_1 Y_s} \d s + \int_0^t K^n_2(t-s) \bra{a_2 + b_2 Y_s} \d W_s.
        \end{equation}
        
        Routine applications of Jensen's and Grownall's convolution inequalities, see for instance the proof of \cite[Lemma 2.4]{jaber2019affine} and \cite[Lemma 9.8.2]{gripenberg1990volterra}, yield that 
        \begin{align}
            \sup_{t \in [0, T]} \E \sqbra{\abs{Y_t -  Y^n_t}^p} \leq C_p \bra{\bra{\int_0^T \abs{K_1^n(s) - K_1(s)}^2 \d s}^\frac{p}{2} + \bra{\int_0^T \abs{K_2^n(s) - K_2(s)}^2 \d s}^\frac{p}{2}}, \quad p \in \N,
        \end{align}
        for some $C_p >0$, which ends the proof. 
    \end{proof}

    \begin{sqremark}\label{rmk:approx}
        Different choices exist in the literature on Volterra processes for the parameters $w_i^k, x_i^k$ to ensure the convergence of the weighted sum of exponentials towards the initial kernels in \eqref{eq:convekrnel} whenever $K_1$ and $K_2$ are completely monotone functions on $(0, \infty)$, i.e.~infinitely differentiable on $(0, \infty)$ such that $(-1)^n K^{(n)} \geq 0$. For instance,  for a Hurst index $H \in (0, 1/2)$, the singular fractional kernel $t^{H - 1/2}$ is completely monotone and admits the following representation (following Berstein's theorem) as a Laplace transform: 
        \begin{equation} \label{laplace}
            \frac{t^{H-\half}}{\Gamma(H+\half)} = \int_{[0, \infty)} e^{-xt} \mu(\d x)
            \quad \text{ and } \quad
            \mu(dx) = \frac{x^{-H-\half}}{\Gamma(H+\half) \Gamma(\half - H)} \d x.
        \end{equation}
        For $n \geq 1$ and $r_n > 1$, such that 
        \begin{align} \label{eq:rn_cond}
            r_n \downarrow 1
            \quad \text{ and } \quad
            n \ln r_n \to \infty,
            \quad \text{ as }
            n \to \infty,
        \end{align} 
        with the following parametrization for the weights and the mean reversions
        \begin{align} \label{eq:ci_and_xi}
            w_k^n = \frac{(r_n^{1-\alpha}-1)r_n^{(\alpha-1)(1+\frac{n}{2})}}{\Gamma(\alpha) \Gamma(2-\alpha)} r_n^{(1-\alpha)k}
            \quad \text{ and } \quad
            x_k^n = \frac{1-\alpha}{2-\alpha} \frac{r_n^{2-\alpha}-1}{r_n^{1-\alpha}-1} r_n^{k-1-\frac{n}{2}},
            \quad \text{ for }
            k = 1, \ldots, n,
        \end{align}
        where $\alpha := H+1/2$ for some $H \in (0, 1/2)$, one can obtain the convergence of $\sum_{k=1}^n c_k^n e^{-x_k^n t}$ towards the fractional kernel in the sense of \eqref{eq:convekrnel}, see \cite{abi2019lifting,jaber2018multifactor} for more details.
    \end{sqremark}

    \begin{sqremark} \label{rmk:generalize_to_multi_d}
        Theorem \ref{thm:volterra} can be readily extended to the case where the Volterra equation is driven by a $(d-1)$-dimensional Brownian motion $W$ of the form
        \begin{equation}
            Y_t = y + \int_0^t K_\word{1}(t-s) \bra{a_\word{1} + b_\word{1} Y_s} \d s + \sum_{\word{j} \in \set{\word{2}, \dots, \word{d}}} \int_0^t K_\word{j}(t-s) \bra{a_\word{j} + b_\word{j} Y_s} \d W_s^\word{j}.
        \end{equation}
        Here each $K_\word{i}$ is of the form $K_\word{i}(u) = \int_{[0, \infty)} e^{-xu} \mu_\word{i}(\d x)$ with a finite measure $\mu_\word{i}$ satisfying \eqref{eq:measure_condition} with the same constant $M$, for $\word{i} \in \set{\word{1}, \dots, \word{d}}$. Then, similar to Theorem~\ref{thm:volterra}, we obtain the representation
        \begin{align} \label{eq:lvol_multi_d}
            Y_t = \bracketsig{\lvol}, \quad \lvol = \pvol \inverse{\emptyword - \qvol},
        \end{align}
        while $\pvol$ and $\qvol$ are updated to:
        \begin{align}
            \pvol :&
            = \sum_{\word{i} \in \set{\word{1}, \dots, \word{d}}} a_\word{i} \word{i} \int_{[0, \infty)} \shuexp{-x \word{1}} \mu_\word{i}(\d x) - \half \sum_{\word{j} \in \set{\word{2}, \dots, \word{d}}} a_\word{j} b_\word{j} K_\word{j}(0) \word{1} \int_{[0, \infty)} \shuexp{-x \word{1}} \mu_\word{j}(\d x) + y \emptyword,
            \\ \qvol :&
            = \sum_{\word{i} \in \set{\word{1}, \dots, \word{d}}} b_\word{i} \word{i} \int_{[0, \infty)} \shuexp{-x \word{1}} \mu_\word{i}(\d x) - \half \sum_{\word{j} \in \set{\word{2}, \dots, \word{d}}} b_\word{j}^2 K_\word{j}(0) \word{1} \int_{[0, \infty)} \shuexp{-x \word{1}} \mu_\word{j}(\d x).
        \end{align}
    \end{sqremark}

\subsection{Linear delay equations} \label{sec:delayed}

    A second equation that we consider is a delay equation with linear delays both in the drift and in the volatility parts. To keep notations simple and highlight the main ideas, we restrict to weighted sum of exponential kernels:
    \begin{align} \label{eq:sde-DE-exp}
        d Z_t &
        = \left( a_1 + b_1 Z_t + \int_0^t K_1(t-s) Z_s \d s \right) \d t + \left( a_2 + b_2 Z_t + \int_0^t K_2(t-s) Z_s \d s \right) \d W_t, 
        \\ Z_0 &
        = z \in \R,
    \end{align}
    where 
    $$ K_i(t) = \sum_{k=1}^{n_i} w_i^k e^{x_i^k t}, $$
    with $a_i, b_i, w_i^k, x_i^k \in \R, n_i \in \N$ for $i = 1, 2$. Again, it is straightforward to obtain the existence and uniqueness of a strong solution $Z$ since the coefficients are linear. We stress that $Z$ cannot be written as a Volterra equation in the sense of \eqref{eq:volterra} unless $K_2$ is zero.
    
    \begin{theorem} \label{thm:delayed}
        The solution $Z$ to the linear delay equation \eqref{eq:sde-DE-exp} admits the time-independent signature representation
        \begin{align} \label{eq:lde}
            Z_t = \bracketsig{\lde}
            \text{ for every $t \leq T$, a.s., with }
            \lde = \pde \inverse{\emptyword - \qde},
        \end{align}
        and
        \begin{align}
            \pde &
            := \bra{a_1 - \thalf b_2 a_2} \word{1} + a_2 \word{2} + z \emptyword, \label{eq:pde}
            \\ \qde &
            := \bra{b_1 - \thalf b_2^2} \word{1} + b_2 \word{2} + \bra{\sum_{k=1}^{n_1} w_1^k \word{1} \shuexp{x_1^k \word{1}} - \half b_2 \sum_{k=1}^{n_2} w_2^k \word{1} \shuexp{x_2^k \word{1}}} \word{1} + \bra{\sum_{k=1}^{n_2} w_2^k \word{1} \shuexp{x_2^k \word{1}}} \word{2}, \label{eq:qde}
        \end{align}
        and are such that $\pde, \qde, \lde \in \Aexp$. In particular, they all belong to $\A$ and $\I(\widehat W)$.
    \end{theorem}
    
    \begin{proof}
        The proof is given in Section \ref{app:delayed}.
    \end{proof}

    \begin{sqremark}
        For more general kernels $K_1, K_2$ a similar stability argument as in Corollary~\ref{corr:convergence-K^n} allows to approximate the solution to the delay equation by linear combination of the signatures.
    \end{sqremark}
    
    \begin{sqremark} \label{rmk:generalize_to_multi_d2}
        Similar to Remark \ref{rmk:generalize_to_multi_d}, Theorem~\ref{thm:delayed} can also be extended to the case where the Delayed equation is driven by a ($d-1$)-dimensional Brownian motion $W$ of the form:
        \begin{align} 
            d Z_t &
            = \left( a_\word{1} + b_\word{1} Z_t + \int_0^t K_\word{1}(t-s) Z_s \d s \right) \d t + \sum_{\word{j} \in \set{\word{2}, \dots, \word{d}}} \left( a_\word{j} + b_\word{j} Z_t + \int_0^t K_\word{j}(t-s) Z_s \d s \right) \d W_t^\word{j}, 
            \\ Z_0 &
            = z \in \R,
        \end{align}
        where 
        $$ K_{\word{i}}(t) = \sum_{k=1}^{n_\word{i}} w_\word{i}^k e^{x_\word{i}^k t}. $$
        Then similar to Theorem \ref{thm:delayed}, we have the representation:
        \begin{align}
            Z_t = \bracketsig{\lde},
            \quad
            \lde = \pde \inverse{\emptyword - \qde},
        \end{align}
        while $\pde$ and $\qde$ are updated to:
        \begin{align}
            \pde &
            = \sum_{\word{i} \in \set{\word{1}, \dots, \word{d}}} a_\word{i} \word{i} - \half \sum_{\word{j} \in \set{\word{2}, \dots, \word{d}}} a_\word{j} b_\word{j} \word{1} + z \emptyword, 
            \\ \qde &
            = \sum_{\word{i} \in \set{\word{1}, \dots, \word{d}}} b_\word{i} \word{i} - 
            \half \sum_{\word{j} \in \set{\word{2}, \dots, \word{d}}} b_\word{j}^2 \word{1} + 
            \sum_{\word{i} \in \set{\word{1}, \dots, \word{d}}} \sum_{k=1}^{n_\word{i}} w_\word{i}^k \word{1} \shuexp{x_\word{i}^k \word{1}} \word{i} 
            - \half \sum_{\word{j} \in \set{\word{2}, \dots, \word{d}}} b_\word{j} \sum_{k=1}^{n_\word{j}} w_\word{j}^k \word{1} \shuexp{x_\word{j}^k \word{1}} \word{1}.
        \end{align}
    \end{sqremark}

\subsection{Gaussian Volterra processes} \label{subsec:gaussianvolterra}

    The final class that we consider are continuous Gaussian Volterra processes of the form 
    \begin{align} \label{eq:gaussianvolterra}
        X_t = \int_0^t K(t-s) \d W_s,
    \end{align}
    for a locally square-integrable deterministic kernel $K$ that is infinitely differentiable on $(0, T]$ and satisfies
    \begin{align}
        \begin{cases} \label{eq:assKgaussian2}
            & \sum_{n=0}^\infty \abs{K^{(n)}(t)} \frac{t^{n + \half}}{n! \sqrt{2n + 1}} < \infty
            \quad \text{ and } \quad
            \int_0^t \left( \sum_{n=0}^\infty \abs{K^{(n)}(t)} \frac{s^n}{n!} \right)^2 \d s < \infty,
            \quad t \in (0, T],\\
            &\int_0^T \big( \sum_{n=0}^\infty \abs{K^{(n)}(t)} \frac{t^{n + \half}}{n! \sqrt{2n + 1}} \big) \d t < \infty,
        \end{cases}
    \end{align}
    where $K^{(n)}$ denotes the $n$-th derivative on $(0, T]$.  \\
    
    We note that here, and in contrast with Theorem~\ref{thm:volterra}, $K$ can have a singularity at $0$, as shown in Example~\ref{ex:RL} for the Riemann-Liouville fractional Brownian motion. For such kernels, Corollary~\ref{corr:convergence-K^n} already provides an approximation result by infinite linear combinations of signature elements of $\widehat{W}$ with time-independent coefficients $\bell$. More interestingly, the following theorem shows that exact infinite linear representation of $X$ are possible but with time-dependent coefficients $\bell_t$. Time-dependent representations seemed crucial for us to obtain exact representations for singular kernels.
    
    \begin{theorem} \label{thm:gaussian}
        Fix $K: (0, T] \to \R$ satisfying \eqref{eq:assKgaussian2}, the Gaussian Volterra process $X$ defined in \eqref{eq:gaussianvolterra} admits a time-dependent signature representation
        $$ X_t = \bracketsig{\lgv}, \quad a.s. $$
        for every $t \leq T$, with $\lgv$ given by 
        \begin{align} \label{eq:lgv}
            \lgv =
            \begin{cases}
                \sum_{n=0}^\infty K^{(n)}(t) (-1)^n \word{1} \conpow{n} \word{2} &
                \text{if } t \in (0, T],
                \\ 0 &
                \text{else.}
            \end{cases}
        \end{align}
    \end{theorem}
    
    \begin{proof}
        The proof is given in Section \ref{app:gaussian}.
    \end{proof}
    \begin{sqremark}
       Unlike the representations in Theorems~\ref{thm:volterra} and \ref{thm:delayed}, here $\bracketsig{\lgv}$ is a modification of $X_t$ and it is not necessarily indistinguishable from the process $X$.
    \end{sqremark}
    The following example shows that the exponential kernel $K(t) = e^{-\kappa t}$ for $\kappa > 0$ satisfies the assumption of Theorem~\ref{thm:gaussian}. This provides a time-dependent representation of the Ornstein-Uhlenbeck process, recall Example~\ref{ex:OUtimeindep} for the time-independent representation. 

    \begin{sqexample}[Ornstein-Uhlenbeck process] \label{ex:OU}
        Let
        $$ X_t = \int_0^t e^{-\kappa (t-s)} \d W_s = \int_0^t K(t-s) \d W_s $$
        with $K(t) = e^{-\kappa t}$ and $\kappa \in \R$. We first check that $K$ satisfies \eqref{eq:assKgaussian2}. Clearly $K$ is infinitely differentiable on $[0, T]$ such that $\abs{K^{(n)}(t)} = \abs{\kappa}^n K(t)$. Therefore,
        \begin{align}
            \int_0^t \left( \sum_{n=0}^\infty \frac{\abs{K^{(n)}(t)} s^n}{n!} \right)^2 \d s 
            = \int_0^t \left( K(t) \sum_{n=0}^\infty \frac{(\abs{\kappa} s)^n}{n!} \right)^2 \d s
            = K(t)^2 \int_0^t e^{2 \abs{\kappa} s} \d s < \infty.
        \end{align}
        Moreover
        \begin{align*}
            \sum_{n=0}^\infty \frac{\abs{K^{(n)}(t)}}{n! \sqrt{2n+1}} t^{n + \half} &= \sqrt{t} e^{-\kappa t} \sum_{n=0}^\infty \frac{(\abs{\kappa} t)^n}{n!} \frac{1}{\sqrt{2n+1}} \leq \sqrt{T} e^{\abs{\kappa} T}\sum_{n=0}^\infty \frac{(\abs{\kappa} t)^n}{n!}= \sqrt{T} e^{2\abs{\kappa}T} < \infty.
        \end{align*}
        Thus
        $$\int_0^T \bra{\sum_{n=0}^\infty \frac{\abs{K^{(n)}(t)}}{n! \sqrt{2n+1}} t^{n + \half}} \d t \leq T^{\frac{3}{2}} e^{2 \abs{\kappa} T} < \infty. $$
        
        Finally, Theorem~\ref{thm:gaussian} yields the signature representation $X_t = \bracketsig{\bell_t^\textnormal{OU}}$, with
        \begin{equation} \label{eq:linear-OU-timedep}
            \bell_t^\textnormal{OU} = e^{-\kappa t} \sum_{n=0}^\infty \kappa^n \word{1} \conpow{n} \word{2}.
        \end{equation}
        Note that we can more compactly write it as
        \begin{align}\label{eq:repOU2}
           \bell_t^\textnormal{OU} = \shuexp{-\kappa (t \emptyword - \word{1})} \word{2}, 
        \end{align}
        which is the algebraic translation of the probabilistic expression 
        $$ X_t = \int_0^t e^{-\kappa (t-s)} \d W_s. $$
    \end{sqexample}

    The following example shows that the fractional kernel $K(t) = t^{H-1/2}$ for $H \in (0, 1)$, with singularity at $0$ for $H<1/2$, satisfies the assumption of Theorem~\ref{thm:gaussian}. This provides a representation of the Riemann-Liouville fractional Brownian motion, and shows that infinite linear combinations of signature elements are not always semimartingales, recall Remark~\ref{rmk:itofrac}.
    
    \begin{sqexample}[Riemann-Liouville fractional Brownian motion] \label{ex:RL}
        Let
        $$ X_t = \frac{1}{\Gamma \left( H + \half \right)} \int_0^t (t-s)^{H - \half} \d W_s = \int_0^t K(t-s) \d W_s, $$
        
        with $K(t) = \frac{t^{H - \half}}{\Gamma(H + \half)}$. We first check that $K$ satisfies \eqref{eq:assKgaussian2}. Clearly $K$ is analytical on $(0, T]$ for $H \in (0, 1)$ :
        \begin{align}
            K^{(n)}(t) &
            = K(t) t^{-n} \left( H - \thalf \right)^{\underline{n}}
            = K(t) (-t)^{-n} \left( \thalf - H \right)^{\bar{n}}
        \end{align}
        where $x^{\underline{n}} = \prod_{k=0}^{n-1} (x-k)$ and $x^{\bar{n}} = \prod_{k=0}^{n-1} (x+k)$ denote the falling and rising factorials respectively. \\
        
        Furthermore, one can then remark that
        \begin{align}
            \abs{K^{(n)}(t)} &
            = \begin{cases}
                K^{(n)}(t) (-1)^{n} &
                \text{if $H < \half$,}
                \\ K^{(n)}(t) (-1)^{n+1} &
                \text{else.}
            \end{cases}
        \end{align}
        
        Therefore
        \begin{align}
            \int_0^t \left( \sum_{n=0}^\infty \frac{\abs{K^{(n)}(t)} s^n}{n!} \right)^2 \d s &
            = \int_0^t \left( K(t) \sum_{n=0}^\infty \left( H - \thalf \right)^{\underline{n}} \frac{1}{n!} \left( \frac{-s}{t} \right)^n \right)^2 \d s
            = \int_0^t K(t-s)^2 \d s < \infty,
        \end{align}
        
        where the last equality comes from the Taylor expansion of $(1 - x)^{H - 1/2}$ which converges for $\abs{x} < 1$. Moreover we have
        \begin{align}
            \abs{K^{(n)}(t)} &
            = K(t) \frac{\Gamma \left( n + \thalf - H \right)}{\Gamma \left( \thalf - H \right)} t^{-n} = O \left( \left( \tfrac{n}{t} \right)^{\half - H} (n-1)! ~t^{-n} \right), \quad n > 0,
        \end{align}
        
        therefore
        $$ \sum_{n=0}^\infty \frac{\abs{K^{(n)}(t)}}{n! \sqrt{2n+1}} t^{n + \half} \leq C t^{H} \sum_{n=1}^\infty n^{-1-H} < \infty. $$
        And $$
         \int_0^T \bra{\sum_{n=0}^\infty \frac{\abs{K^{(n)}(t)}}{n! \sqrt{2n+1}} t^{n + \half} } \d t \leq \frac{C T^{H+1}}{H+1} \sum_{n=1}^\infty n^{-1-H} < \infty. 
        $$
        
        Finally, Theorem~\ref{thm:gaussian} yields the signature representation $X_t = \bracketsig{\bell_t^\textnormal{RL}}$, with
        \begin{align} \label{eq:linear-RL}
            \bell_t^\textnormal{RL} &
            = \frac{t^{H - \half}}{\Gamma(H + \half)} \sum_{n=0}^\infty t^{-n} \left( \thalf - H \right)^{\bar{n}} \word{1} \conpow{n} \word{2},
        \end{align}
    
        where $(\cdot)^{\bar{n}}$ is the rising factorial. 
        Remarking that $\left( \thalf - H \right)^{\bar{n}} = 0$ for all $n > 0$ when $H = \half$ gives $\bell_t = \word{2}$, the coordinate of $W_t$ in $\sig$. Note that we can also formally write it as
        $$ \bell_t^\textnormal{RL} = \frac{1}{\Gamma(H + \half)} (t \emptyword - \word{1}) \shupow{H - \half} \word{2}, $$
        which is the algebraic translation of the probabilistic expression 
        $$ X_t = \frac{1}{\Gamma(H + \half)} \int_0^t (t-s)^{H - \half} \d W_s. $$
    \end{sqexample}

\subsection{Numerical illustrations} \label{subsec:numerics}

    In this subsection, we provide numerical implementations of the truncated form of the signature processes $X^{\leq M}$ defined by
    $$ X_t^{\leq M} := \bracketsigtrunc[M]{\bell}, \quad M \geq 0, $$
    where $\bell \in \eTA[2]$ and $\sig^{\leq M}: [0, T] \to \tTA[2]{M}$ is defined in \eqref{def:sig-trunc}, i.e.~its $M$ first levels coincide with $\sig$ and everything else is set to $0$.

\subsubsection{Shifted Riemann-Liouville fractional Brownian motion}
    
    Figure \ref{fig:ex_traj-RL} and Table~\ref{tab:mse_traj-RL} display simulations of a shifted Riemann-Liouville fractional Brownian motion $X^\varepsilon$ together with mean-squared errors, i.e.
    $$ X_t^\varepsilon = \frac{1}{\Gamma \left( H + \half \right)} \int_0^t (t+\varepsilon-s)^{H-\half} \d W_s, $$
    and its linear representation counterpart $\bell_{\cdot + \varepsilon}^{\textnormal{RL}}$, following from \eqref{eq:linear-RL},  that is
    \begin{align}
        \bell_{t + \varepsilon}^\textnormal{RL} &
        = \frac{(t + \varepsilon)^{H - \half}}{\Gamma \left( H + \half \right)} \sum_{n=0}^\infty (t + \varepsilon)^{-n} \left( \thalf - H \right)^{\bar{n}} \word{1} \conpow{n} \word{2}
        = \frac{1}{\Gamma \left( H + \half \right)} \bra{(t + \varepsilon) \emptyword - \word{1}} \shupow{H - \half} \word{2},
    \end{align}
    for several truncation orders $M \in \N$. The reason for shifting the kernel is to speed up the convergence as it can be quite slow for small values of $H$ when $\varepsilon=0$.
    
    \begin{figure}[H]
        \centering 
        \subfloat[\centering $H=0.1$]{\includegraphics[width=\twoplotswidth]{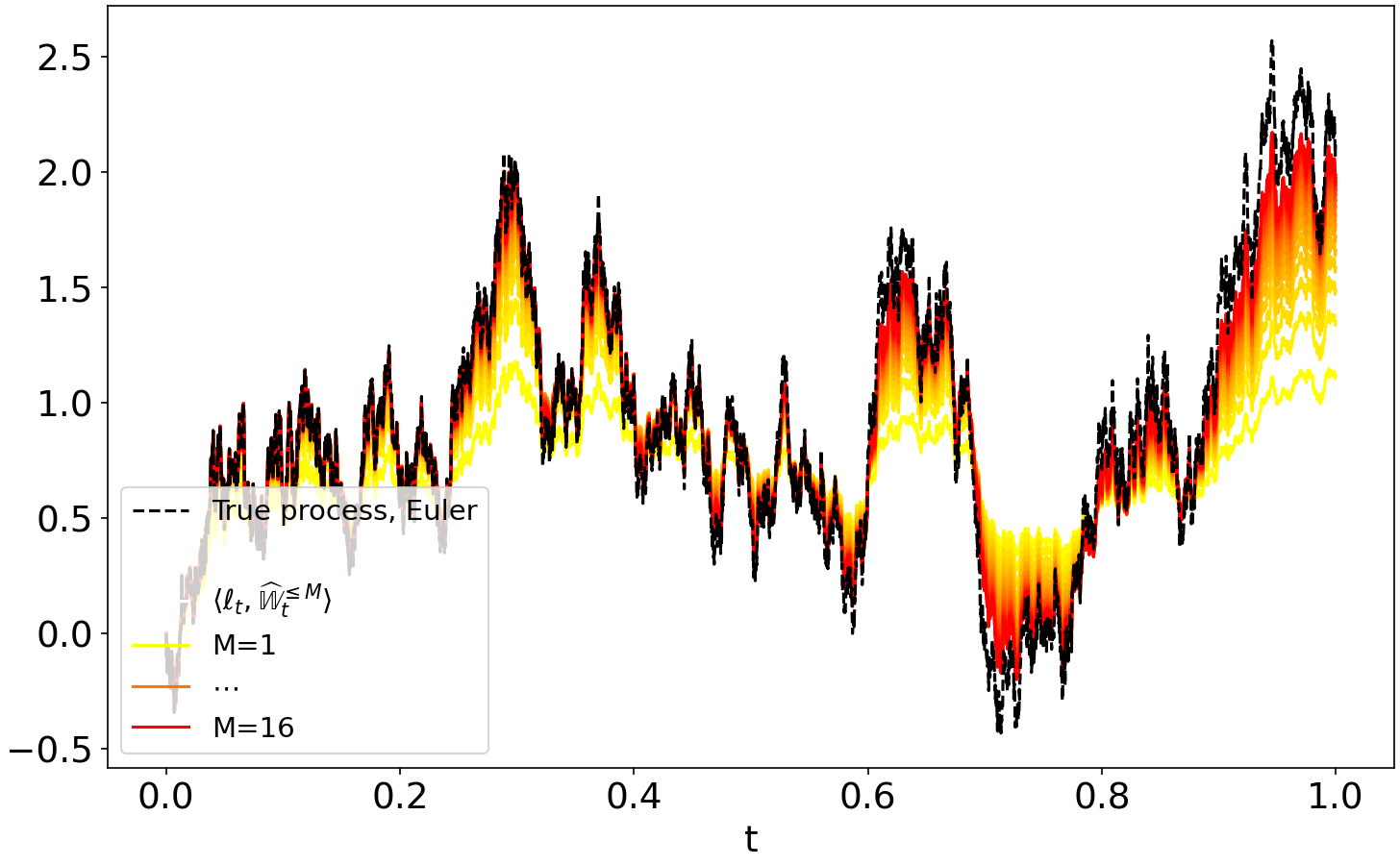}}
        \quad
        \subfloat[\centering $H=0.3$]{\includegraphics[width=\twoplotswidth]{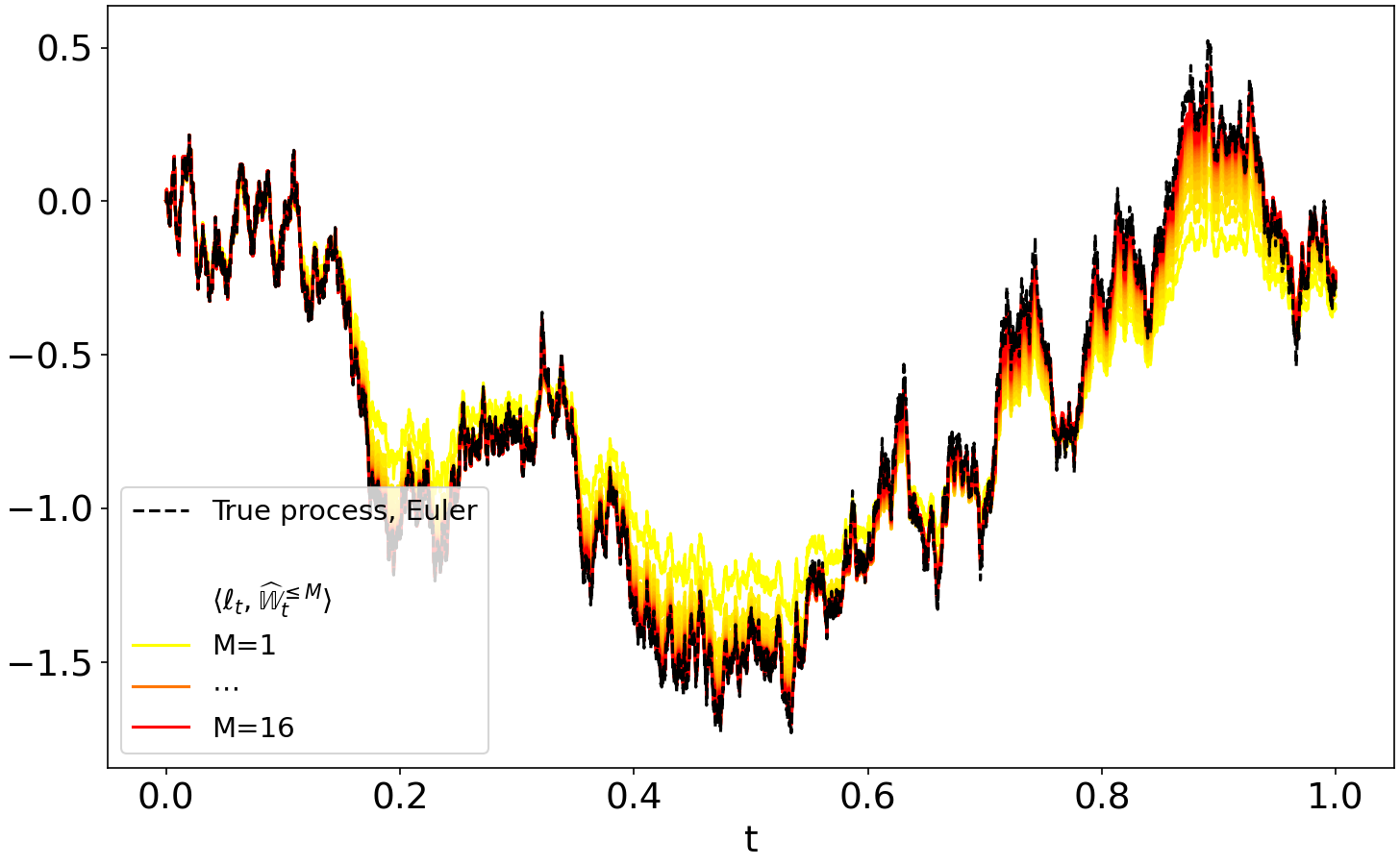}}
        \qquad
        \subfloat[\centering $H=0.7$]{\includegraphics[width=\twoplotswidth]{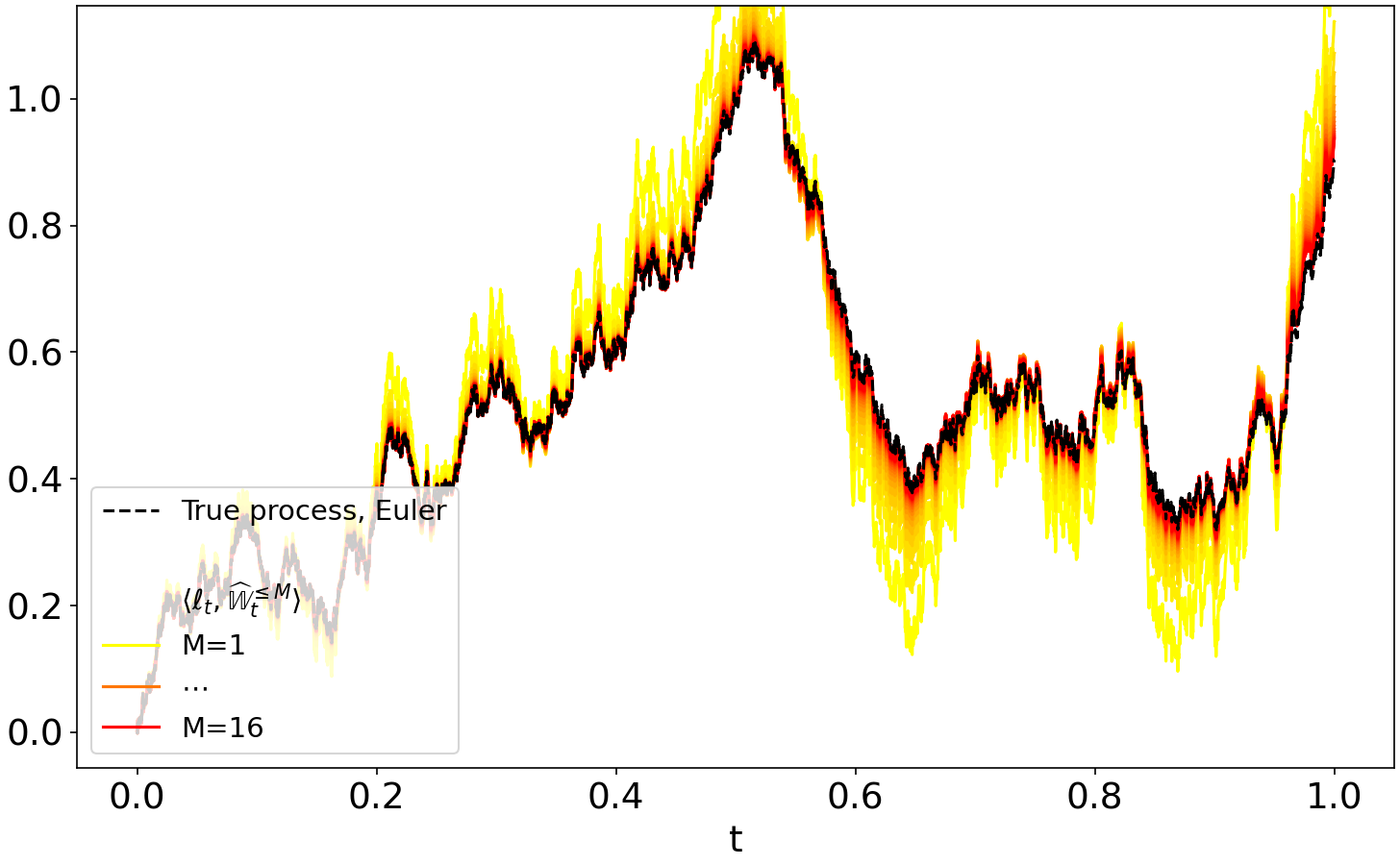}}
        \quad
        \subfloat[\centering $H=0.9$]{\includegraphics[width=\twoplotswidth]{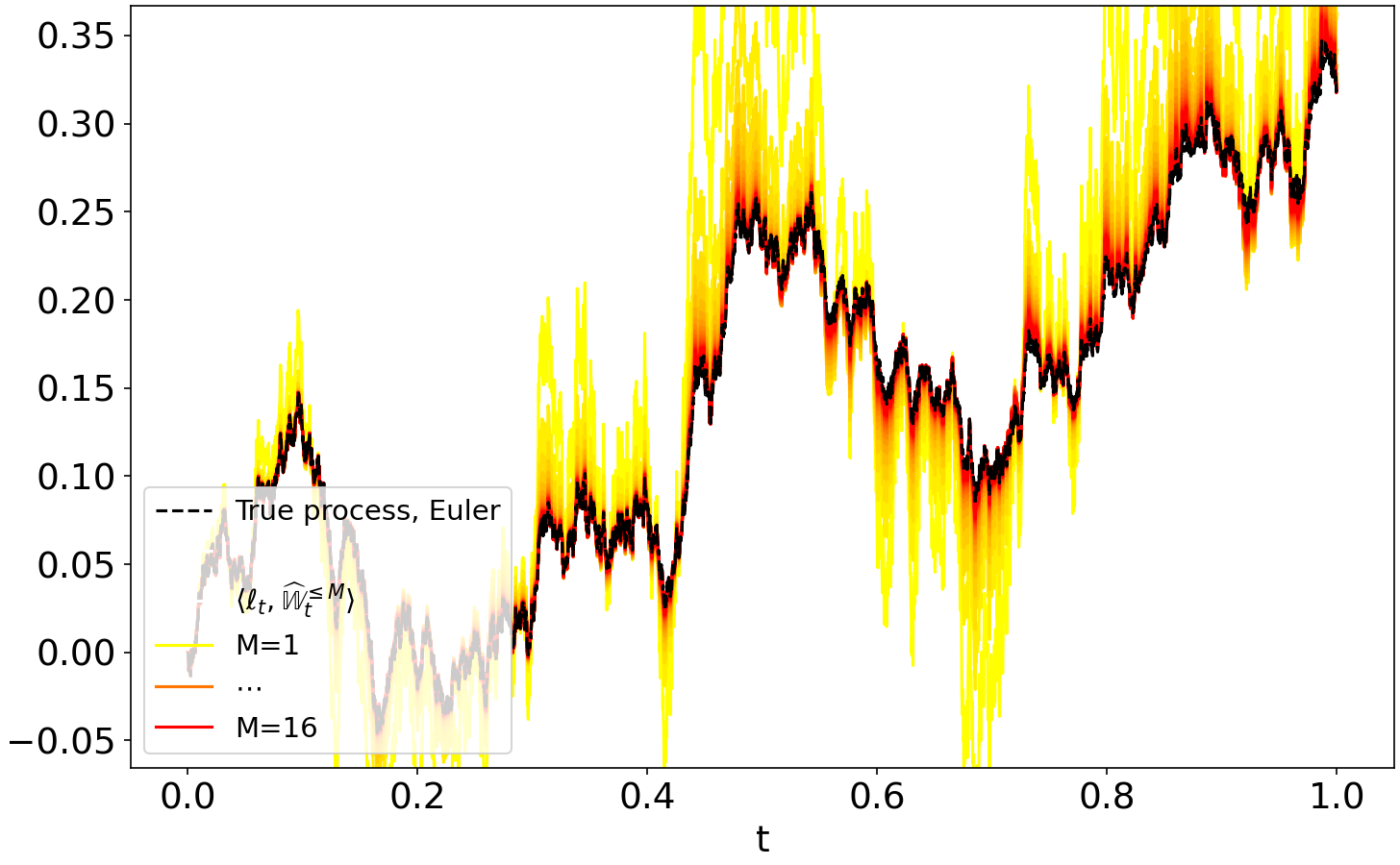}}
        \caption{Trajectories of a shifted Riemann-Liouville fractional Brownian motion (black) against their truncated time-dependent linear representation \protect \eqref{eq:linear-RL}, i.e. $\bracketsigtrunc[M]{\bell_{t + \varepsilon}^\textnormal{RL}}$, for several truncation orders $M$ and $\varepsilon=1/52$.}
        \label{fig:ex_traj-RL}
    \end{figure}

    \begin{table}[H]
        \centering
        \begin{tabular}{|l|cccc|}
            \hline
                    & H=0.1     & H=0.3     & H=0.7     & H=0.9     \\
            \hline
             M=2    & 9.915e-02 & 1.600e-02 & 4.015e-03 & 6.217e-03 \\
             M=4    & 4.735e-02 & 6.197e-03 & 9.022e-04 & 9.931e-04 \\
             M=8    & 1.851e-02 & 1.992e-03 & 1.819e-04 & 1.528e-04 \\
             M=16   & 5.298e-03 & 4.712e-04 & 2.793e-05 & 1.848e-05 \\
            \hline
        \end{tabular}
        \caption{Mean squared error between the shifted Riemann-Liouville fractional Brownian motion and its truncated time-dependent linear representation \protect \eqref{eq:linear-RL}, i.e.~$\bracketsigtrunc[M]{\bell_{t + \varepsilon}^\textnormal{RL}}$, for several truncation orders $M$ and $\varepsilon=1/52$, averaged across 100,000 simulations over 1000 time-steps.}
        \label{tab:mse_traj-RL}
    \end{table}
    
    We can see in Table~\ref{fig:ex_traj-RL} and Figure~\ref{tab:mse_traj-RL} the clear convergence of the truncated linear representation for several values of $H$. The smaller $H$ is, the slower the convergence. We can also see that the fit deteriorates with the time variable $t$.

\subsubsection{Linear delay process}

    Figure \ref{fig:ex_traj-DE} and Table~\ref{tab:mse_traj-DE} display simulations of a delay equation and its truncated representations together with mean-squared errors following Theorem~\ref{thm:delayed}. This convergence is shown for two sets of parameters: (a) has its $b_1=0$ and $x_i>0$ whereas (b) has $b_2=0$ and $x_i<0$.
    
    \begin{figure}[H]
        \centering
        \subfloat[\centering $z=0.25, a_1=0.25, b_1=0, n_1=1, w_1=-4, x_1=1, a_2=1, b_2=1, n_2=1, w_2=-2, x_2=1.5$]{\includegraphics[width=\twoplotswidth]{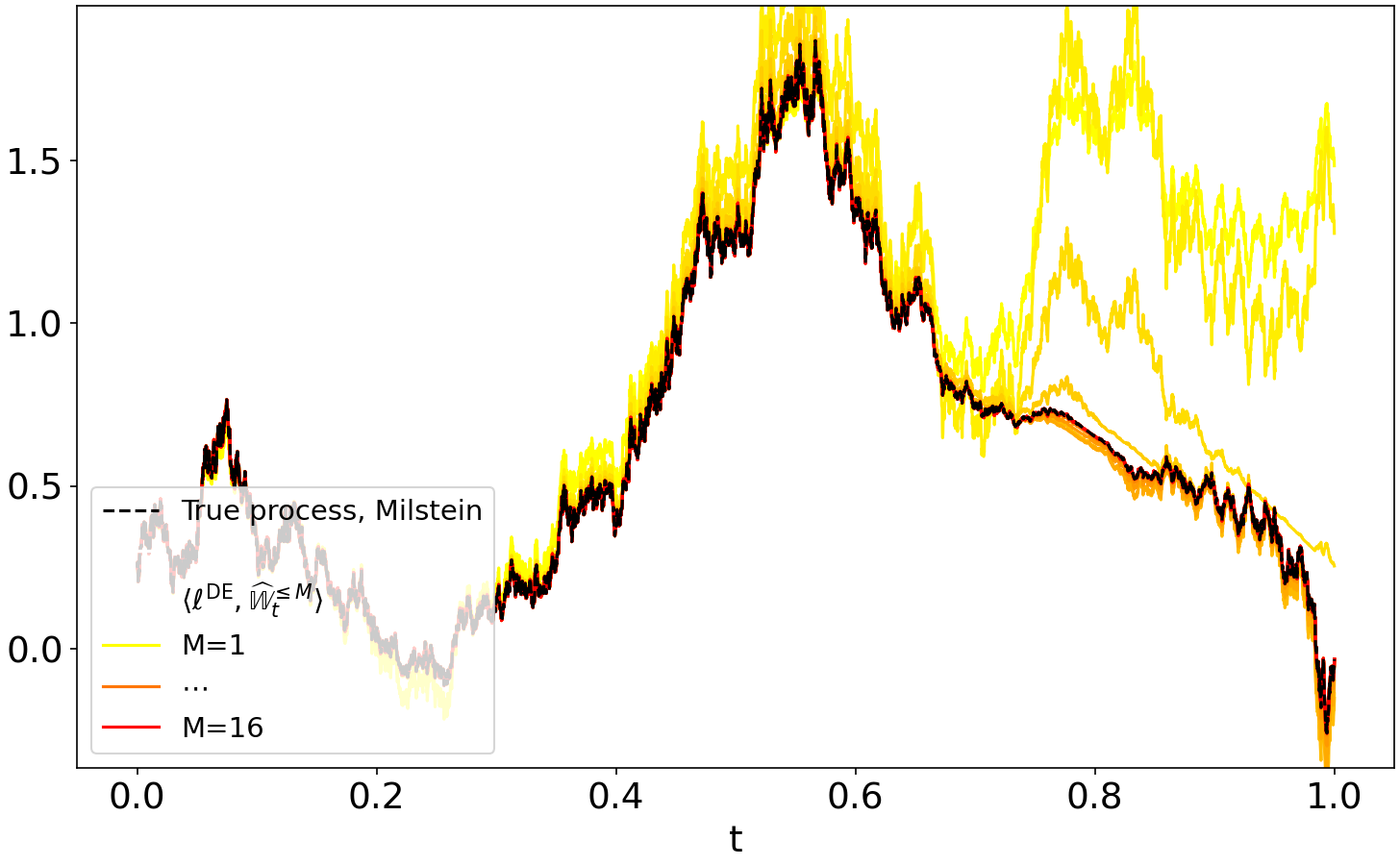}} 
        \quad
        \subfloat[\centering $z=0.25, a_1=0.25, b_1=-4, n_1=1, w_1=-4, x_1=-4, a_2=1, b_2=0, n_2=1, w_2=2, x_2=-2$]{\includegraphics[width=\twoplotswidth]{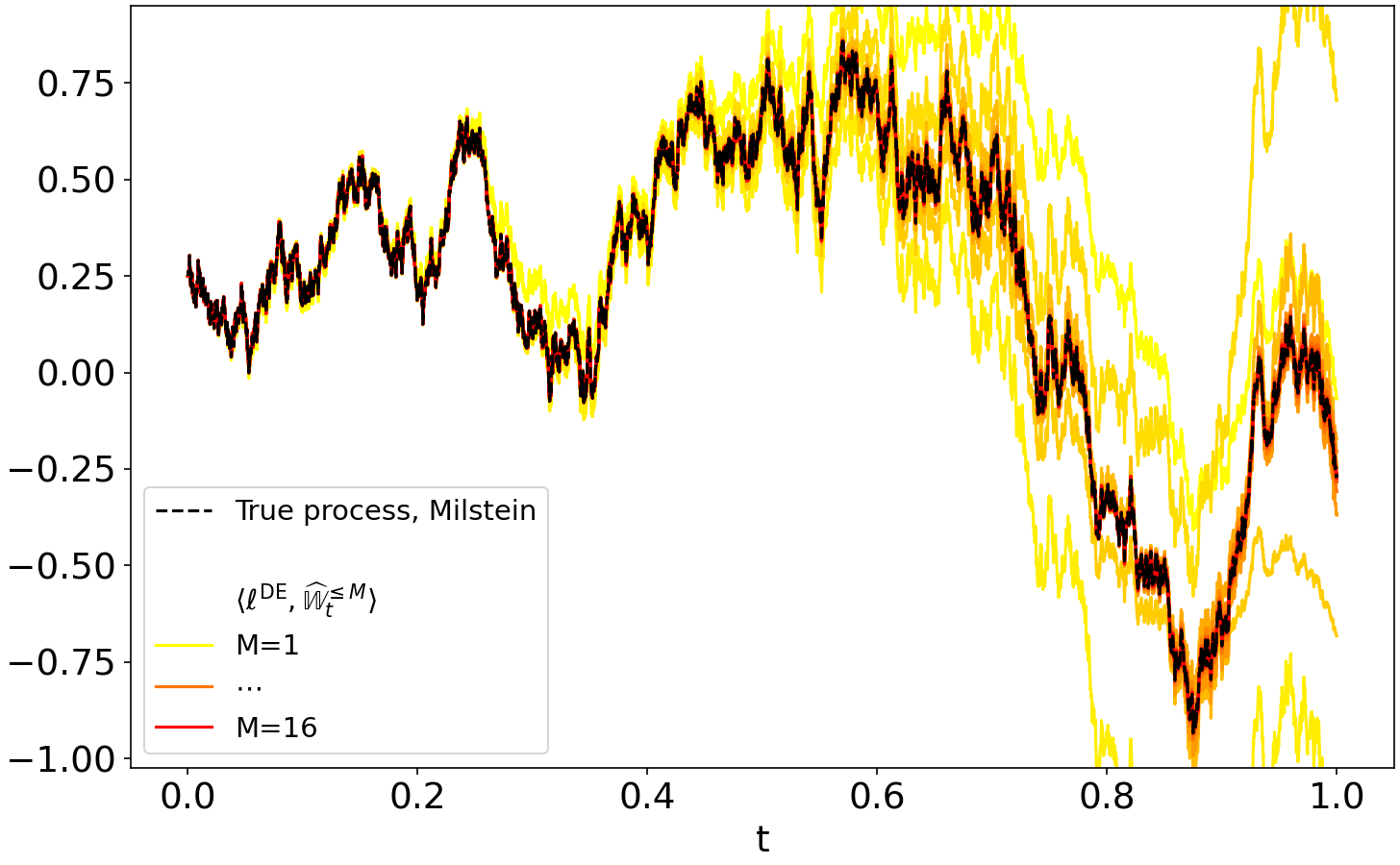}} 
        \caption{Trajectories of a delay process (black) against their truncated time-independent linear representation \protect \eqref{eq:lde}, i.e. $\bracketsigtrunc[M]{\lde}$, for several truncation orders $M$.}
        \label{fig:ex_traj-DE}
    \end{figure}

    \begin{table}[H]
        \centering
        \begin{tabular}{|l|cc|}
            \hline
                    & (a)       & (b)       \\
            \hline
             M=2    & 4.808e-01 & 4.405e-01 \\
             M=4    & 4.015e-02 & 2.795e-01 \\
             M=8    & 4.153e-04 & 2.511e-02 \\
             M=16   & 7.584e-07 & 2.739e-07 \\
            \hline
        \end{tabular}
        \caption{Mean squared error between the delay process and its truncated time-independent linear representation \protect \eqref{eq:lde}, i.e.~$\bracketsigtrunc[M]{\lde}$, for several truncation orders $M$, averaged across 100,000 simulations over 1000 time-steps.}
        \label{tab:mse_traj-DE}
    \end{table}
    
    We can see in Table~\ref{fig:ex_traj-DE} and Figure~\ref{tab:mse_traj-DE} the clear convergence of the truncated linear representation for the delay equation process. \\

    Recall that here since $\lde \in \Ah$, the inequality \eqref{eq:boundconvergenceh} could in principle yield a control of the strong error, i.e. for every $\epsilon > 0$, there exists $M_{\epsilon} \in \N$ such that 
    $$ \E \sqbra{\sup_{t \in [0, T]} \abs{\bracketsigtrunc[M_{\epsilon}]{\lde} - \bracketsig{\lde}}} \leq \epsilon. $$

    For related numerical illustrations with other processes, we refer to \citet{sig_vol}.

{\revone
\subsubsection{Linear Volterra process} \label{subsec:numerics_volterra}

    For $n \in \N$, we consider a linear Volterra process $Y^n$ solution to the Volterra equation \eqref{eq:volterra}, where the kernel is given by a sum of $n$ weighted exponentials
    \begin{align} \label{eq:kernelepsn}
        K_n^{\alpha, \varepsilon} (t) = \sum_{i=1}^n w_i e^{-x_i t}.
    \end{align}
    We denote by $\lvol_n$ the exact coefficients given from the signature  representation in Theorem~\ref{thm:volterra} and we consider $\langle \lvol_n, \sig^{\leq M} \rangle$ its truncated exact representation at level $M \geq 1$. The parameters $(x_i, w_i)_{i=1,\ldots,n}$ are chosen in such a way to optimally approximates the $\varepsilon$-shifted fractional kernel
    \begin{align}
        K^{\alpha, \varepsilon} (t) := \frac{(t+\varepsilon)^{\alpha-1}}{\Gamma(\alpha)} = \int_0^\infty e^{yt} \mu^{\alpha, \varepsilon} (\d y),
        \quad \text{with} \quad
        \mu^{\alpha, \varepsilon} (\d y) := \frac{y^{-\alpha} e^{-y \varepsilon}}{\Gamma(\alpha) \Gamma(1-\alpha)} \d y.
    \end{align}
    Given a partition $(\eta_i)_{0 \leq i \leq n}$ where $0 = \eta_0 < \eta_1 < \cdots < \eta_n < \infty$, we can readily compute its corresponding points $x_i$ and weights $w_i$
    \begin{align}
        w_i
        &= \int_{\eta_{i-1}}^{\eta_i} \mu^{\alpha, \varepsilon} (\d y)
        = \varepsilon^{\alpha-1} \frac{\Gamma \left( 1-\alpha; \varepsilon \, \eta_{i-1}, \varepsilon \, \eta_i \right)}{\Gamma(\alpha) \Gamma(1-\alpha)},
        \\
        x_i
        &= \frac{1}{w_i} \int_{\eta_{i-1}^n}^{\eta_i^n} y \; \mu^{\alpha, \varepsilon} (\d y)
        = \frac{\varepsilon^{\alpha-2}}{w_i} \frac{\Gamma \left( 2-\alpha; \varepsilon \, \eta_{i-1}, \varepsilon \, \eta_i \right)}{\Gamma(\alpha) \Gamma(1-\alpha)},
    \end{align}
    where $\Gamma \left( z; a, b \right) = \int_a^b t^{z-1} e^{-t} \d z$ is the incomplete gamma function. We chose a geometric partition $\eta_i = (r_n)^{i - n/2}$ for some $L^2$-optimal $r_n>0$ as in \cite{abi2019lifting, jaber2018multifactor, sergio_fract}, i.e.
    \begin{align}
        r_n
        &= \arg \min_r || K_n^{\alpha, \varepsilon} - K^{\alpha, \varepsilon} ||_{L^2(\tau, T)}
        \\
        &= \arg \min_r \left( \sum_{i,j=1}^n w_i w_j \frac{e^{-(x_i + x_j) \tau} - e^{-(x_i + x_j) T}}{x_i + x_j} - 2 \sum_{i=1}^n w_i (x_i)^{-\alpha} e^{x_i \varepsilon} \frac{\Gamma \left( \alpha; (\tau + \varepsilon) \, x_i, (T + \varepsilon) \, x_i \right)}{\Gamma{(\alpha)}} \right).
    \end{align}

    In Figure~\ref{fig:ex_traj-VOL}, and later in Figure~\ref{fig:mom-VOL_m}, we compare the truncated signature representation $\langle \lvol_n, \sig^{\leq M} \rangle$ of the multifactor approximation of the Volterra process $Y^n$ with the kernel \eqref{eq:kernelepsn} to the Euler discretization of the (true) Volterra process $Y$, i.e.
    $$ Y_t = \int_0^t K^{\alpha, \epsilon} (t-s) ((a_1 + b_1 Y_s) \d s + ((a_2 + b_2 Y_s) \d W_s). $$
    In particular, this provides a numerical illustration of the approximation result given in Corollary~\ref{corr:convergence-K^n}.  Figure~\ref{fig:ex_traj-VOL} shows that for different values of $\alpha$, the truncated signature of the multifactor approximation with $n=10$ factors yields very accurate sample paths, at least for short time horizons. 
    \begin{figure}[H]
        \centering 
        \subfloat[\centering $\alpha=0.2, r_n=2.936$]{\includegraphics[width=\twoplotswidth]{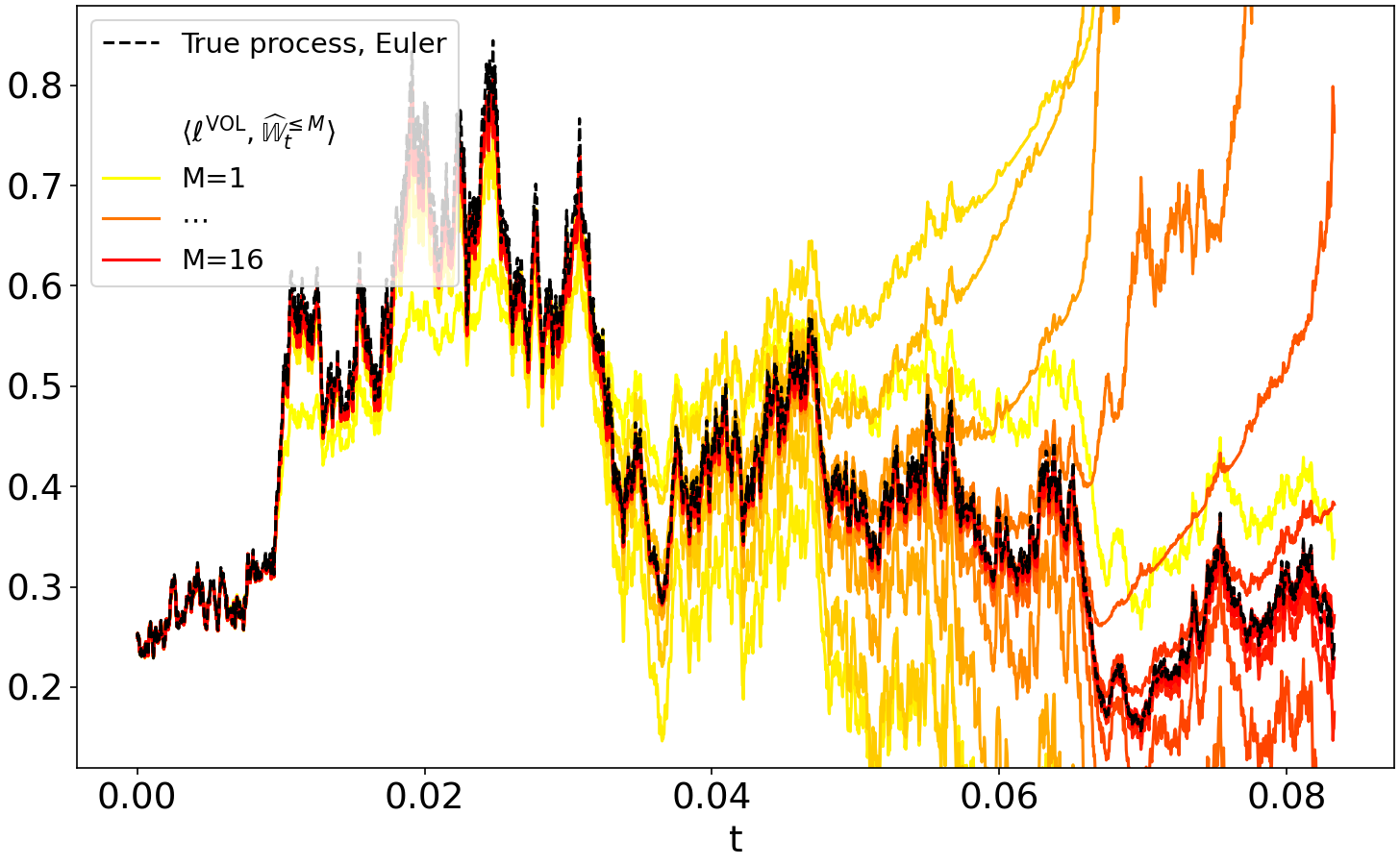}}
        \quad
        \subfloat[\centering $\alpha=0.4, r_n=2.884$]{\includegraphics[width=\twoplotswidth]{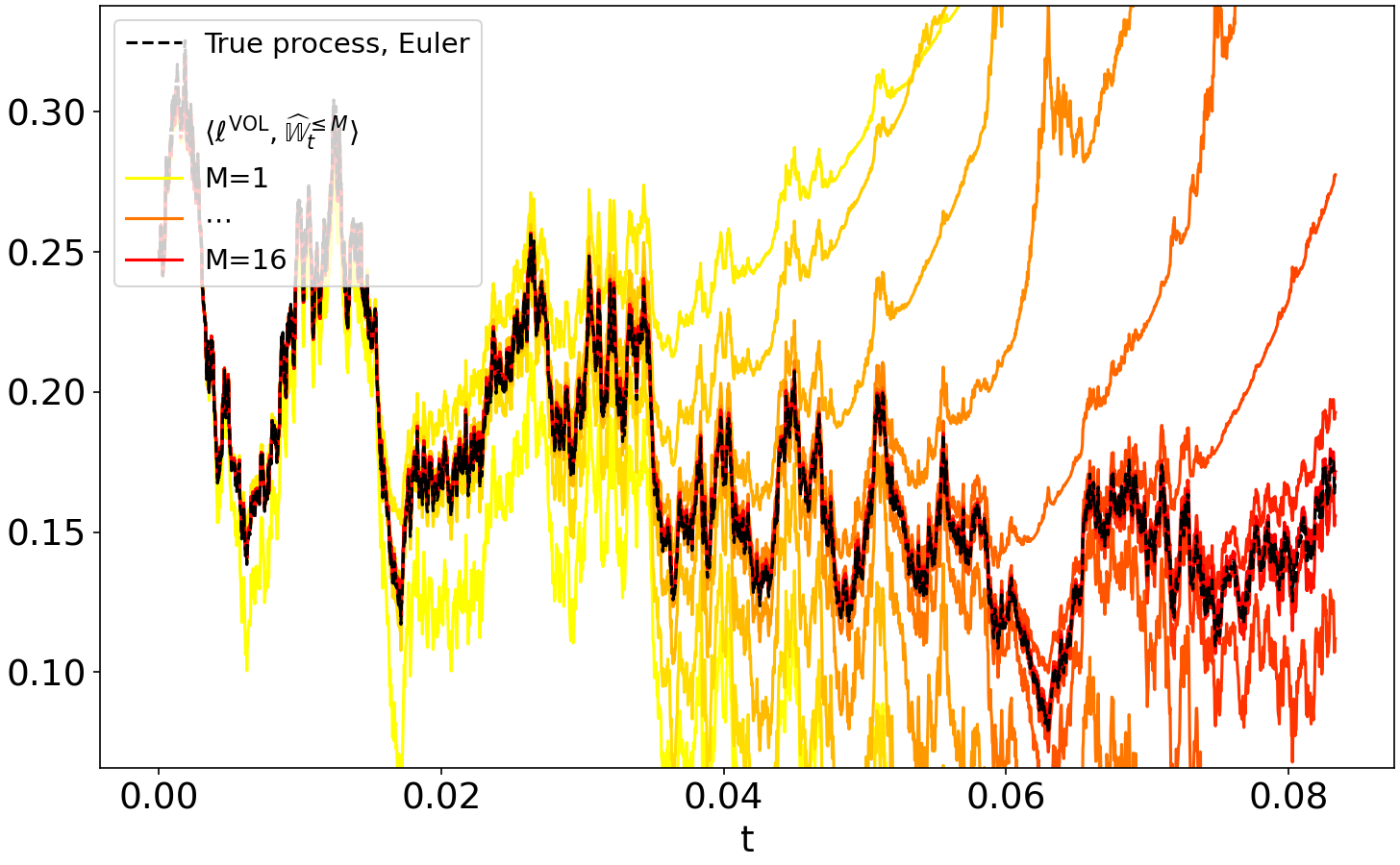}}
        \qquad
        \subfloat[\centering $\alpha=0.6, r_n=2.831$]{\includegraphics[width=\twoplotswidth]{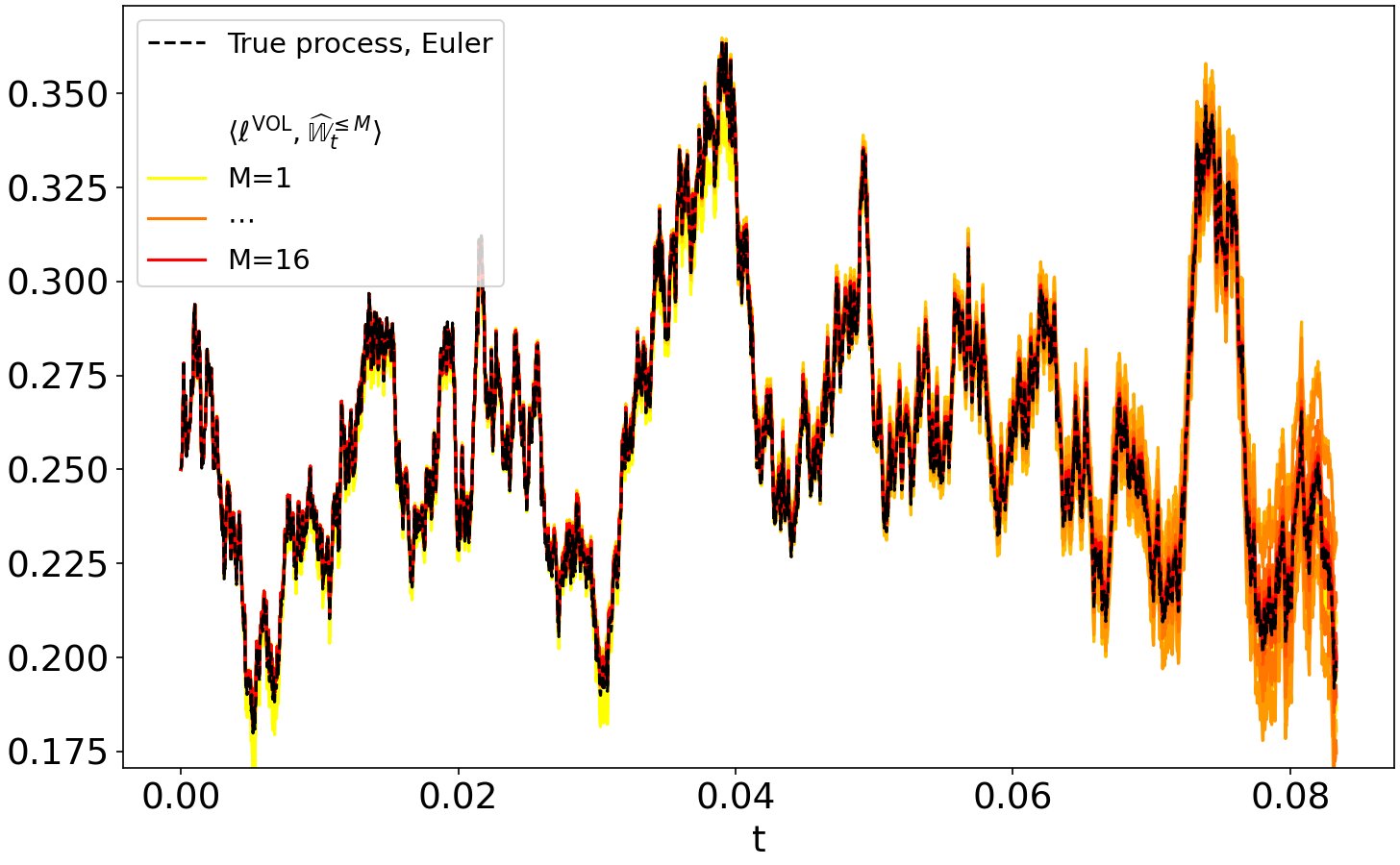}}
        \quad
        \subfloat[\centering $\alpha=0.8, r_n=2.778$]{\includegraphics[width=\twoplotswidth]{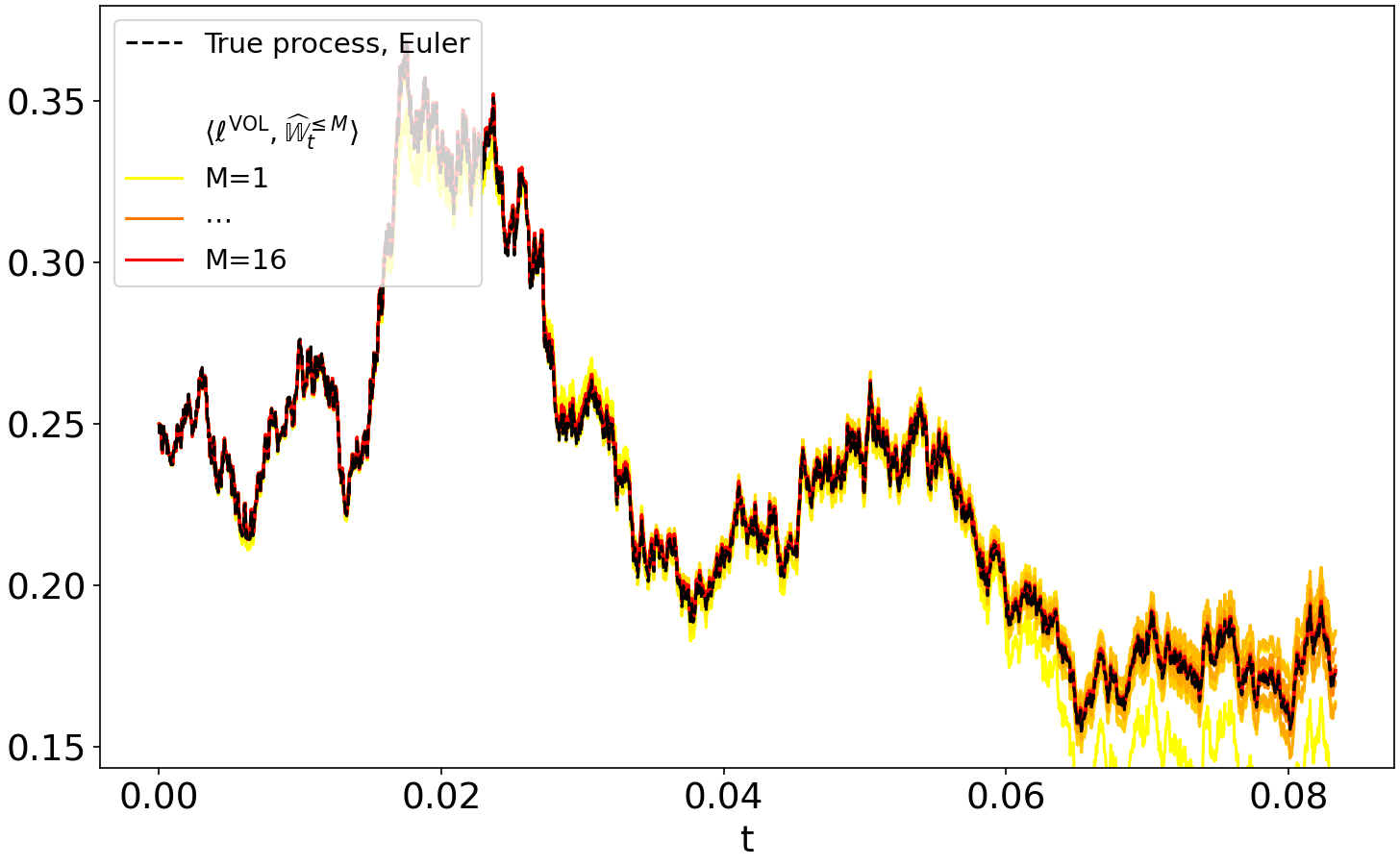}}
        \caption{Trajectories of the shifted fractional Volterra process (black) against their truncated time-independent linear multifactor approximation representation \protect \eqref{eq:lvol}, i.e.~$\bracketsigtrunc[M]{\lvol_n}$, for several truncation orders $M$ and $y=0.25, a_1=0.25, b_1=-1, a_2=0, b_2=1, \varepsilon=1/52, T=1/12, \tau=1/106, n=10$.}
        \label{fig:ex_traj-VOL}
    \end{figure}

\section{Conditional moments} \label{sec:conditional_moments}

    As a key application of Section~\ref{sec:volterra} and Section~\ref{sec:delayed}, our representation formulas provide explicit expressions for the conditional moments of Volterra and delay processes, which can be efficiently computed numerically without relying on Monte Carlo simulations. \\
    
    The computation of unconditional and conditional moments is a central but notoriously difficult problem for non-Markovian processes. Our approach leverages the Markovianity of the signature of $\widehat{\mathbb{W}}$ to make even conditional moments straightforward to obtain. The linear Volterra process \eqref{eq:volterra} is a particular case of a polynomial Volterra process studied in \cite{polynomial_volterra}, where unconditional moments were derived in terms of deterministic integral equations, without any treatment of conditional moments or numerical methods. To the best of our knowledge, this is the first time that explicit formulas are used to implement efficient numerics for both unconditional and conditional moments in this setting. Moreover, the delayed case \eqref{eq:sde-DE-exp} with nonzero $b_2$ or $K_2$ lies outside the polynomial Volterra family and thus cannot be recovered from \cite{polynomial_volterra}, further underlining the significance of our representation formulas. \\
    
    Theorem~\ref{thm:bell_shupow_m_conditional} establishes the conditional moment formulas combining our representation formulas in \eqref{eq:lvol} and \eqref{eq:lde} with the celebrated Fawcett formula \cite{fawcett}: 
    \begin{align*}
        \E[\sig] = \fawcett, \quad t \leq T, 
    \end{align*}
    with
    $$ \fawcett := \conexp{\bra{\word{1} + \half \word{22}} t}. $$

    Before stating Theorem~\ref{thm:bell_shupow_m_conditional}, we will generalize the projection defined in \eqref{eq:projection} against linear functionals of the form
    $$ \bphi = \conexp{a_1 \word{1} + b_2 \word{22}} = \sum_{m=0}^\infty \frac{1}{m!} \bra{a_1 \word{1} + a_2 \word{22}} \conpow{m}, $$
    and show that under this framework it is the adjoint concatenation operation.
    
    \begin{lemma} \label{lem:proj_fawcett}
        Let $\bell \in \Aexp$ and $\bphi = \conexp{a_1 \word{1} + \sum_{\word{i} \in \set{\word{2}, \cdots, \word{d}}} a_i \word{ii}}$ for some constants $a_1, a_2, \cdots, a_d \in \R$. Then
        $$ \bell |_{\bphi} := \sum_{n=0}^\infty \sum_{\word{v} \in V_n} \bphi^\word{v} \bell \proj{v} \in \Aexp $$
        and 
        $$ \bracketsig{\bell |_{\bphi}} = \bracket{\bell}{\sig \otimes \bphi}, \quad t \leq T. $$
    \end{lemma}
    
    \begin{proof}
        Without loss of generality, we will assume $d=2$. Now, recall that for all $C > 0$ and $\word{v} \in V_n$ with $n \in \N$, $\shuexp{C (\word{1} + \word{2})} \proj{v} \dominated C^n \shuexp{C (\word{1} + \word{2})}$. This essentially means that we are only interested in the length $n$ of $\word{v}$ when dominating projections of elements of $\Aexp$. Then, since by assumption on $\bell$, there exists $C > 1$ such that $\bell \dominated C \shuexp{C (\word{1} + \word{2})}$, we can write
        \begin{align}
            \bell |_{\bphi}
            &
            \dominated \sum_{m=0}^\infty \frac{1}{m!} \sum_{k=0}^m \binom{m}{k} \abs{a_1}^{m-k} \abs{a_2}^k C \shuexp{C (\word{1} + \word{2})} |_{\word{1} \conpow{m-k} \word{2} \conpow{2k}}
            \\
            &
            \dominated C \shuexp{C (\word{1} + \word{2})} \sum_{m=0}^\infty \sum_{k=0}^m \frac{\abs{a_1}^{m-k}}{(m-k)!} \frac{\abs{a_2}^k}{k!} C^{m+k}
            \\
            &
            \dominated C \shuexp{C (\word{1} + \word{2})} e^{\bra{C \abs{a_1} + C^2 \abs{a_2}}},
        \end{align}
        which clearly shows $\bell |_{\bphi} \in \Aexp$. We can now decompose $\bracketsig{\bell |_{\bphi}}$ as follows
        \begin{align}
            \bracketsig{\bell |_{\bphi}}
            &
            = \sum_{n=0}^\infty \sum_{\word{v} \in V_n}  \bphi^\word{v} \bracketsig{\bell \proj{v}}
            \\
            &
            = \sum_{n=0}^\infty \sum_{\word{v} \in V_n}   \bphi^\word{v} \sum_{m=0}^\infty \sum_{\word{u} \in V_m} \bell^\word{uv} \sig^\word{u}
            \\
            &
            = \sum_{n=0}^\infty \sum_{m=0}^\infty  \sum_{\word{w} \in V_{m+n}} \bell^\word{w} \sum_{\word{v} \in V_n} \sum_{\word{u} \in V_m} \sig^\word{u} \bphi^\word{v} \indic{\word{w} = \word{uv}}
            \\
            &
            = \sum_{k=0}^\infty \sum_{n=0}^k \sum_{\word{w} \in V_{k}} \bell^\word{w} \sum_{\word{v} \in V_n} \sum_{\word{u} \in V_{k-n}} \sig^\word{u} \bphi^\word{v} \indic{\word{w} = \word{uv}}
            \\
            &
            = \sum_{k=0}^\infty \sum_{\word{w} \in V_k} \bell^\word{w} \bra{\sig \otimes \bphi}^\word{w}
            \\
            &
            = \bracket{\bell}{\sig \otimes \bphi},
        \end{align}
        thereby completing the proof.
    \end{proof}

    \begin{theorem} \label{thm:bell_shupow_m_conditional}
        Consider the representation formulas for $Y$ and $Z$ given by  \eqref{eq:lvol} and \eqref{eq:lde}.  
        Let $m \in \N$, then
        \begin{align}
            \E \sqbra{Y_T^m \mid \F_t}
            = \bracketsig{(\lvol) \shupow{m} |_{\fawcett[T-t]}},
            \quad
            \E \sqbra{Z_T^m \mid \F_t}
            = \bracketsig{(\lde) \shupow{m} |_{\fawcett[T-t]}},
            \quad t \leq T. 
        \end{align}
    \end{theorem}

\begin{proof}
        By Theorems~\ref{thm:volterra} and \ref{thm:delayed}, we have $\lvol, \lde \in \Aexp$. Proposition~\ref{prop:Aexp_closed}~\ref{prop:Aexp_closed_shuffle} thus makes it clear that $(\lvol) \shupow{m}, (\lde) \shupow{m} \in \Aexp$. We then apply the shuffle property of Proposition~\ref{prop:shuffle_property} to write
        $$ Y_T^m = \bracketsig[T]{(\lvol) \shupow{m}}, \qquad Z_T^m = \bracketsig[T]{(\lde) \shupow{m}}. $$
        Applying the dominated convergence theorem, we get
        $$ \E \sqbra{\bracketsig[T]{\bell \shupow{m}} \mid \F_t} = \bracket{\bell \shupow{m}}{\E \sqbra{\sig[T] \mid \F_t}}, $$
        for $\bell \in \big\{ \lvol, \lde \big\}$. Using Chen's identity \cite{chen1957integration}, together with Fawcett's formula \cite{fawcett}, we have
        $$ \E \sqbra{\sig[T] \mid \F_t} = \sig[t] \otimes \E \sqbra{\sig[T-t] \mid \F_t} = \sig[t] \otimes \fawcett[T-t]. $$
        Finally, we apply Lemma~\ref{lem:proj_fawcett} which completes the proof.
    \end{proof}


    Figure~\ref{fig:mom-VOL_m} displays the unconditional and conditional third and fourth moments of the linear Volterra process as in Subsection~\ref{subsec:numerics_volterra} computed using Theorem~\ref{thm:bell_shupow_m_conditional}, illustrating the convergence as the truncation level $M$ increases for relatively short maturities. The reference value and the confidence intervals in black are computed using Monte Carlo simulation with 50,000 simulations and 500 time steps. We recall that, as was the case for Figure~\ref{fig:ex_traj-VOL}, the moments of the true process $Y$ are estimated with the true kernel $K^{\alpha, \epsilon}$, and the signature representations are computed using the multifactor approximation. This illustrates once again the usefulness of Corollary~\ref{corr:convergence-K^n}.


    
    \begin{figure}[H]
        \centering
        \subfloat[\centering Unconditional, $m=3$]{\includegraphics[width=\twoplotswidth]{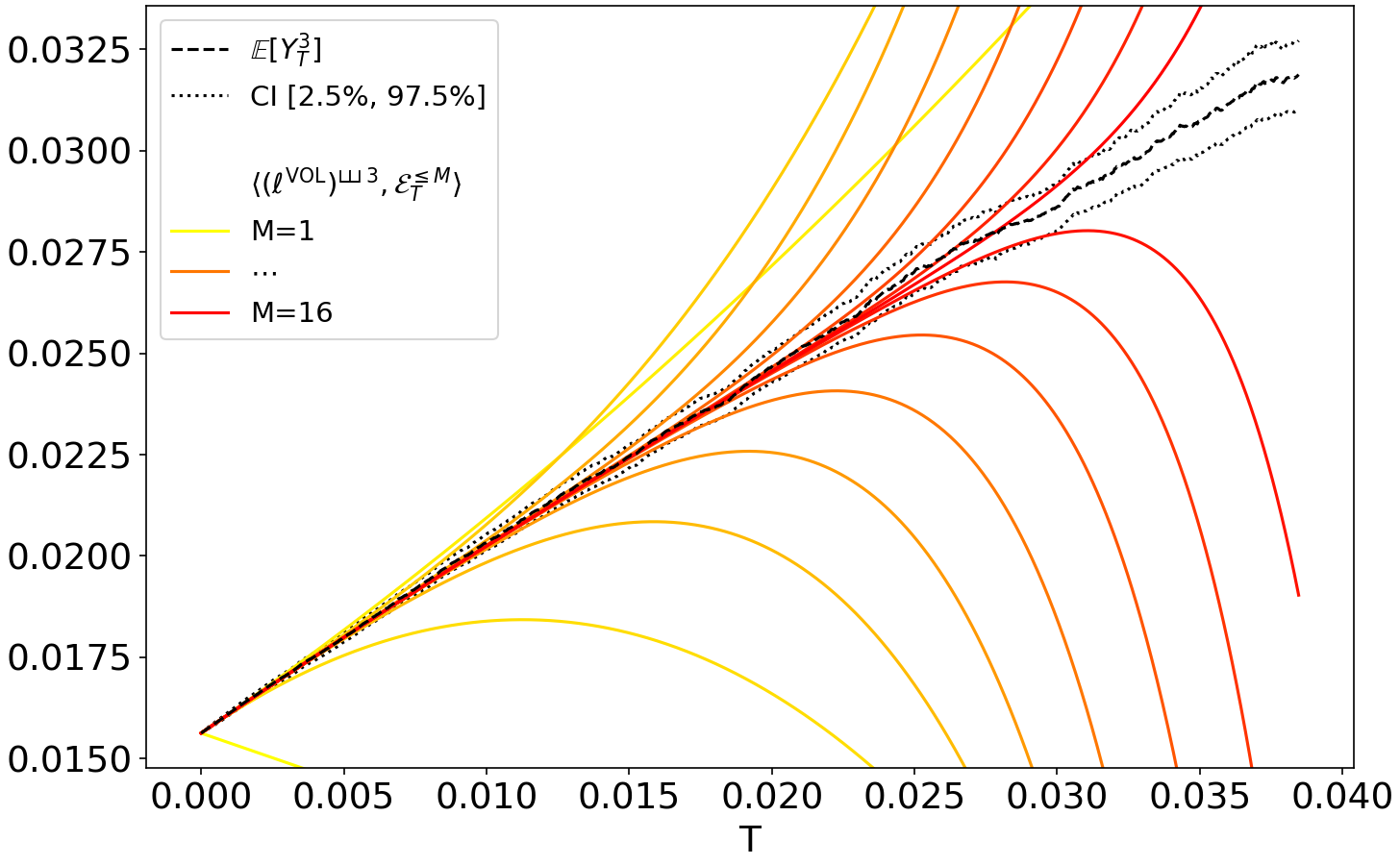}}
        \quad
        \subfloat[\centering Unconditional, $m=4$]{\includegraphics[width=\twoplotswidth]{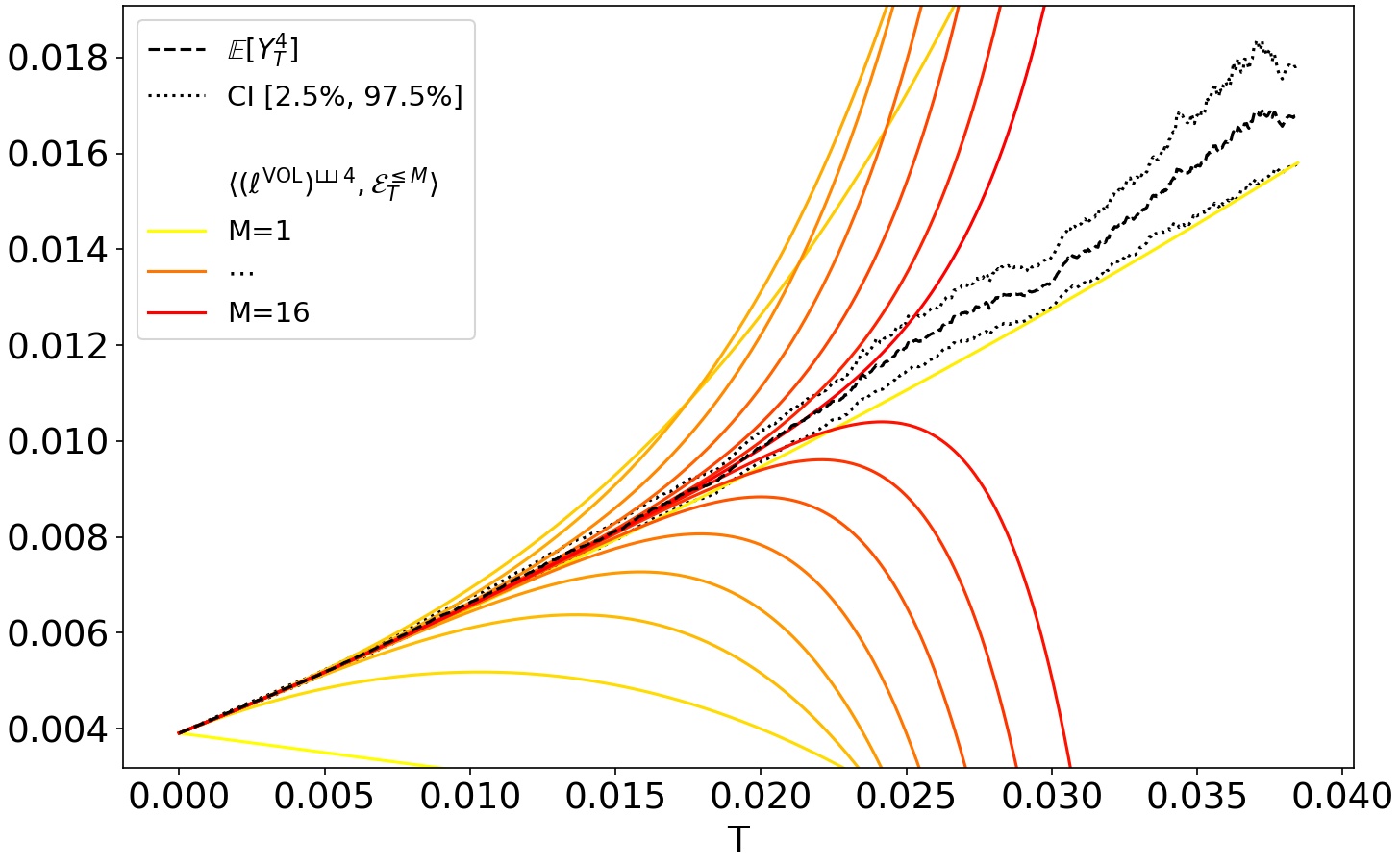}}
        \qquad
        \subfloat[\centering Conditional, $m=3$]{\includegraphics[width=\twoplotswidth]{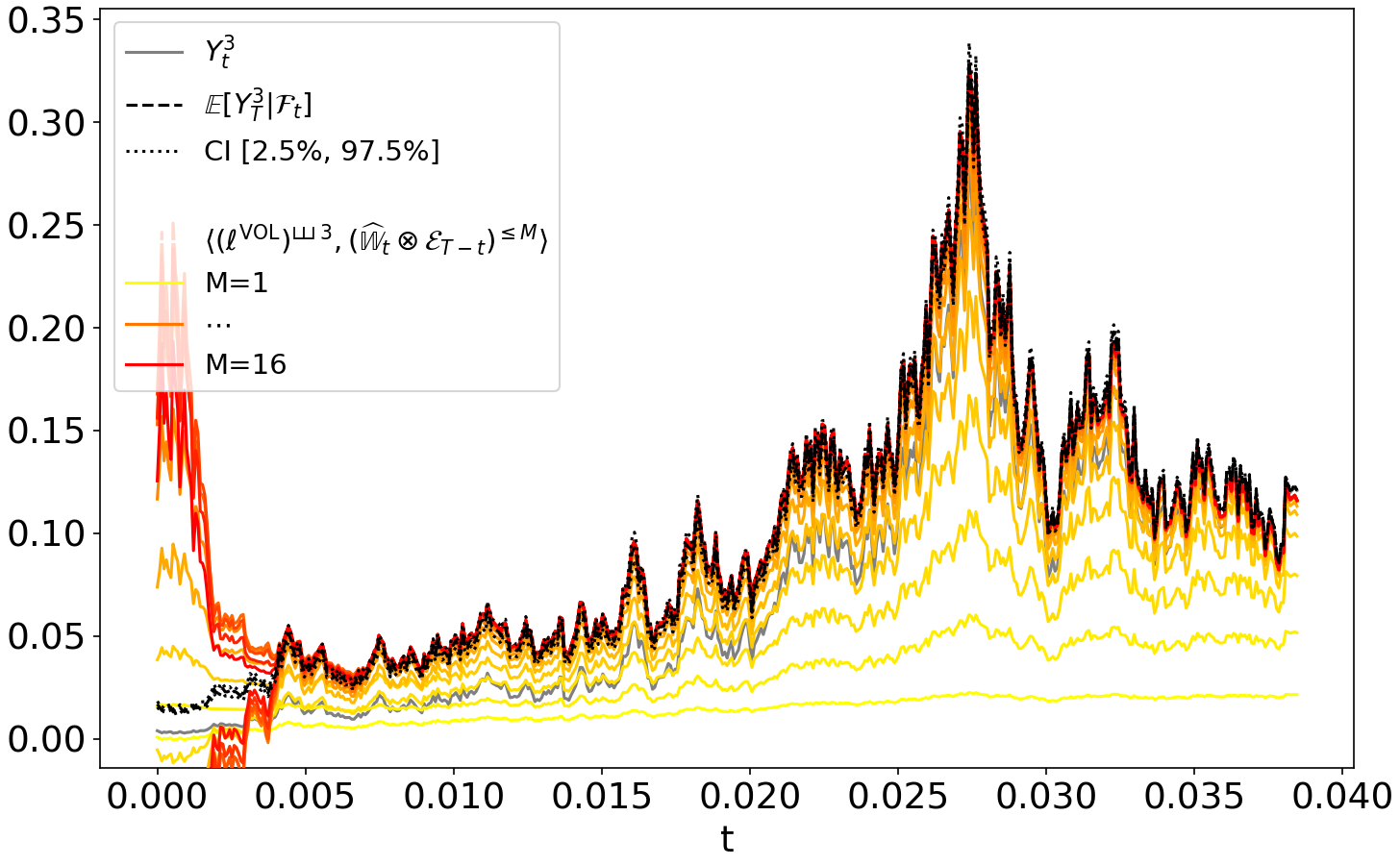}}
        \quad
        \subfloat[\centering Conditional, $m=4$]{\includegraphics[width=\twoplotswidth]{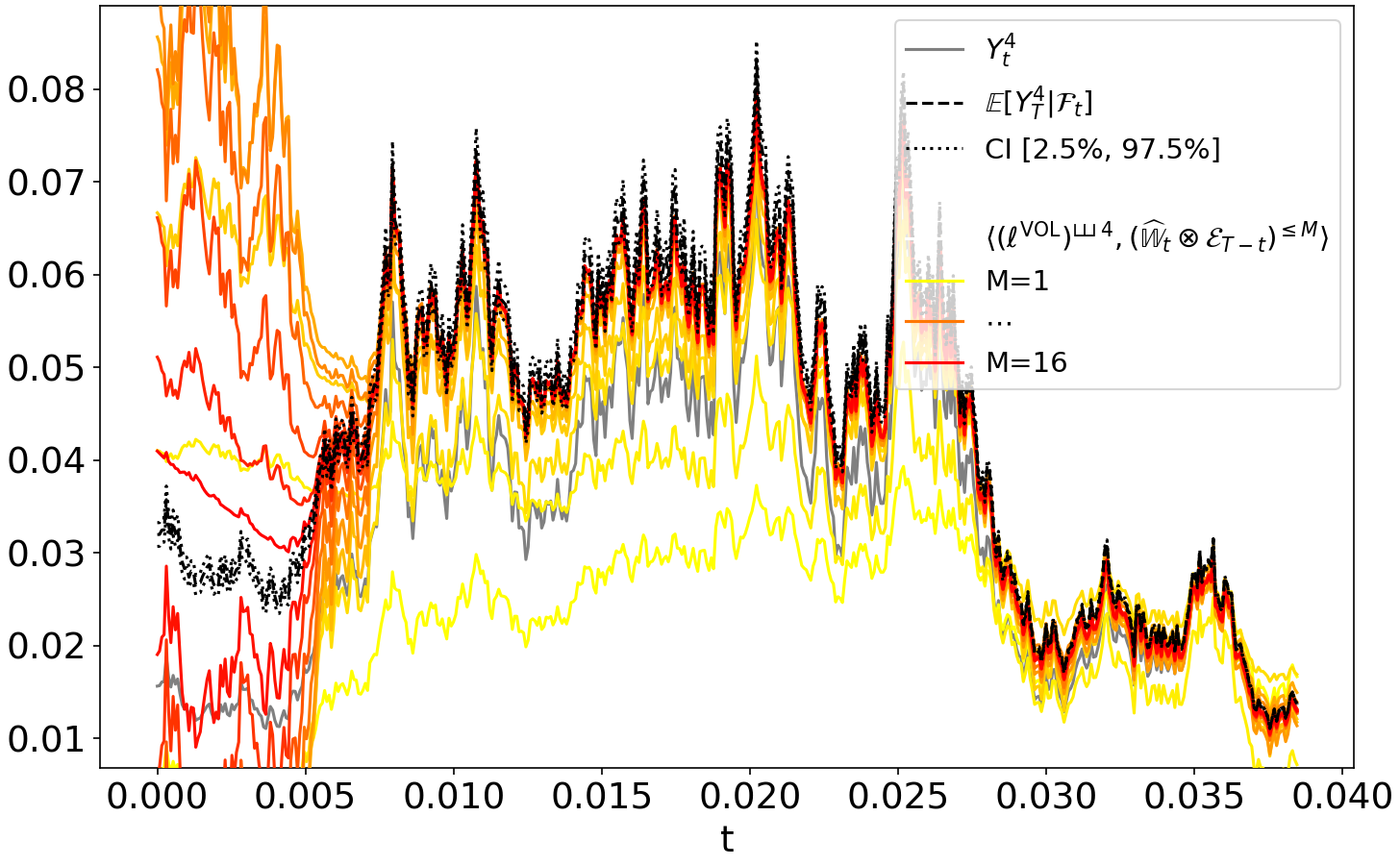}}
        \caption{$m^{th}$ unconditional, (a) and (b), and conditional, (c) and (d), moments of the shifted fractional Volterra process (black) against their truncated time-independent linear multifactor approximation representation \protect \eqref{eq:lvol} to the $m^{th}$ (shuffle) power, i.e.~$\bracketsigtrunc[M]{(\lvol_n) \shupow{m}}$, for several truncation orders $M$ and $y=0.25, a_1=0.25, b_1=-1, a_2=-0.1, b_2=1, \alpha=0.6, \varepsilon=1/52, T=1/26, \tau=1/106, n=10$.}
        \label{fig:mom-VOL_m}
    \end{figure}

    Finally, Figure~\ref{fig:mom-DE_m} displays the unconditional and conditional third and fourth moments of the linear Delayed equation process as in Section~\ref{sec:delayed} computed using Theorem~\ref{thm:bell_shupow_m_conditional}, illustrating the convergence as the truncation level $M$ increases for longer maturities with this choice of parameters. The reference value and the confidence intervals in black are also computed using Monte Carlo simulation with 50,000 simulations and 500 time steps.
    
    \begin{figure}[H]
        \centering
        \subfloat[\centering $m=3$]{\includegraphics[width=\twoplotswidth]{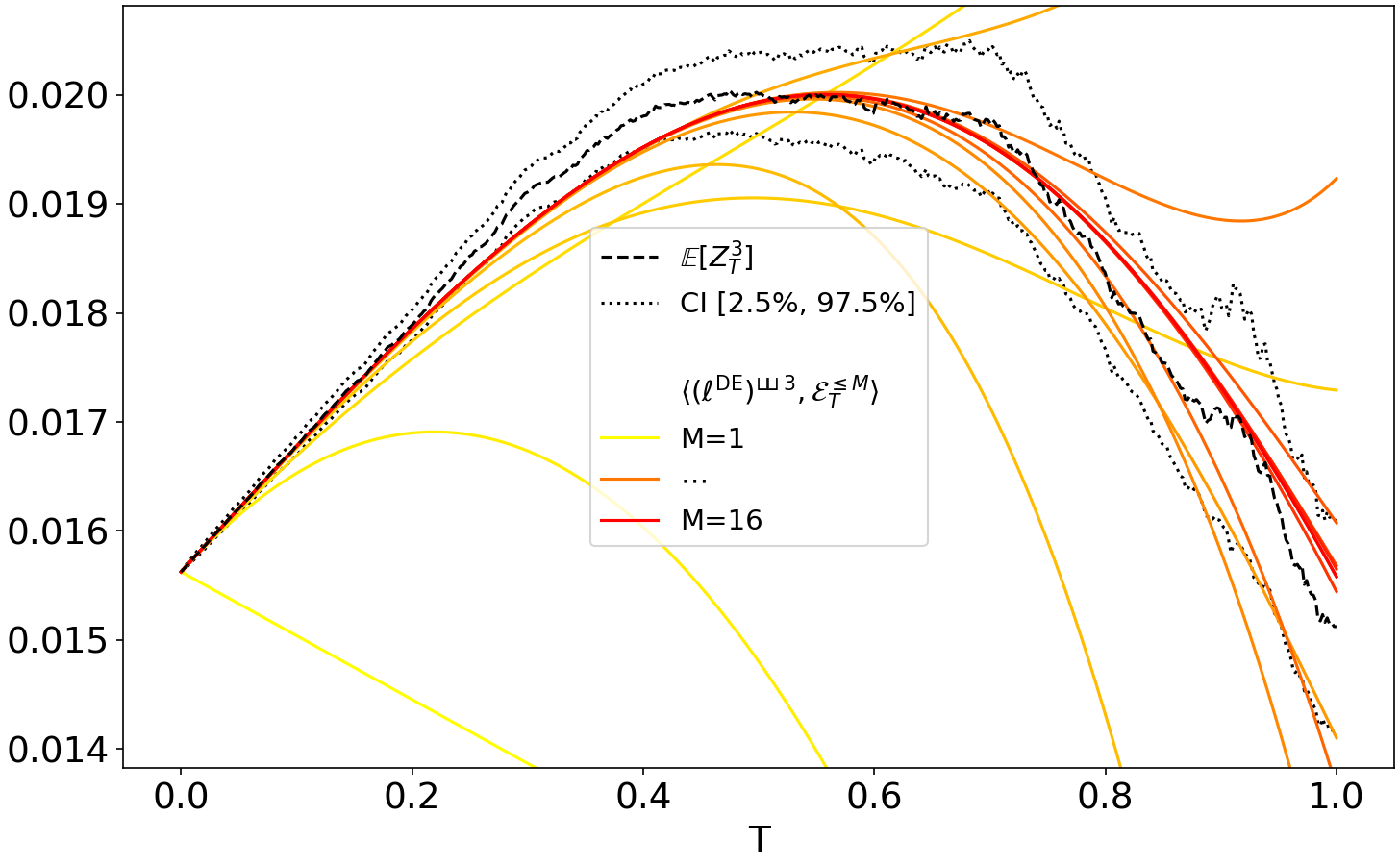}}
        \quad
        \subfloat[\centering $m=4$]{\includegraphics[width=\twoplotswidth]{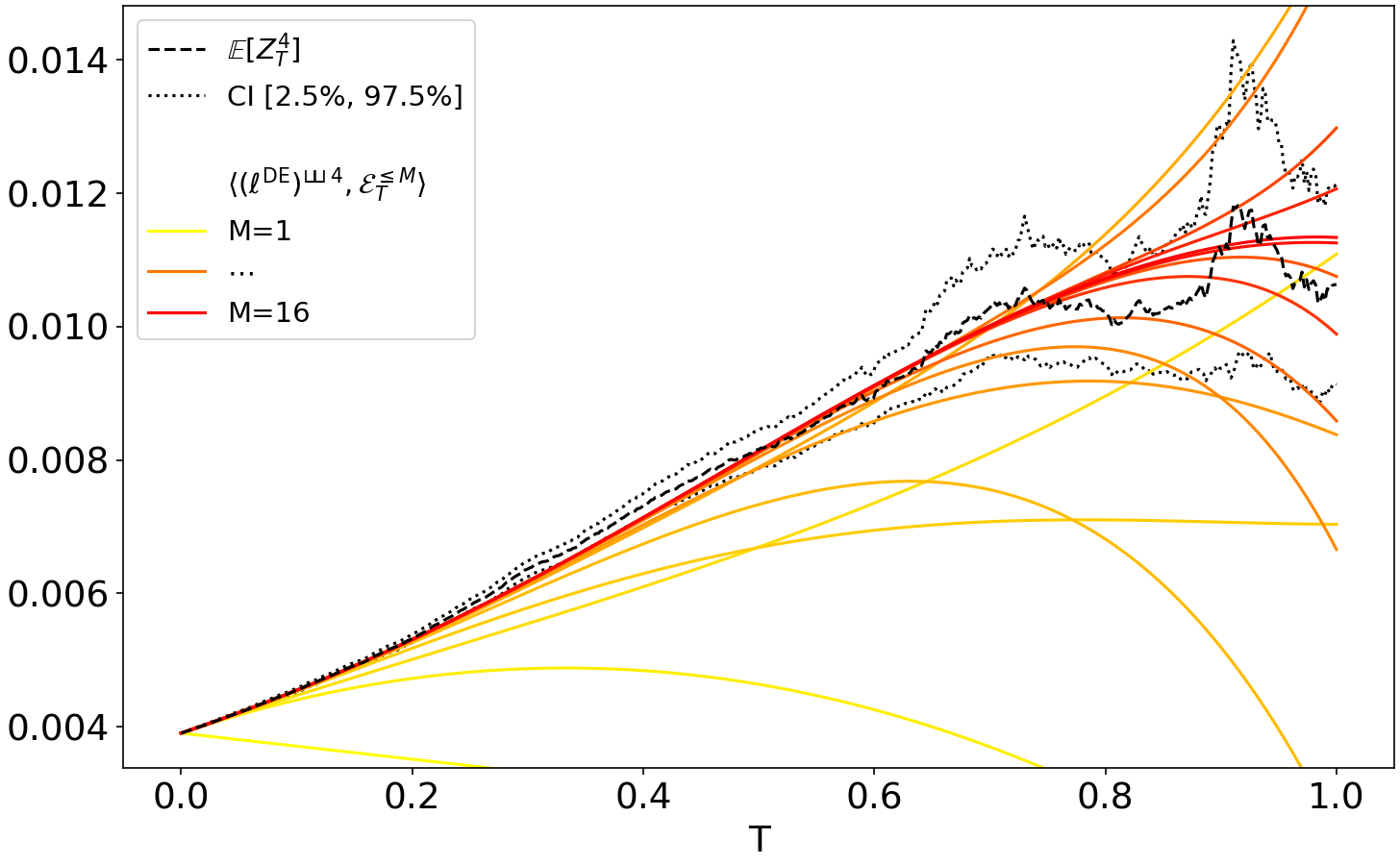}}
        \qquad
        \subfloat[\centering $m=3$]{\includegraphics[width=\twoplotswidth]{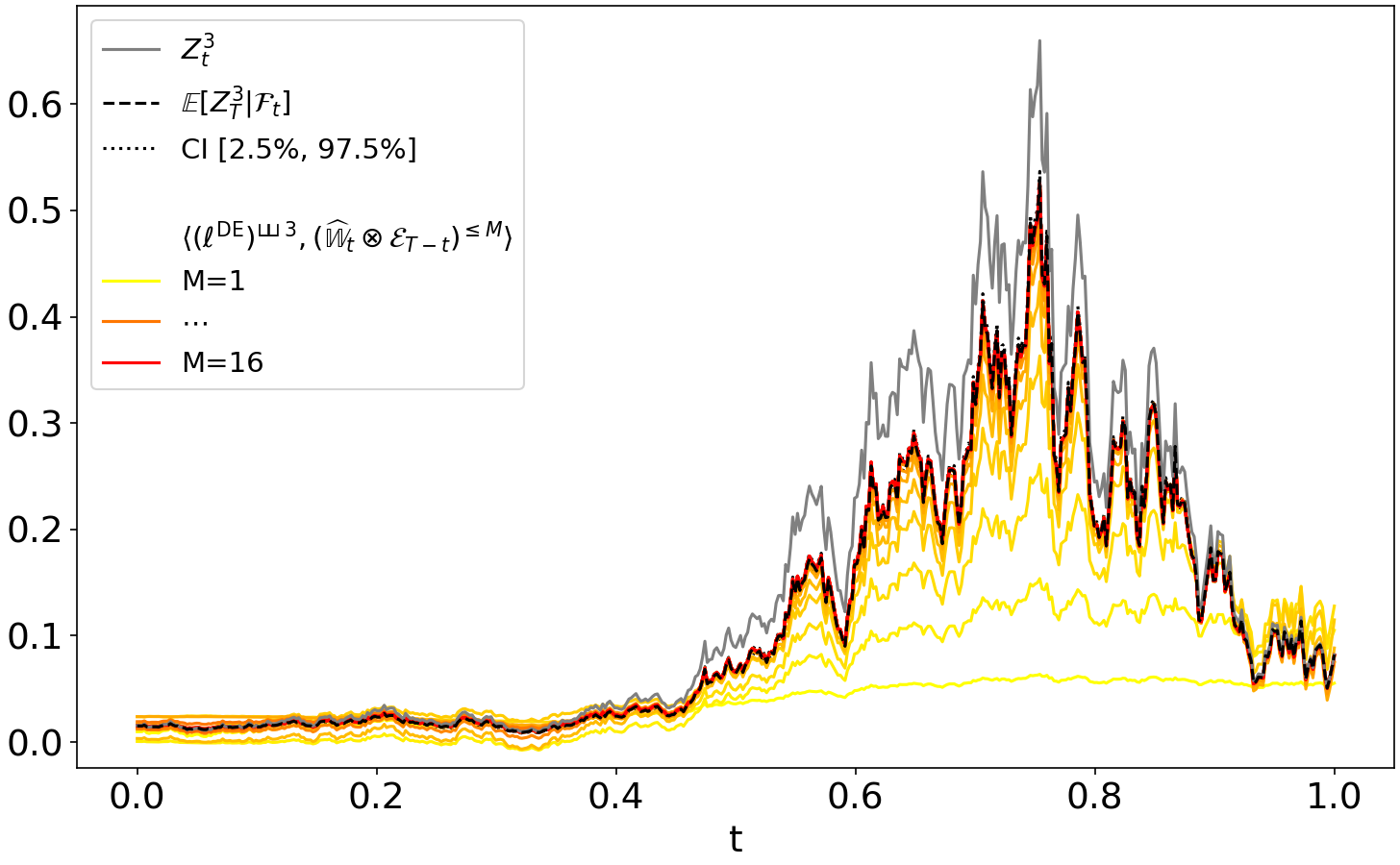}}
        \quad
        \subfloat[\centering $m=4$]{\includegraphics[width=\twoplotswidth]{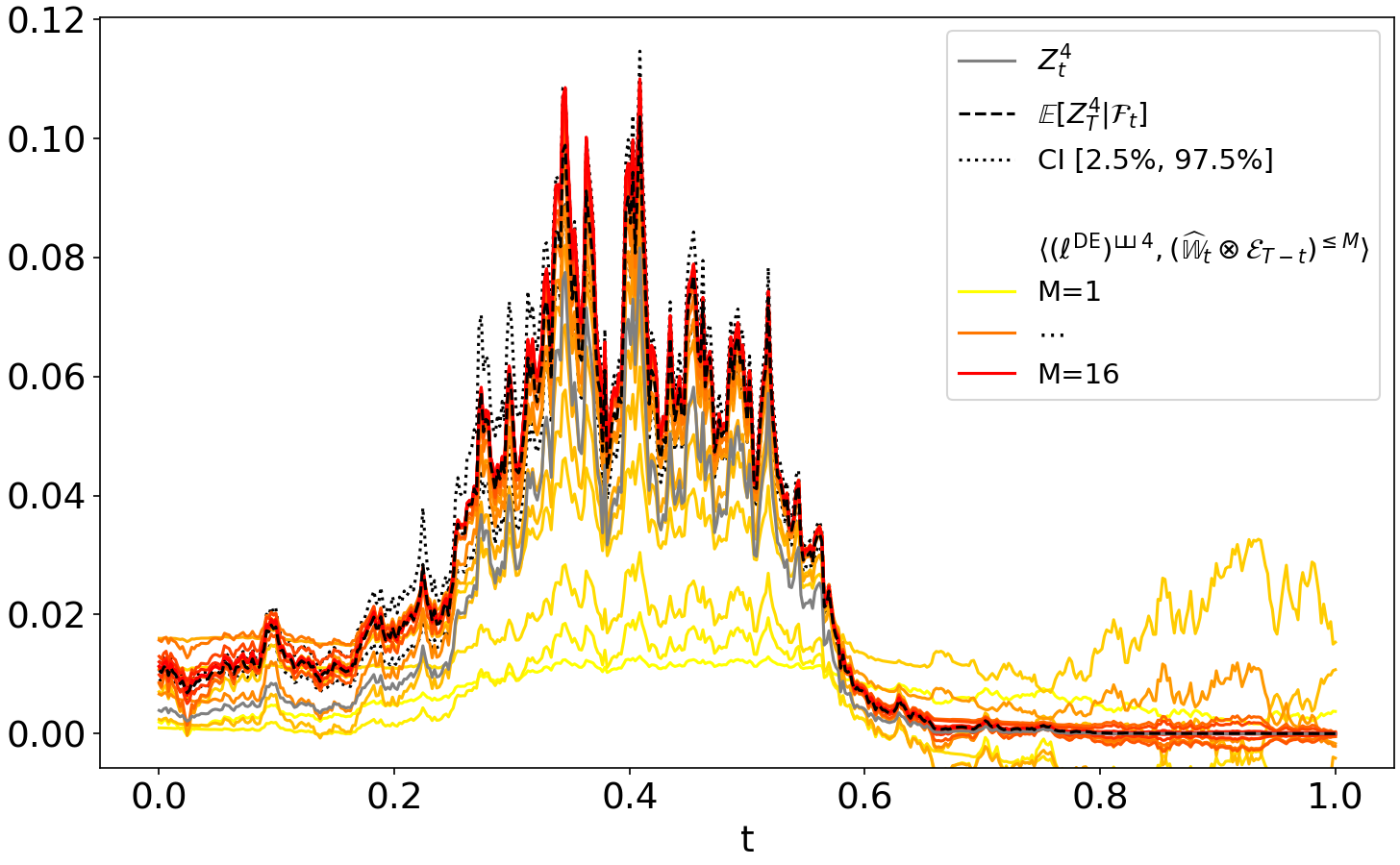}}
        \caption{$m^{th}$ unconditional, (a) and (b), and conditional (c) and (d) moments of the Delayed equation process \protect \eqref{eq:sde-DE-exp} to the $m^{th}$ (shuffle) power, i.e.~$\bracketsigtrunc[M]{(\lde) \shupow{m}}$, for several truncation orders $M$ and $z=0.25, a_1=0.125, b_1=-0.5, c_1=-0.5, \alpha_1=1, a_2=0, b_2=0.5, c_2=0.5, \alpha_2=-1$.}
        \label{fig:mom-DE_m}
    \end{figure}
}

\section{Proof of Theorem~\ref{thm:volterra}} \label{app:volterra}
    {\revtwo
    In a first easy step, we will show that $\pvol$ and $\qvol$ both belong to $\Aexp$. Recall the definition of the kernels $K_1$ and $K_2$ \eqref{eq:assKlinear} and the assumption~\eqref{eq:measure_condition} on them. We can now readily exponentially dominate them:
    \begin{align}
        \pvol &
        = a_1 \word{1} \int_{[0, \infty)} \shuexp{-x \word{1}} \mu_1(\d x) + a_2 \bra{\word{2} - \half K_2(0) b_2 \word{1}} \int_{[0, \infty)} \shuexp{-x \word{1}} \mu_2(\d x) + y \emptyword
        \\ &
        \dominated \bra{\abs{a_1} + \abs{a_2} + \half \abs{a_2b_2 K_2(0)}} (\word{1} + \word{2}) \sum_{n=0}^\infty \word{1} \conpow{n} \int_{[0,\infty)} x^n (\mu_1(\d x) + \mu_2(\d x))
        \\ &
        \dominated \bra{\abs{a_1} + \abs{a_2} + \half \abs{a_2b_2 K_2(0)}} (\word{1} + \word{2}) \shuexp{M \word{1}} + \abs{y} \emptyword,
        \\ \qvol &
        = b_1 \word{1} \int_{[0, \infty)} \shuexp{-x \word{1}} \mu_1({\d x}) + b_2 \bra{\word{2} - \half K_2(0) b_2 \word{1}} \int_{[0, \infty)} \shuexp{-x \word{1}} \mu_2(\d x)
        \\ &
        \dominated \bra{\abs{b_1} + \abs{b_2} + \half \abs{b_2^2 K_2(0)}} (\word{1} + \word{2}) \shuexp{M \word{1}}.
    \end{align}
    By Corollary~\ref{coro:pq_Aexp_implies_ell_Aexp}, $\lvol$ belongs to $\Aexp$ and in particular to $\I(\widehat{W})$. The process  
    \begin{align} \label{eq:barY}
       \bar{Y}_t :=\bracketsig{\lvol} 
    \end{align}
    is then well-defined and admits the Itô decomposition
    \begin{align} \label{eq:vol_decomposition}
        \bar{Y}_t &
        = \bra{\lvol}^\emptyword
        + \int_0^t \bracketsig[s]{\lvol \proj{1} + \tfrac{1}{2} \lvol \proj{22}} \d s
        + \int_0^t \bracketsig[s]{\lvol \proj{2}} \d W_s.
    \end{align}

    It then remains to prove that $\bar{Y}$ solves the equation
    \begin{align} \label{eq:barYequation}
        \begin{split}
            \d \bar{Y}_t &
            = K_1(0) \bra{a_1 + b_1 \bar{Y}_t} \d t + \int_0^t K_1'(t-s) \bra{a_1 + b_1 \bar{Y}_s} \d s \d t
            \\ &
            \quad + K_2(0) \bra{a_2 + b_2 \bar{Y}_t} \d W_t + \int_0^t K_2'(t-s) \bra{a_2 + b_2 \bar{Y}_s} \d W_s \d t,
            \\ \bar{Y}_0 &
            = y,
        \end{split}
    \end{align}
    which is the differential version of \eqref{eq:volterra} obtained from Itô's lemma, since the kernels $K_1, K_2$ are smooth. Then, we immediately deduce that $Y_t = \bar{Y}_t = \bracketsig{\lvol}$ by uniqueness of the strong solution, which will complete the proof of Theorem~\ref{thm:volterra}.
    }

    \begin{lemma} \label{lem:lvolbrackets}
        We have that 
        \begin{align}
            \bracketsig{\lvol \proj{1} + \tfrac{1}{2} \lvol \proj{22}} &
            = K_1(0) \bra{a_1 + b_1 \bar{Y}_t}
            \\ &
            \quad + \int_0^t K_1'(t-s) \bra{a_1 + b_1 \bar{Y}_s} \d s
            + \int_0^t K_2'(t-s) \bra{a_2 + b_2 \bar{Y}_s} \d W_s, \label{eq:lvol1bracket}
            \\ \bracketsig{\lvol \proj{2}} &
            = K_2(0) (a_2 + b_2 \bar{Y}_t). \label{eq:lvol2bracket}
        \end{align}
    \end{lemma}
    
     
    

\subsection{Proof of Lemma~\ref{lem:lvolbrackets}} \label{S:proofLemmalvolbrackets}

    Now, as we will be using integral of elements of $\eTA[2]$, recall the Definition~\ref{def:integral_sig}, we need to make sure its bracket with the signature and concatenation make sense. This is the goal of the next two lemmas:
    
    \begin{lemma} \label{lem:concat}
        Let $\bp \in L^1(\eTA[2], \mu)$, then $\bell \bp, \bp \bell \in L^1(\eTA[2], \mu)$ for every $\bell \in \eTA[2]$,
        $$ \int_{[0, \infty)} \bell \bp(x) \mu(\d x) = \bell \bra{\int_{[0, \infty)} \bp(x) \mu (\d x)} $$
        and 
        $$ \int_{[0, \infty)} \bp(x) \bell \mu(\d x) = \bra{\int_{[0, \infty)} \bp(x) \mu (\d x)} \bell. $$
    \end{lemma}
    
    \begin{proof}
        Obvious.
    \end{proof}

    \begin{lemma} \label{lem:commutation_sig_integral}
        Let $\bp \in L^1(\eTA[2], \mu)$ such that $\int_{[0, \infty)} \normA{\bp(x)} \mu(\d x) < \infty$ for every $t \in [0, T]$, then
        $$ \bracketsig{\int_{[0, \infty)} \bp(x) \mu(\d x)} = \int_{[0, \infty)} \bracketsig{\bp(x)} \mu(\d x). $$ 
    \end{lemma}
    
    \begin{proof}
        By assumption
        \begin{align}
            \int_{[0, \infty)} \normA{\bp(x)} \mu(\d x) &
            =\int_{[0,\infty)} \sum_{n=0}^\infty \abs{\sum_{\wv \in V_n} \bp^\wv(x) \sig^\wv} \mu(\d x)
            < \infty.
        \end{align}
        This implies that $\bp(x) \in \A$ $\mu$-a.e., so by the dominated convergence theorem
        \begin{align*}
            \bracketsig{\int_{[0, \infty)} \bp(x) \mu(\d x)} &
            = \sum_{n=0}^\infty \sum_{\wv \in V_n} \bra{\int_{[0, \infty)} \bp^\wv(x) \mu(\d x)} \sig^\wv
            \\ &
            = \int_{[0, \infty)} \sum_{n=0}^\infty \sum_{\wv \in V_n} \bp^\wv(x) \sig^\wv \mu(\d x)
            \\ &
            = \int_{[0, \infty)} \bracketsig{\bp(x)} \mu(\d x).
        \end{align*}
    \end{proof}

    We now introduce two important propositions, which establish the commutativity of the kernel integral and the bracket.

    \begin{proposition} \label{prop:combined_commute}
        Let $\bell \in \Ah$. Assume that $\mu$ is a measure on $[0, \infty)$ satisfying $\int_{[0, \infty)} x^n \mu(\d x) < M^n$ for every $n \in \N$ and a positive constant $M$. We then have for $\word{i} = \word{1}, \word{2}$
        \begin{align}
            \bracketsig{\bell \wi \bra{\int_{[0, \infty)} x \shuexp{-x \word{1}} \mu(\d x)}} &
            = \bracketsig{\int_{[0, \infty)} \bell \wi x \shuexp{-x \word{1}} \mu(\d x)} \label{eq:prop5.6_i}
            \\ &
            = \int_{[0, \infty)} \bracketsig{\bell \wi x \shuexp{-x \word{1}}} \mu(\d x). \label{eq:prop5.6_ii}
        \end{align}
    \end{proposition}
    
    \begin{proof}
        We start by proving the first equality \eqref{eq:prop5.6_i}. Fix $\word{i} \in \set{\word{1}, \word{2}}$. By assumption of $\mu$, it is easy to show that $\int_{[0, \infty)} x \shuexp{-x \word{1}} \mu(\d x) \in \Ah$ by direct computation. In particular $x \shuexp{-x \word{1}} \in L^1(\A, \mu)$. A direct application of Lemma~\ref{lem:concat} shows that
        $$ \bell \wi \bra{\int_{[0, \infty)} x \shuexp{-x \word{1}} \mu(\d x)} = \int_{[0, \infty)} \bell \wi x \shuexp{-x \word{1}} \mu(\d x) \in \eTA[2]. $$ 
        
        Moreover, Proposition~\ref{prop:submultiplicativity_Sh} ensures that it actually belongs to $\Ah$ so that the bracket is well-defined, which proves the first equality. \\
        
        For the second equality \eqref{eq:prop5.6_ii}, it suffices to verify that $\int_{[0, \infty)} \normA{\bell \wi x \shuexp{-x \word{1}}} \mu(\d x) < \infty$ for every $t \in [0, T]$: 
        \begin{align*}
            \int_{[0, \infty)} \normA{\bell \wi x \shuexp{-x \word{1}}} \mu(\d x) &
            \leq \int_{[0, \infty)} \normh{\bell \wi x \shuexp{-x \word{1}}} \mu(\d x)
            \leq \int_{[0, \infty)} \normh{\bell} \normh{\wi} \normh{x \shuexp{-x \word{1}}} \mu(\d x),
        \end{align*}
        where the second equality comes from Proposition~\ref{prop:submultiplicativity_Sh}. It then follows
        \begin{align*}
            \int_{[0, \infty)} \normA{\bell \wi x \shuexp{-x \word{1}}} \mu(\d x) &
            \leq \normh{\bell} \normh{\wi} \int_{[0, \infty)} \normh{x \shuexp{-x \word{1}}} \mu(\d x)
            \\ & 
            \leq C \normh{\bell} \normh{\wi} \int_{[0, \infty)} \bra{\sum_{n=0}^\infty x^{n+1} (n+1)^{-\frac{1}{4}} \frac{(2t)^n}{\sqrt{(2n-1)!}}} \mu(\d x)
            \\ &
            \leq C
            \normh{\bell} \normh{\wi} \sum_{n=0}^\infty M^{n+1} (n+1)^{-\frac{1}{4}} \frac{(2t)^n}{\sqrt{(2n-1)!}}
             \\ &
            \leq C
            \normh{\bell} \normh{\wi} M \bra{\sum_{n=1}^\infty M^{n}  \frac{(2t)^n}{(n-1)!} + 1}
            \\ &
            = C \normh{\bell} \normh{\wi} M \bra{2 t M e^{2 t M} + 1}
            \\ &
            < \infty.
        \end{align*}
        
        An application of Lemma~\ref{lem:commutation_sig_integral} completes the proof.
    \end{proof}

    \begin{proposition} \label{prop:prop6.5}
        $\mu$ is a measure on $[0, \infty)$ satisfying $\int_{[0, \infty)} x^n \mu(\d x) < M^n$ for every $n \in \N$ and a positive constant $M$, then
        \begin{align}
            \int_{[0, \infty)} \bracketsig{\lvol \word{1} x \shuexp{-x \word{1}}} \mu(\d x) &
            = - \int_0^t J'(t-s) \bar{Y}_s \d s \label{eq:prop_5.7_eq1}
        \intertext{and}
            \int_{[0, \infty)} \bracketsig{\lvol \word{2} x \shuexp{-x \word{1}}} \mu(\d x) &
            = - \int_0^t J'(t-s) \bar{Y}_s \d W_s
            - \half K_2(0) \int_0^t J'(t-s) (a_2 + b_2 \bar{Y}_s) \d s \label{eq:prop_5.7_eq2}
        \end{align}
        where $J(t) = \int_{[0, \infty)]} e^{-xt} \mu(\d x)$.
    \end{proposition}
    
    \begin{proof}
        By Lemma~\ref{lem:exponential-identity} and the shuffle property, we derive
        \begin{align}\label{eq:fubini_1}
            \bracketsig{\lvol \word{1} x \shuexp{-x \word{1}}} 
            &
            = x \bracketsig{\shuexp{-x \word{1}}} \int_0^t \bracketsig[s]{\shuexp{x \word{1}} \shuprod \lvol} \d s
            = x \int_0^t e^{-x(t-s)} \bar{Y}_s \d s,
        \end{align}
        
        Similarly, using Lemma~\ref{lem:exponential-identity} and the shuffle property,
        \begin{align}\label{eq:fubini_2}
            \bracketsig{\lvol \word{2} x \shuexp{-x \word{1}}} &
            = x \bracketsig{\shuexp{-x \word{1}}} \bracketsig{\bra{\shuexp{x \word{1}} \shuprod \lvol} \word{2}}
            \\ &
            = x \bracketsig{\shuexp{-x \word{1}}} \int_0^t \bracketsig[s]{\shuexp{x \word{1}} \shuprod \lvol} \d W_s
            \\ &
            \quad + \half x \bracketsig{\shuexp{-x \word{1}}} \int_0^t \bracketsig[s]{\shuexp{x \word{1}} \shuprod \lvol \proj{2}} \d s
            \\ &
            = x \int_0^t e^{-x(t-s)} \bar{Y}_s \d W_s + \half K_2(0) x \int_0^t e^{-x(t-s)} (a_2 + b_2 \bar{Y}_t) \d s,
        \end{align}
        where the second and third equalities come from Theorem~\ref{thm:sig_semimartingale} and \eqref{eq:lvol_proj_2} respectively.
        
        We will now show that we can apply the deterministic and stochastic Fubini theorem. First, noticing that for all $t \in [0, T]$,
        $$ \E \abs{\bar{Y}_t} = \E \abs{\bracketsig{\lvol}} \leq \normh{\lvol}, $$
        and thus
        $$ \sup_{t \in [0, T]} \E \abs{\bar{Y}_t} < \infty. $$
        
        We can now apply Tonelli's theorem:
        \begin{align*}
            \E \sqbra{\int_{[0, \infty)} \int_0^t e^{-x(t-s)} \abs{\bar{Y}_s} \d s ~\mu(\d x)} &
            = \int_{[0, \infty)} \int_0^t e^{-x(t-s)} \E \abs{\bar{Y}_s} \d s ~\mu(\d x)
            \\ &
            \leq \int_{[0, \infty)} \int_0^t \E \abs{\bar{Y}_s} \d s ~\mu(\d x)
            \\ &
            \leq \mu([0, \infty)) \int_0^t \E \abs{\bar{Y}_s} \d s
            \\ &
            \leq \mu([0, \infty)) ~ T \sup_{t \in [0, T]} \E \abs{\bar{Y}_t}
            < \infty.
        \end{align*}
        
        This leads to
        \begin{equation} \label{eq:fubini_condition1}
            \int_{[0, \infty)} \int_0^t e^{-x(t-s)} \abs{\bar{Y}_s} \d s ~\mu(\d x) < \infty \textnormal{ a.s.}.
        \end{equation} 

        Now we can apply Fubini theorem on equation \eqref{eq:fubini_1} and derive:
        $$ \int_{[0, \infty)} \bracketsig{\lvol \word{1} \shuexp{-x \word{1}}} \mu(\d x) = \int_0^t \bra{\int_{[0, \infty)} x e^{-x(t-s)} \mu(\d x)} \bar{Y}_s \d s = -\int_0^t J'(t-s) \bar{Y}_s \d s, $$
        thus proving \eqref{eq:prop_5.7_eq1}. On the other hand, by Proposition~\ref{prop:Ah}~\ref{lem:to_verify_stochastic_fubini}, 
        $$ \sup_{t \in [0, T]} \E \sqbra{\bar{Y}_t^2} \leq \bra{\normh[T]{\lvol}}^2 < \infty $$
        and
        \begin{align*}
            \E \sqbra{\int_{[0, \infty)} \sqrt{\int_0^t e^{-2x(t-s)} \abs{\bar{Y}_s}^2 \d s} ~\mu(\d x)} &
            \leq \mu([0, \infty)) \E \sqbra{\sqrt{\int_0^t \abs{\bar{Y}_s}^2 \d s}}
            \\ &
            \leq \mu([0, \infty)) \sqrt{\E \sqbra{\int_0^t \abs{\bar{Y}_s}^2 \d s}}
            \\ &
            \leq \mu([0, \infty)) \sqrt{T \sup_{t \in [0, T]} \E \abs{\bar{Y}_t}^2}
            < \infty,
        \end{align*}
        
        which leads to
        \begin{equation} \label{eq:fubini_condition2}
            \int_{[0, \infty)} \sqrt{\int_0^t e^{-2x(t-s)} \abs{\bar{Y}_s}^2 \d s} ~\mu(\d x) < \infty \textnormal{ a.s.}
        \end{equation}
        
        Note that \eqref{eq:fubini_condition2} is actually the condition for stochastic Fubini theorem to hold, see \cite[Theorem 2.2]{veraar2012stochastic}. Applying it, we derive
        \begin{align*}
            \int_{[0,\infty)} \bracketsig{\lvol \word{2} x \shuexp{-x \word{1}}} \mu(\d x) &
            = -\int_0^t J'(t-s) \bar{Y}_s \d W_s - \half K_2(0) \int_0^t J'(t-s) (a_2 + b_2 \bar{Y}_s) \d s,
        \end{align*}
        thus completing the proof.
    \end{proof}

    Now, taking $\mu = \mu_1, \mu_2$, we combine Proposition~\ref{prop:combined_commute} with Proposition~\ref{prop:prop6.5} and derive, for $i = 1, 2$:
    \begin{align} \label{eq:for_the_proof_of_lemma_5.2}
        \bracketsig{\lvol \word{1} \int_{[0, \infty)} x \shuexp{-x \word{1}} \mu_i(\d x)} & = -\int_0^t K_i'(t-s) \bar{Y}_s \d s, \\
        \label{eq:for_the_proof_of_lemma_5.2_2}\bracketsig{\lvol \word{2} \int_{[0, \infty)} x \shuexp{-x \word{1}} \mu_i(\d x)} &= -\int_0^t K_i'(t-s) \bar{Y}_s \d W_s - \half K_2(0) \int_0^t K_i'(t-s) (a_2 + b_2 \bar{Y}_s) \d s. \quad 
    \end{align}
    
    Finally, we can now compute the projections $\lvol \proj{1}, \lvol \proj{2}, \lvol \proj{22}$. First recall
    \begin{align}
        \pvol &
        = a_1 \word{1} \int_{[0, \infty)} \shuexp{-x \word{1}} \mu_1(\d x) + a_2 \bra{\word{2} - \half K_2(0) b_2 \word{1}} \int_{[0, \infty)} \shuexp{-x \word{1}} \mu_2(\d x) + y \emptyword,
        \\ \qvol &
        = b_1 \word{1} \int_{[0, \infty)} \shuexp{-x \word{1}} \mu_1({\d x}) + b_2 \bra{\word{2} - \half K_2(0) b_2 \word{1}} \int_{[0, \infty)} \shuexp{-x \word{1}} \mu_2(\d x).
    \end{align}
    We then have 
    \begin{align} \label{eq:projpqvol}
        \begin{split}
            \pvol \proj{1} &
            = a_1 \int_{[0, \infty)} \shuexp{-x \word{1}} \mu_1(\d x) - a_2 \int_{[0, \infty)} \bra{x \word{2} + \half K_2(0) b_2 \emptyword} \shuexp{-x \word{1}} \mu_2(\d x),
            \\ \qvol \proj{1} &
            = b_1 \int_{[0, \infty)} \shuexp{-x \word{1}} \mu_1(\d x) - b_2 \int_{[0, \infty)} \bra{x \word{2} + \half K_2(0) b_2 \emptyword} \shuexp{-x \word{1}} \mu_2(\d x),
            \\ \pvol \proj{2} &=\int_{[0,\infty)} \mu_2(dx) a_2 \emptyword 
            = K_2(0) a_2 \emptyword,
            \qquad
            \pvol \proj{22} = 0,
            \\ \qvol \proj{2} &= \int_{[0,\infty)} \mu_2(dx) b_2 \emptyword 
            = K_2(0) b_2 \emptyword,
            \qquad
            \qvol \proj{22} = 0.
        \end{split}
    \end{align}
       
    Then, using Lemma~\ref{lem:proj-resolvent} and noticing $(\qvol)^\emptyword = 0$, it immediately follows that
    \begin{align}
        \lvol \proj{2} &
        = \pvol \proj{2} + \lvol \bra{\qvol \proj{2}} = K_2(0) (a_2 \emptyword + b_2 \lvol), \label{eq:lvol_proj_2}
        \\ \lvol \proj{22} &
        = K_2(0) b_2 \lvol \proj{2} = K_2^2(0) b_2 (a_2 \emptyword + b_2 \lvol), \label{eq:lvol_proj_22}
    \end{align}
    and that
    \begin{align}
        \lvol \proj{1} &
        = \pvol \proj{1} + \lvol \bra{\qvol \proj{1}}
        \\ &
        = \bra{a_1 \emptyword + b_1 \lvol} \int_{[0, \infty)} \shuexp{-x \word{1}} \mu_1(\d x)
        - \bra{a_2 \emptyword + b_2 \lvol} \int_{[0, \infty)} \bra{x \word{2} + \half K_2(0) b_2 \emptyword} \shuexp{-x \word{1}} \mu_2(\d x)
        \\ &
        = K_1(0) \bra{a_1 \emptyword + b_1 \lvol} - \half K_2^2(0) b_2 \bra{a_2 \emptyword + b_2 \lvol} \label{eq:lvol_proj_1}
        \\ &
        \quad + \bra{a_1 \emptyword + b_1 \lvol} \word{1} \int_{[0, \infty)} -x \shuexp{-x \word{1}} \mu_1(\d x)
        \\ &
        \quad + \bra{a_2 \emptyword + b_2 \lvol} \bra{\word{2} - \half K_2(0) b_2 \word{1}} \int_{[0, \infty)} -x \shuexp{-x \word{1}} \mu_2(\d x).
    \end{align}
    
    \begin{proof}[Proof of Lemma~\ref{lem:lvolbrackets}]
        Following directly from \eqref{eq:lvol_proj_2}-\eqref{eq:lvol_proj_22}, we get
        \begin{align}
            \bracketsig{\lvol \proj{2}} &
            = K_2(0) (a_2 + b_2 \bar{Y}_t)
            \quad \text{and} \quad \bracketsig{\lvol \proj{22}} 
            = K_2^2(0) b_2 (a_2 + b_2 \bar{Y}_t).
        \end{align}
        
        It now only remains to compute $\bracketsig{\lvol \proj{1}}$. For this, recall \eqref{eq:lvol_proj_1}
        \begin{align}
            \lvol \proj{1} &
            = K_1(0) \bra{a_1 \emptyword + b_1 \lvol} - \half K_2^2(0) b_2 \bra{a_2 \emptyword + b_2 \lvol}
            \\ &
            \quad + \bra{a_1 \emptyword + b_1 \lvol} \word{1} \int_{[0, \infty)} -x \shuexp{-x \word{1}} \mu_1(\d x)
            \\ &
            \quad + \bra{a_2 \emptyword + b_2 \lvol} \bra{\word{2} - \half K_2(0) b_2 \word{1}} \int_{[0, \infty)} -x \shuexp{-x \word{1}} \mu_2(\d x).
        \end{align}
        A direct application of equation \eqref{eq:for_the_proof_of_lemma_5.2} and \eqref{eq:for_the_proof_of_lemma_5.2_2} yields:
        \begin{align}
            \bracketsig{\lvol \proj{1}} &
            = K_1(0)\bra{a_1 + b_1 \bar{Y}_t} - \half K_2^2(0) b_2 (a_2 + b_2 \bar{Y}_t)
            \\ &
            \quad + \int_0^t K_1'(t-s) \bra{a_1 + b_1 \bar{Y}_s} \d s
            + \int_0^t K_2'(t-s) \bra{a_2 + b_2 \bar{Y}_s} \d W_s, 
        \end{align}
        which concludes the proof of Lemma~\ref{lem:lvolbrackets} and thus of Theorem~\ref{thm:volterra}.
    \end{proof}

\section{Proof of Theorem \ref{thm:delayed}} \label{app:delayed}
    {\revtwo
    It is straightforward to see that $\pde \in \Aexp$ as 
    $$ \pde = \bra{a_1 - \thalf b_2 a_2} \word{1} + a_2 \word{2} + z \emptyword. $$
    Using Proposition~\ref{prop:Aexp_closed}~\ref{prop:Aexp_closed_concat}-\ref{prop:Aexp_closed_linear} we can easily see $\qde \in \Aexp$ as
    $$ \qde = \bra{\bra{b_1 - \thalf b_2^2} \emptyword + \sum_{m=1}^{k_1} c_1^m \word{1} \shuexp{\alpha_1^m \word{1}} - \thalf b_2 \sum_{m=1}^{k_2} c_2^m \word{1} \shuexp{\alpha_2^m \word{1}}} \word{1} + \bra{b_2 \emptyword + \sum_{m=1}^{k_2} c_2^m \word{1} \shuexp{\alpha_2^m \word{1}}} \word{2}. $$
    Then applying Corollary~\ref{coro:pq_Aexp_implies_ell_Aexp} we have $\lde \in \Aexp$ and in particular $\lde \in \I(\widehat{W})$.
    }
    Now recalling Lemma~\ref{lem:proj-resolvent}, i.e. for all $\bp, \bq \in \eTA$ such that $\bq^\emptyword = 0$, 
    \begin{align}
        \bra{\bp \inverse{\emptyword - \bq}} \proj{i} &
        = \bp \proj{i} + \bp \inverse{\emptyword - \bq} (\bq \proj{i}),
    \end{align}

    it directly follows that    
    \begin{align} \label{eq:de_proj}
        \lde \proj{1} &
        = \lde \bra{\bra{b_1 - \thalf b_2^2} \emptyword + \sum_{m=1}^{k_1} c_1^m \word{1} \shuexp{\alpha_1^m \word{1}} - \thalf b_2 \sum_{m=1}^{k_2} c_2^m \word{1} \shuexp{\alpha_2^m \word{1}}} + \bra{a_1 - \thalf b_2 a_2} \emptyword,
        \\ \lde \proj{2} &
        = \lde \bra{b_2 \emptyword + \sum_{m=1}^{k_2} c_2^m \word{1} \shuexp{\alpha_2^m \word{1}}} + a_2 \emptyword,
        \\ \lde \proj{22} &
        = b_2 \lde \proj{2}.
    \intertext{Finally,}
        \lde \proj{1} + \thalf \lde \proj{22} &
        = \lde \bra{b_1 \emptyword + \sum_{m=1}^{k_1} c_1^m \word{1} \shuexp{\alpha_1^m \word{1}}} + a_1 \emptyword.
    \end{align}
    
    Now ready to apply Theorem~\ref{thm:sig_ito_formula} on $\lde$, we get
    \begin{align}
        \bracketsig{\lde} &
        = \bra{\lde}^\emptyword + \int_0^t \bracketsig[s]{\lde \proj{1} + \thalf \lde \proj{22}} \d s + \int_0^t \bracketsig[s]{\lde \proj{2}} \d W_s
        \\ &
        = z + \int_0^t \bracketsig[s]{\lde \bra{b_1 \emptyword + \sum_{m=1}^{k_1} c_1^m \word{1} \shuexp{\alpha_1^m \word{1}}} + a_1 \emptyword} \d s
        \\ &
        \qquad + \int_0^t \bracketsig[s]{\lde \bra{b_2 \emptyword + \sum_{m=1}^{k_2} c_2^m \word{1} \shuexp{\alpha_2^m \word{1}}} + a_2 \emptyword} \d W_s
        \\ &
        = z + \int_0^t \bra{a_1 + b_1 \bracketsig[s]{\lde} + \sum_{m=1}^{k_1} c_1^m \bracketsig[s]{\lde \word{1} \shuexp{\alpha_1^m \word{1}}}} \d s
        \\ &
        \qquad + \int_0^t \bra{a_2 + b_2 \bracketsig[s]{\lde} + \sum_{m=1}^{k_2} c_2^m \bracketsig[s]{\lde \word{1} \shuexp{\alpha_2^m \word{1}}}} \d W_s.
    \end{align}

    Finally, using Lemma~\ref{lem:exponential-identity}, i.e. $\bell \word{i} \shuexp{\gamma \word{j}} = \shuexp{\gamma \word{j}} \shuprod \bra{\bra{\shuexp{-\gamma \word{j}} \shuprod \bell} \word{i}}$ for all $\gamma \in \R$ and $\word{i}, \word{j} \in \alphabet$, and the fact that $\lde \word{1} \shuexp{\alpha_i^m \word{1}} \in \A$ for all $m$ and $i = 1, 2$, it is then straightforward to see that
    \begin{align}
        \bracketsig{\lde \word{1} \shuexp{\alpha_i^m \word{1}}} &
        = \int_0^t e^{\alpha_i^m (t-s)} \bracketsig[s]{\lde} \d s.
    \end{align}
    This concludes the proof of Theorem \ref{thm:delayed} that $\bracketsig{\lde}$ is solution to \eqref{eq:sde-DE-exp} and hence $Z_t=\bracketsig{\lde}$ by strong uniqueness.

\section{Proof of Theorem \ref{thm:gaussian}} \label{app:gaussian}
    We first illustrate that $\bracketsig{\lgv}$ with $\lgv$ defined in \eqref{eq:lgv} is a well-defined stochastic process.
    \begin{align}
        \E \bigg[\int_0^T \sum_{n=0}^\infty \abs{K^{(n)}(t) \bracketsig{\word{1} \conpow{n} \word{2}}} \d t \bigg] &=  \int_0^T \sum_{n=0}^\infty \abs{K^{(n)}(t)} \E \abs{\bracketsig{\word{1} \conpow{n} \word{2}}} \d t \\
        &= \int_0^T\sum_{n=0}^\infty \abs{K^{(n)}(t)} \E \abs{\int_0^t \frac{s^n}{n!} dW_s} \d t
        \\ &
        \leq \int_0^T \sum_{n=0}^\infty \abs{K^{(n)}(t)} \frac{t^{n + \half}}{n! \sqrt{2n + 1}} ~\d t
    \end{align}

    The last inequality is derived from Jensen inequality. Then by assumption we know that 
    $$ \E \bigg[\int_0^T \sum_{n=0}^\infty \abs{K^{(n)}(t) \bracketsig{\word{1} \conpow{n} \word{2}}} \d t \bigg] <\infty, $$
    thus
    $$ \sum_{n=0}^\infty \abs{K^{(n)}(t)\bracketsig{\word{1} \conpow{n} \word{2}}}<\infty, ~\P(\d \omega)\otimes \d t-a.e. $$
    
    Consequently,
    $\sum_{n=0}^\infty K^{(n)}(t)\bracketsig{\word{1} \conpow{n} \word{2}}$ converges $\P(\d \omega)\otimes \d t-a.e.$, ensuring $\bracketsig{\lgv}$ to be a $\P(\d \omega)\otimes \d t-a.e.$-well-defined map. Notice that for every $N\in\N$, $\sum_{n=0}^N K^{(n)}(t)\bracketsig{\word{1} \conpow{n} \word{2}}$ is progressively measurable, so $\bracketsig{\lgv}$ is a also progressively measurable. It remains to prove that for every fixed $t$,
    \begin{align} \label{eq:proofgvequality}
        \bracketsig{\lgv} = \int_0^t K(t-s) \d W_s.
    \end{align}

    For this we write 
    \begin{align*}
        \bracketsig{\lgv} &
        =\sum_{n=0}^\infty K^{(n)}(t) (-1)^n \bracketsig{\word{1} \conpow{n} \word{2}}
        = \sum_{n=0}^{\infty}K^{(n)}(t) \int_{0}^t \frac{(-s)^n}{n!} \d W_s.
    \end{align*}
    
    Since $\int_0^t \left( \sum_{n=0}^\infty \abs{K^{(n)}(t)} \frac{s^n}{n!} \right)^2 \d s < \infty$ by assumption \eqref{eq:assKgaussian2}, the dominated convergence theorem of stochastic integrals yields
    
    $$ \sum_{n=0}^N K^{(n)}(t) \int_0^t \frac{(-s)^n}{n!} \d W_s \longrightarrow \int_0^t \sum_{n=0}^\infty K^{(n)}(t) \frac{(-s)^n}{n!} \d W_s = \int_0^t K(t-s) \d W_s, $$
    
    in probability, as $N \to \infty$. While the previous analysis ensures that
    
    $$ \sum_{n=0}^N K^{(n)}(t) \int_0^t \frac{(-s)^n}{n!} \d W_s \overset{\text{a.s.}}{\longrightarrow} \sum_{n=0}^\infty K^{(n)}(t) \int_0^t \frac{(-s)^n}{n!} \d W_s = \bracketsig{\lgv}, $$
    which proves \eqref{eq:proofgvequality} and completes the proof.

\appendix
\section{Properties of the resolvent} \label{sec:pf_control_of_l}

    We start by deriving a transformation formula, which is simple but crucial to simplify the algebraic expressions.

    \begin{lemma} \label{lem:exponential-identity}
        Let $\bell \in \eTA$, $\word{i} \in \alphabet$ and $\bm{b} := \sum_{\word{j} \in \alphabet} \bm{b}^\word{j} \word{j}$ for $\bm{b}^\word{j} \in \R$, then
        $$ \bell \word{i} \shuexp{\bm{b}} = \shuexp{\bm{b}} \shuprod \left( (\shuexp{-\bm{b}} \shuprod \bell) \word{i} \right). $$
    \end{lemma}
     
    \begin{proof}
        Let $\bgamma := \bell \word{i} \shuexp{\bm{b}} - \shuexp{\bm{b}} \shuprod \bra{\bra{\shuexp{-\bm{b}} \shuprod \bell} \word{i}}$, then
        \begin{align*}
            \bgamma &
            = \bell \word{i} \bra{\emptyword + \shuexp{\bm{b}} \bm{b}} - \bra{\emptyword + \shuexp{\bm{b}} \bm{b}} \shuprod \bra{\bra{\shuexp{-\bm{b}} \shuprod \bell} \word{i}}
            \\ &
            = \bell \word{i} + \bell \word{i} \shuexp{\bm{b}} \bm{b}
            - \sqbra{\emptyword \shuprod \bra{\shuexp{-\bm{b}} \shuprod \bell}} \word{i}
            - \sqbra{(\shuexp{\bm{b}} \bm{b}) \shuprod \bra{\shuexp{-\bm{b}} \shuprod \bell}} \word{i}
            - \sqbra{\shuexp{\bm{b}} \shuprod {\bra{\shuexp{-\bm{b}} \shuprod \bell} \word{i}}} \bm{b}
            \\ &
            = \sqbra{\bell - \bra{\shuexp{\bm{b}} \shuprod \shuexp{-\bm{b}} \shuprod \bell}} \word{i}
            + \sqbra{\bell \word{i} \shuexp{\bm{b}} - \shuexp{\bm{b}} \shuprod \bra{\bra{\shuexp{-\bm{b}} \shuprod \bell} \word{i}}} \bm{b}
            \\ &
            = 0 + \bgamma \bm{b},
        \end{align*}
        implying that $\bgamma = 0$.
    \end{proof}

    We then provide the useful decomposition of the unique solution of the linear equation $\bell = \bp + \bell \bq$.

    \begin{lemma} \label{lem:proj-resolvent}
        Let $\bp, \bq \in \eTA$ such that $\bq^\emptyword \neq 1$ and $\word{i} \in \alphabet$. Then
        $$ \bra{\bp \res{\bq}} \proj{i} = \frac{1}{1 - \bq^\emptyword} \sqbra{\bp \proj{i} + \bp \res{\bq} (\bq \proj{i})}. $$
    \end{lemma}
    
    \begin{proof}
        First recall the decomposition in Remark~\ref{rem:proj-decomposition}, i.e.
        $$ \bq = \bq^\emptyword \emptyword + \sum_{\word{i} \in \alphabet} \bq \proj{i} \word{i}, $$
        and Proposition~\ref{prop:resolvent-solution}, i.e.
        $$ \res{\bq} = \emptyword + \res{\bq} \bq. $$
        One can then easily get that
        \begin{align}
            \res{\bq} &
            = \emptyword + \res{\bq} \bra{\bq^\emptyword \emptyword + \sum_{\word{i} \in \alphabet} \bq \proj{i} \word{i}}
            = \frac{1}{1 - \bq^\emptyword} \bra{\emptyword + \res{\bq} \sum_{\word{i} \in \alphabet} \bq \proj{i} \word{i}}.
        \end{align}

        It then follows that 
        \begin{align}
            \bp \res{\bq} &
            = \frac{1}{1 - \bq^\emptyword} \bra{\bp + \bp \res{\bq} \sum_{\word{i} \in \alphabet} \bq \proj{i} \word{i}}.
        \end{align}
        
        Taking the projection ends the proof.
    \end{proof}

\section{Proof of Proposition~\ref{prop:uniform_bound_L2}} \label{app:gen_results}

    In this section, we are specifically interested in the case $X_t = \widehat{W}_t = (t, W_t^2, \dots, W_t^d)$ where $W = (W^2, \dots, W^d)$ is a ($d-1$)-dimensional Brownian motion. We will denote $n(\wv)$ the size of $\wv$, i.e. its number of letters, and $x(\wv)$ the number of $\word{1}$s specifically. If it does not cause ambiguity, we will write $n$ for $n(\wv)$ and $x$ for $x(\wv)$. \\
    
    Let $g^p: [0, T] \times V \to \R$ be defined for all even $p \in \N^*$ by
    \begin{align}
        g_t^p(\wv) = \sqbra{\prod_{j=2}^p \binom{jn}{n}} \frac{t^{\frac{p}{2} (n+x)}}{\bra{\frac{p}{2}(n+x)}!} 2^{\frac{p}{2} (x-n)}
    \end{align}
    and for all odd $p \in \N$ by 
    \begin{align}
        g_t^p(\wv) = \left( g_t^{p+1}(\wv) \right)^{\frac{p}{p+1}},
    \end{align}
    for all $t \in [0, T]$ and all $\wv \in V$. Notice that $g_t^p(\wv)$ only depends on $\wv$ through its size and the number of its $\word{1}$s, i.e. $n(\wv)$ and $x(\wv)$.

    \begin{proposition} \label{prop:uniform_boundary}
        Fix  $t \in [0, T]$. We have
        \begin{align}
            \E \sqbra{\abs{\bracketsig{\wv}}^p} \leq g_t^p(\wv), \quad \wv \in V, \quad p \in \N.
        \end{align}
    \end{proposition}

    \begin{proof}
        We first use Fawcett's formula \cite[Proposition 4.10]{CubatureWienerSpace} to write 
        $$ \E \sqbra{\sig} = \sum_{m=0}^\infty \frac{t^m}{m!} \bra{\word{1} + \half \sum_{\word{j} \in \set{\word{2}, \dots, \word{d}}} \word{jj}} \conpow{m}. $$
        
        We then define for all $n \in \N$
        \begin{equation} \label{eq:sig_basis}
            \widehat{V}_n := \set{\wv \in V_n : \exists j \in \N \text{ such that } \wv = \word{u_1 \cdots u_j} \text{ for some } \word{u_k} \in \set{\word{1}, \word{22}, \dots, \word{dd}} \text{ with } k = 1, 2, \dots, j}.
        \end{equation}
        
        This set corresponds to the coordinate of every signature element with non-zero expectation, i.e.~for every $n \in \N$ and $\wv \in V_n$, 
        \begin{equation}
            \E \sqbra{\bracketsig{\wv}}
            = \begin{cases}
                \frac{t^{\frac{1}{2}(n+x)}}{\left( \frac{1}{2}(n+x) \right)!} \cdot 2^{\frac{1}{2}(x-n)} &
                \text{if } \wv \in \widehat{V}_n,
                \\ 0 &
                \text{else}.
            \end{cases}
        \end{equation}
        
        We are now ready to compute even moments of every signature element. First, remark that for all $\wu \in V_k$ and $\wv \in V_m$, $\wu \shuprod \wv$ is a sum of $\binom{k+m}{k}$ elements $\ww \in V_{k+m}$ of length $n(\ww) = k+m$ each, having exactly $x(\ww) = x(\wu)+x(\wv)$ number of $\word{1}$s, so that 
        $$ \E \sqbra{\bracketsig{\word{u} \shuprod \word{v}}} \leq \binom{k+m}{k} \max_{\substack{\word{w} \in V_{k+m} \\ \textnormal{s.t. } x(\word{w}) = x(\word{u}) + x(\word{v})}} \E \bracketsig{\word{w}}. $$
        
        Consequently, for every $p \in \N$, we can iteratively apply the shuffle product $p$-times, and obtain the claimed upper bound for the even powers of the signature:
        \begin{align}
            \E \left[ \bracketsig{\wv}^{2p} \right] &
            = \E \bracketsig{\wv \shupow{2p}}
            \leq \left[ \prod_{j=2}^{2p} \binom{jn}{n} \right] \max_{\substack{\ww \in V_{2p n} \\ s.t. \: x(\ww) = 2p \: x(\wv)}} \E \bracketsig{\ww}
            = \left[ \prod_{j=2}^{2p} \binom{jn}{n} \right] \frac{t^{p(n+x)}}{(p(n+x))!} 2^{p(x-n)}.
        \end{align}
        
        We can finally treat the odd moments by applying Jensen's inequality on the convex function $y \to y^{\frac{2p}{2p-1}}$ on $\R_+$:
        \begin{align}
            \E \sqbra{\abs{\bracketsig{\wv}}^{2p - 1}}^{\frac{1}{2p - 1}} \leq \E \sqbra{\bracketsig{\wv}^{2p}}^{\frac{1}{2p}} \leq \bra{h_t^{2p}(\wv)}^{\frac{1}{2p}} = \bra{h_t^{2p-1}(\wv)}^{\frac{1}{2p-1}},
        \end{align}
        concluding the proof of the proposition.
    \end{proof}

    We are now ready to prove Proposition~\ref{prop:uniform_bound_L2}:

    \begin{proof}[Proof of Proposition~\ref{prop:uniform_bound_L2}]
        Proposition~\ref{prop:uniform_bound_L2}-\ref{prop:uniform_bound_L2:item1} follows directly from an application of the above Proposition~\ref{prop:uniform_boundary}. In order to prove \eqref{eq:uniform_bound_L2_sup} we will consider two cases: either when word $\wv$ ends with a $\word{1}$ or it ends with a $\word{j}$ in $\set{\word{2}, \dots, \word{d}}$.
    
        \begin{enumerate}[label=\textbf{(\arabic*)}]
            \item
            \textbf{Case $\wv = \word{u1}$ for some word $\wu \in V$.}
            We apply Cauchy-Schwarz inequality to get 
            \begin{align}
                \E \sqbra{\sup_{s \in [0, t]} \bracketsig[s]{\wv}^2} &
                = \E \sqbra{\sup_{s \in [0, t]} \bra{\int_0^s \bracketsig[u]{\wu} \d u}^2}
                \\ &
                \leq \E \sqbra{\bra{\int_0^t \abs{\bracketsig[u]{\wu}} \d u}^2}
                \\ &
                \leq \E \sqbra{t \int_0^t \bracketsig[u]{\wu}^2 \d u}
                \\ &
                = t \int_0^t \E \sqbra{\bracketsig[u]{\wu}^2} \d u.
            \end{align}
            
            Thus, using \eqref{eq:uniform_bound_L2} with $n = n(\wv) = n(\wu) + 1$ and $x = x(\wv) = x(\wu) + 1$, we have
            \begin{align} \label{eq:bound_part_1}
                \E \sqbra{\sup_{s \in [0, t]} \bracketsig[s]{\word{v}}^2} &
                \leq \binom{2(n-1)}{n-1} \frac{t}{2^{n-x} (n+x-2)!} \int_0^t u^{n+x-2} \d u
                \\ &
                = \binom{2(n-1)}{n-1} \frac{t^{n+x}}{2^{n-x} (n+x-1)!}
                \\ &
                \leq \frac{2^{n+x}}{\sqrt{n}} \frac{t^{n+x}}{(n+x-1)!}, \label{eq:apn:bound_for_Esup_u1}
            \end{align}
            where the last inequality comes from an application of Stirling's approximation that reads $\binom{2m}{m} \leq \frac{2^{2m}}{\sqrt{m+1}}$ for all $m \in \N$.
            
            \item
            \textbf{Case $\wv = \word{uj}$ for some word $\wu \in V$ and some letter $\word{j} \in \set{\word{2}, \dots, \word{d}}$.} First we apply Theorem~\ref{thm:sig_ito_formula} to get
            \begin{align}
                \bracketsig[s]{\wv}^2 &
                = \bra{\int_0^s \bracketsig[u]{\wu} \d W_u^\word{j} + \half \int_0^s \bracketsig[u]{\wu \proj{j}} \d u}^2
                \\ &
                \leq 2 \bra{\int_0^s \bracketsig[u]{\wu} \d W_u^\word{j}}^2 + \half \bra{\int_0^s \bracketsig[u]{\wu \proj{j}} \d u}^2. \label{eq:apn:item:v=u2:1st_decompo}
            \end{align}
            
            Now, since $\int_0^s \bracketsig[u]{\wu} dW_u^\word{j}$ is a true martingale, we apply Doob's inequality followed by the previous result \eqref{eq:apn:bound_for_Esup_u1} with $n = n(\wv) = n(\wu) + 1$ and $x = x(\wv) = x(\wu)$ and get
            \begin{align}
                \E \sqbra{\sup_{s \in [0, t]} \bra{\int_0^s \bracketsig[u]{\wu} \d W_u^\word{j}}^2} 
                \leq 4~\E \sqbra{\int_0^t \bracketsig[u]{\wu} \d W_u^\word{j}}^2
                = 4 \int_0^t \E \sqbra{\bracketsig[u]{\wu}^2} \d u
                \leq 4 \frac{2^{n+x}}{\sqrt{n}} \frac{t^{n+x}}{(n+x)!}.
            \end{align}
            Then, again applying \eqref{eq:apn:bound_for_Esup_u1} to the right-most element in \eqref{eq:apn:item:v=u2:1st_decompo}, we get
            \begin{align}
                \E \sqbra{\sup_{s \in [0, t]} \bra{\int_0^s \bracketsig[u]{\wu \proj{j}} \d u}^2} 
                = \E \sqbra{\sup_{s \in [0, t]} \bracketsig[s]{\wu \proj{j} \word{1}}^2}
                \leq \frac{2^{n+x}}{\sqrt{n}} \frac{t^{n+x}}{(n+x-1)!}.
            \end{align}
            Note that if $\wu \proj{j} = 0$, the left-hand side is 0 and the equation still holds. \\
            
            Combining the above inequalities, we derive
            \begin{align}
                \E \sqbra{\sup_{s \in [0, t]} \bracketsig[s]{\wv}^2} &
                \leq C_1 \frac{2^{n+x-1}}{\sqrt{n}} \frac{t^{n+x}}{(n+x-1)!}, \label{eq:apn:bound_for_Esup_u2}
            \end{align}
            for some $C_1 \in \R$ independent of $t$ and $\wv$.
        \end{enumerate}
        
        Combining the two equalities \eqref{eq:apn:bound_for_Esup_u1}-\eqref{eq:apn:bound_for_Esup_u2} above, we have that for every $\wv \in V_n$, $n \geq 1$, 
        \begin{equation}
            \E \sqbra{\sup_{s \in [0, t]} \abs{\bracketsig[s]{\wv}}^2}
            \leq C_2 \frac{2^{n+x-1}}{\sqrt{n}} \frac{t^{n+x}}{(n+x-1)!}.
        \end{equation}
        for some $C_2 \in \R$ independent of $t$ and $\wv$. \\
        
        Moreover, as $n+1 \sim n$, there exists a constant $C$ such that
        \begin{equation} \label{eq:Appendix_B_last}
            \E \sqbra{\sup_{s \in [0, t]} \abs{\bracketsig[s]{\wv}}^2}
            \leq C \frac{2^{n+x-1}}{\sqrt{n+1}} \frac{t^{n+x}}{(n+x-1)!}.
        \end{equation}
        This includes the case $\wv = \emptyword \in V_0$ and proves \eqref{eq:uniform_bound_L2_sup}. 
        
        It remains to prove the sub-multiplicative property \eqref{eq:prop:h-uv}.
        First, in the case $\wu = \emptyword$ or $\ww = \emptyword$, \eqref{eq:prop:h-uv} holds trivially. Now we can easily see that
        \begin{align} \label{eq:proof_new_h}
            \bra{\frac{h_t(\wu) h_t(\wv)}{h_t(\wu \wv)}}^2 &
            = C^2 A(\wu, \wv) B(\wu, \wv),
        \end{align}
        where 
        \begin{align}
            A(\wu, \wv) :
            = \frac{(n(\wu) + x(\wu) + n(\wv) + x(\wv) - 1)!}{(n(\wu) + x(\wu) - 1)! (n(\wv) + x(\wv) - 1)!} \quad 
        \text{and} \quad 
            B(\wu, \wv) :
            = \sqrt{\frac{n(\wu) + n(\wv) + 1}{(n(\wu) + 1) (n(\wv) + 1)}}.
        \end{align}
        
        We now only need to show that $C^2 A(\wu, \wv) B(\wu, \wv) \geq 1$ for $\wu, \wv \neq \emptyword$. Without loss of generality, we assume $\wu \neq \emptyword$, i.e. $n(\wu) > 0$, otherwise we could simply swap $\wu$ and $\wv$ in the following equations. Let us then remark that        
        \begin{align}
            A(\wu, \wv) &
            = \binom{n(\wu) + x(\wu) + n(\wv) + x(\wv) - 1}{n(\wu) + x(\wu)} (n(\wu) + x(\wu))
        \end{align}
        and that 
        \begin{align}
            \sqrt{n(\wu) + 1} &
            \leq C^2 (n(\wu) + x(\wu))
            \\ \sqrt{n(\wv) + 1} \sqrt{n(\wu) + 1} &
            \leq C^2 \sqrt{n(\wv) + 1} (n(\wu) + x(\wu))
            \leq C^2 \sqrt{n(\wu) + n(\wv) + 1} (n(\wu) + x(\wu)),
        \end{align}
        which implies
        \begin{align}
            B(\wu, \wv) \geq \frac{1}{C^2 (n(\wu) + x(\wu))}
        \end{align}
        and thus
        \begin{align}
            \bra{\frac{h_t(\wu) h_t(\wv)}{h_t(\wu \wv)}}^2 &
            \geq \binom{n(\wu) + x(\wu)) + n(\wv) + x(\wv) - 1}{n(\wu) + x(\wu)} \geq 1.
        \end{align}
    \end{proof}

\section{Proof of Proposition~\ref{prop:shuffle_property}} \label{A:proofshuffle}

    Let $n \leq k \in \N$, then for all $\wu \in V_n$, $\wv \in V_{k-n}$, each element of $\wu \shuprod \wv$ is of size $k$. This leads to the fact that $\bracket{\wu \shuprod \wv}{\ww} = 0$ whenever $\ww \notin V_k$, i.e.
    \begin{align} \label{eq:first_observation}
        \bracketsigX{\wu \shuprod \wv} = \sum_{\ww \in V_k} \bracket{\wu \shuprod \wv}{\ww} \sigX^\ww.
    \end{align}

    Moreover, by bilinearity of the shuffle product and for $\bell, \bphi \in \eTA$,
    \begin{align}
        \langle \bell \shuprod \bphi, \ww \rangle &
        = \sum_{m=0}^\infty \sum_{n=0}^\infty \sum_{\wu \in V_n} \sum_{\wv \in V_m} \bell^\wu \bphi^\wv \langle \wu \shuprod \wv, \ww \rangle
       = \sum_{n=0}^k \sum_{\wu \in V_n} \sum_{\wv \in V_{k-n}} \bell^\wu \bphi^\wv \langle \wu \shuprod \wv, \ww \rangle.
    \end{align}
    
    Hence, it is easy to see that for all $k \in \N$
    \begin{align}
        \sum_{\ww \in V_k} \bracket{\bell \shuprod \bphi}{\ww} \sigX^\ww = \sum_{n=0}^k \sum_{\wu \in V_n} \sum_{\wv \in V_{k-n}} \bell^\wu \bphi^\wv \bracketsigX{\wu \shuprod \wv}.
    \end{align}

    It then follows that the semi-norm of $\bell \shuprod \bphi$ is of the form
    \begin{align*}
        \normA{\bell \shuprod \bphi} &
        = \sum_{k=0}^\infty \abs{\sum_{\ww \in V_k} \bracket{\bell \shuprod \bphi}{\ww} \sigX^\ww}
        \\ &
        = \sum_{k=0}^\infty \abs{\sum_{n=0}^k \sum_{\wu \in V_n} \sum_{\wv \in V_{k-n}} \bell^\wu \bphi^\wv \bracketsigX{\wu \shuprod \wv}}
        \\ &
        \leq \sum_{k=0}^\infty \sum_{n=0}^k \abs{\sum_{\wu \in V_n} \sum_{\wv \in V_{k-n}} \bell^\wu \bphi^\wv \bracketsigX{\wu \shuprod \wv}}.
    \end{align*}
    
    Finally, using the shuffle product property on the right-hand side, recall Lemma~\ref{lem:shuffleproperty}, the semi-norm can be bounded
    \begin{align*}
        \normA{\bell \shuprod \bphi} &
        \leq \sum_{k=0}^\infty \sum_{n=0}^k \abs{\sum_{\wu \in V_n} \sum_{\wv \in V_{k-n}} \bell^\wu \bphi^\wv \sigX^\wu \sigX^\wv}
        \leq \sum_{m=0}^\infty \sum_{n=0}^\infty \abs{\sum_{\wu \in V_n} \bell^\wu \sigX^\wu \sum_{\wv \in V_m} \bphi^\wv \sigX^\wv}
        \leq \normA{\bell} \cdot \normA{\bphi}
        < \infty.
    \end{align*}
    
    It can then be argued very similarly that $\bracketsigX{\bell} \bracketsigX{\bphi} = \bracketsigX{\bell \shuprod \bphi}$, ending the proof.

\bibliographystyle{plainnat}
\bibliography{lin_sig.bib}

\end{document}